\documentclass[reqno]{amsart}
\usepackage{amssymb}
\usepackage{graphicx}

\usepackage[usenames, dvipsnames]{color}
\usepackage{verbatim}
\usepackage{mathrsfs}
\usepackage{bm}
\usepackage{cite}

\numberwithin{equation}{section}

\newtheorem{theorem}{Theorem}[section]
\newtheorem{corollary}[theorem]{Corollary}
\newtheorem{lemma}[theorem]{Lemma}
\newtheorem{prop}[theorem]{Proposition}

\theoremstyle{definition}
\newtheorem{remark}[theorem]{Remark}

\theoremstyle{definition}

\theoremstyle{definition}

\makeatletter
\def\dashint{\operatorname%
{\,\,\text{\bf-}\kern-.98em\DOTSI\intop\ilimits@\!\!}}
\makeatother

\def\\det{\text{\det}}

\def\.5{\frac{1}{2}}

\newcommand{\RN}[1]{%
  \textup{\uppercase\expandafter{\romannumeral#1}}%
}

\renewcommand{\epsilon}{\varepsilon}

\newcounter{marnote}

\begin{document}

\title[Stress concentration factors]{Stress concentration factors for the Lam\'{e} system arising from composites}

\author[C.X. Miao]{Changxing Miao}
\address[C.X. Miao] {1. Beijing Computational Science Research Center, Beijing 100193, China.}
\address{2. Institute of Applied Physics and Computational Mathematics, P.O. Box 8009, Beijing, 100088, China.}
\email{miao\_changxing@iapcm.ac.cn}

\author[Z.W. Zhao]{Zhiwen Zhao}

\address[Z.W. Zhao]{Beijing Computational Science Research Center, Beijing 100193, China.}
%\address{2. School of Mathematical Sciences, Beijing Normal University, Beijing 100875, China.}
\email{zwzhao365@163.com}

%\footnote{}

\date{\today} % delete this line to display the current date

%%% BEGIN DOCUMENT

\maketitle
%\tableofcontents
\begin{abstract}
For two neighbouring stiff inclusions, the stress, which is the gradient of a solution to the Lam\'{e} system of linear elasticity, may exhibit singular behavior as the distance between these two inclusions becomes arbitrarily small. In this paper, a family of stress concentration factors, which determine whether the stress will blow up or not, are accurately constructed in the presence of the generalized $m$-convex inclusions in all dimensions. We then use these stress concentration factors to establish the optimal upper and lower bounds on the stress blow-up rates in any dimension and meanwhile give a precise asymptotic expression of the stress concentration for interfacial boundaries of inclusions with different principal curvatures in dimension three. Finally, the corresponding results for the perfect conductivity problem are also presented.

\end{abstract}

%\noindent{\bf{Keywords}}: Composites; stress concentration factors; gradient estimates; asymptotic formulas.

\section{Introduction}%\label{intro}

Let $D\subseteq\mathbb{R}^{d}\,(d\geq2)$ be a bounded domain with $C^{2,\alpha}\,(0<\alpha<1)$ boundary. It contains two $C^{2,\alpha}$-subdomains $D_{1}^{\ast}$ and $D_{2}$, which are far away from $\partial D$ and touch only at the origin with the inner normal vector of $\partial D_{1}^{\ast}$ pointing in the positive $x_{d}$-axis. By translating $D_{1}^{\ast}$ by a small positive constant $\varepsilon$ along $x_{d}$-axis, we obtain
\begin{align*}
D_{1}^{\varepsilon}:=D_{1}^{\ast}+(0',\varepsilon).
\end{align*}
Here and afterwards, we utilize superscript prime to represent ($d-1$)-dimensional domains and variables, such as $B'$ and $x'$. For the sake of simplicity, we drop superscripts and write
\begin{align*}
D_{1}:=D_{1}^{\varepsilon},\quad\Omega:=D\setminus\overline{D_{1}\cup D_{2}},\quad\mathrm{and}\quad\Omega^{\ast}:=D\setminus\overline{D_{1}^{\ast}\cup D_{2}}.
\end{align*}
Before giving the Lam\'{e} system of the elasticity problem arising from high-contrast composites, we first introduce some physical quantities. Let $\mathbb{C}^{0}:=(C_{ijkl}^{0})$ be the elasticity tensor whose elements are represented as
\begin{align}\label{DWQ}
C_{ijkl}^0=\lambda\delta_{ij}\delta_{kl} +\mu(\delta_{ik}\delta_{jl}+\delta_{il}\delta_{jk}),\quad i,j,k,l=1,2,...,d,
\end{align}
where the Lam\'{e} pair $(\lambda,\mu)$ verifies the following strong ellipticity condition
\begin{align}\label{JFAMC001}
\tau_{0}<\mu,\,d\lambda+2\mu<\frac{1}{\tau_{0}},\quad\text{for\;some\;positive\;constant\;}\tau_{0}\;\text{independent\;of\;}\varepsilon,
\end{align}
$\delta_{ij}$ represents the kronecker symbol: $\delta_{ij}=0$ for $i\neq j$, $\delta_{ij}=1$ for $i=j$. For later use, we list some properties of the tensor $\mathbb{C}^{0}$. First, the components $C_{ijkl}^{0}$ defined in \eqref{DWQ} satisfy the symmetry property as follows:
\begin{align}\label{symm}
C_{ijkl}^{0}=C_{klij}^{0}=C_{klji}^{0},\quad i,j,k,l=1,2,...,d.
\end{align}
Given any two $d\times d$ matrices $\mathbb{A}=(a_{ij})$ and $\mathbb{B}=(b_{ij})$, denote
\begin{align*}
(\mathbb{C}^{0}\mathbb{A})_{ij}=\sum_{k,l=1}^{d}C_{ijkl}^{0}a_{kl},\quad\hbox{and}\quad(\mathbb{A},\mathbb{B})\equiv \mathbb{A}:\mathbb{B}=\sum_{i,j=1}^{d}a_{ij}b_{ij}.
\end{align*}
Then we have
$$(\mathbb{C}^{0}\mathbb{A},\mathbb{B})=(\mathbb{A}, \mathbb{C}^{0}\mathbb{B}).$$
Using (\ref{symm}), we get that $\mathbb{C}^{0}$ satisfies the following ellipticity condition: for every $d\times d$ real symmetric matrix $\xi=(\xi_{ij})$,
\begin{align}\label{ellip}
\min\{2\mu, d\lambda+2\mu\}|\xi|^2\leq(\mathbb{C}^{0}\xi, \xi)\leq\max\{2\mu, d\lambda+2\mu\}|\xi|^2,\quad |\xi|^2=\sum\limits_{i,j=1}^{d}\xi_{ij}^2.
\end{align}

Denote the linear space of rigid displacement by
\begin{align}\label{LAK01}
\Psi:=\{\psi\in C^1(\mathbb{R}^{d}; \mathbb{R}^{d})\ |\ \nabla\psi+(\nabla\psi)^T=0\}
\end{align}
and a basis of $\Psi$ by
\begin{align}\label{OPP}
\left\{\,e_{i},\,x_{k}e_{j}-x_{j}e_{k}\,\big|\,1\leq i\leq d,\,1\leq j<k\leq d\,\right\},
\end{align}
where $\{e_{1},...,e_{d}\}$ is the standard basis of $\mathbb{R}^{d}$. For notational simplicity, rewrite this basis as $\big\{\psi_{\alpha}\big|\,\alpha=1,2,...,\frac{d(d+1)}{2}\big\}$. With regard to the order of every element $\psi_{\alpha}$, we make the rule as follows: $\psi_{\alpha}=e_{\alpha}$ if $\alpha=1,2,...,d$; $\psi_{\alpha}=x_{d}e_{\alpha-d}-x_{\alpha-d}e_{d}$ if $\alpha=d+1,...,2d-1$; if $\alpha=2d,...,\frac{d(d+1)}{2}\,(d\geq3)$, then there exist two indices $1\leq i_{\alpha}<j_{\alpha}<d$ such that
$\psi_{\alpha}=(0,...,0,x_{j_{\alpha}},0,...,0,-x_{i_{\alpha}},0,...,0)$.

Denote the elastic displacement field by a vector-valued function $u=(u^{1},u^{2},...,u^{d})^{T}:D\rightarrow\mathbb{R}^{d}$.  The mathematical model of a high-contrast composite material can be formulated as the following Lam\'{e} system:
\begin{align}\label{La.002}
\begin{cases}
\mathcal{L}_{\lambda, \mu}u:=\nabla\cdot(\mathbb{C}^0e(u))=0,\quad&\hbox{in}\ \Omega,\\
u|_{+}=u|_{-},&\hbox{on}\ \partial{D}_{i},\,i=1,2,\\
e(u)=0,&\hbox{in}~D_{i},\,i=1,2,\\
\int_{\partial{D}_{i}}\frac{\partial u}{\partial \nu_0}\big|_{+}\cdot\psi_{\alpha}=0,&i=1,2,\,\alpha=1,2,...,\frac{d(d+1)}{2},\\
u=\varphi,&\hbox{on}\ \partial{D},
\end{cases}
\end{align}
where $\varphi=(\varphi^{1},\varphi^{2},...,\varphi^{d})^{T}\in C^{2}(\partial D;\mathbb{R}^{d})$ is a given boundary data, $e(u)=\frac{1}{2}(\nabla u+(\nabla u)^{T})$ is the strain tensor, $\nabla u$ represents the stress, the co-normal derivative $\frac{\partial u}{\partial \nu_0}\big|_{+}$ is defined by
\begin{align*}
\frac{\partial u}{\partial \nu_0}\Big|_{+}&:=(\mathbb{C}^0e(u))\nu=\lambda(\nabla\cdot u)\nu+\mu(\nabla u+(\nabla u)^T)\nu,
\end{align*}
and $\nu$ denotes the unit outer normal of $\partial D_{i}$, $i=1,2$. Here and below, the subscript $\pm$ shows the limit from outside and inside the domain, respectively. We here remark that the existence, uniqueness and regularity of weak solutions to \eqref{La.002} have been established in previous work \cite{BLL2015} by using a variational argument. Moreover, the $H^{1}$ weak solution $u$ of problem \eqref{La.002} has been proved to be in $C^1(\overline{\Omega};\mathbb{R}^{d})\cap C^1(\overline{D}_{1}\cup\overline{D}_{2};\mathbb{R}^{d})$ for any $C^{2,\alpha}$-domains.

A straightforward calculation gives that
\begin{align*}
\mathcal{L}_{\lambda, \mu}u=\mu\Delta u+(\lambda+\mu)\nabla(\nabla\cdot u),
\end{align*}
and, for any open set $O$ and $u, v\in C^2(O;\mathbb{R}^{d})$,
\begin{align}\label{Le2.01222}
\int_O(\mathbb{C}^0e(u), e(v))\,dx=-\int_O\left(\mathcal{L}_{\lambda, \mu}u\right)\cdot v+\int_{\partial O}\frac{\partial u}{\partial \nu_0}\Big|_{+}\cdot v.
\end{align}

As shown in \cite{BLL2015,BLL2017}, it follows from the second and third lines of \eqref{La.002} that the solution $u$ of \eqref{La.002} can be decomposed as follows:
\begin{equation}\label{Decom}
u=\sum_{\alpha=1}^{\frac{d(d+1)}{2}}C_1^{\alpha}v_{1}^{\alpha}+\sum_{\alpha=1}^{\frac{d(d+1)}{2}}C_2^{\alpha}v_2^{\alpha}+v_{0},\qquad\mathrm{in}\;\Omega,
\end{equation}
where the free constants $C_{i}^{\alpha}$, $i=1,2,\,\alpha=1,2,...,\frac{d(d+1)}{2}$ can be determined by boundary condition in the fourth line of \eqref{La.002}, $v_{0}$ and $v_{i}^{\alpha}\in{C}^{2}(\Omega;\mathbb{R}^d)$, $i=1,2$, $\alpha=1,2,...,\frac{d(d+1)}{2}$, respectively, solve
\begin{equation}\label{qaz001}
\begin{cases}
\mathcal{L}_{\lambda,\mu}v_{0}=0,&\mathrm{in}~\Omega,\\
v_{0}=0,&\mathrm{on}~\partial{D}_{1}\cup\partial{D_{2}},\\
v_{0}=\varphi,&\mathrm{on}~\partial{D},
\end{cases}\quad
\begin{cases}
\mathcal{L}_{\lambda,\mu}v_{i}^{\alpha}=0,&\mathrm{in}~\Omega,\\
v_{i}^{\alpha}=\psi_{\alpha},&\mathrm{on}~\partial{D}_{i},~i=1,2,\\
v_{i}^{\alpha}=0,&\mathrm{on}~\partial{D_{j}}\cup\partial{D},~j\neq i.
\end{cases}
\end{equation}
The mathematical problem of interest is to study the blow-up feature of the stress $\nabla u$ in the thin gap between inclusions. Denote
\begin{align}\label{CTL001}
u_{b}:=\sum_{\alpha=1}^{\frac{d(d+1)}{2}}C_{2}^\alpha({v}_{1}^\alpha+{v}_{2}^\alpha)+v_{0}.
\end{align}
By linearity, $u_{b}$ actually satisfies
\begin{align}\label{GLQ}
\begin{cases}
\mathcal{L}_{\lambda,\mu}u_{b}=0,&\mathrm{in}\;\Omega,\\
u_{b}=\sum\limits^{\frac{d(d+1)}{2}}_{\alpha=1}C^{\alpha}_{2}\psi_{\alpha},&\mathrm{on}\;\partial D_{1}\cup\partial D_{2},\\
%\int_{\partial D_{1}}\frac{\partial u_{b}}{\partial\nu_{0}}\big|_{+}\cdot\psi_{\beta}=0,&\beta=1,2,...,\frac{d(d+1)}{2},\\
u_{b}=\varphi,&\mathrm{on}\;\partial D.
\end{cases}
\end{align}
According to \eqref{Decom} and \eqref{CTL001}, we split $\nabla u$ into two parts as follows:
\begin{align}\label{Decom002}
\nabla{u}=&\sum_{\alpha=1}^{\frac{d(d+1)}{2}}(C_{1}^\alpha-C_{2}^\alpha)\nabla{v}_{1}^\alpha+\nabla u_{b}.
\end{align}
We here would like to point out that $\sum_{\alpha=1}^{\frac{d(d+1)}{2}}(C_{1}^\alpha-C_{2}^\alpha)\nabla{v}_{1}^\alpha$ is the singular part and appears blow-up, while $\nabla u_{b}$ is the regular part and possesses the exponentially decaying property. For every $\alpha=1,2,...,\frac{d(d+1)}{2}$, the singular behavior of $\nabla v_{1}^{\alpha}$ can be made clear by constructing its explicit leading term. Then it remains to calculate the difference $C_{1}^{\alpha}-C_{2}^{\alpha}$ to capture the singularity of $\sum_{\alpha=1}^{\frac{d(d+1)}{2}}(C_{1}^\alpha-C_{2}^\alpha)\nabla{v}_{1}^\alpha$. In light of decomposition \eqref{Decom002}, it follows from the fourth line of \eqref{La.002} that
\begin{align}\label{JGRO001}
\sum\limits_{\alpha=1}^{\frac{d(d+1)}{2}}(C_{1}^\alpha-C_{2}^{\alpha}) a_{11}^{\alpha\beta}=\mathcal{B}_{\beta}[\varphi],\quad\beta=1,2,...,\frac{d(d+1)}{2},
\end{align}
where, for $\alpha,\beta=1,2,...,\frac{d(d+1)}{2}$,
\begin{align}\label{LGBC}
a_{11}^{\alpha\beta}:=-\int_{\partial{D}_{1}}\frac{\partial v_{1}^{\alpha}}{\partial \nu_0}\large\Big|_{+}\cdot\psi_{\beta},\quad \mathcal{B}_{\beta}[\varphi]:=\int_{\partial D_{1}}\frac{\partial u_{b}}{\partial\nu_{0}}\Big|_{+}\cdot\psi_{\beta}.
\end{align}
In the following, we will give a precise computation for each element of the coefficient matrix $(a_{11}^{\alpha\beta})_{\frac{d(d+1)}{2}\times\frac{d(d+1)}{2}}$ of \eqref{JGRO001}, whose values are explicit and depend not on the boundary data $\varphi$, see Lemma \ref{lemmabc} below. Each $\mathcal{B}_{\beta}[\varphi]$, by contrast, will vary with the boundary data $\varphi$ according to the definition of $u_{b}$ in \eqref{GLQ}. This fact implies that the values of  $C_{1}^{\alpha}-C_{2}^{\alpha}$, $\alpha=1,2,...,\frac{d(d+1)}{2}$ are determined by the free boundary value feature of $\mathcal{B}_{\beta}[\varphi]$, $\beta=1,2,...,\frac{d(d+1)}{2}$. So, we call $\mathcal{B}_{\beta}[\varphi]$ the $stress\; concentration\;factor\;or\;blow$-$up\;factor$.

Li \cite{L2018} firstly studied the asymptotic behavior of stress concentration factor $\mathcal{B}_{\beta}[\varphi]$ and obtained the following result (see Proposition 2.1 of \cite{L2018}): for $\beta=1,2,...,\frac{d(d+1)}{2}$,
\begin{align}\label{KGA001}
\mathcal{B}_{\beta}[\varphi]=\mathcal{B}_{\beta}^{\ast}[\varphi]+O(\max\{\varepsilon^{\frac{1}{3}},\rho_{d}(\varepsilon)\}),\quad\rho_{d}(\varepsilon)=
\begin{cases}
\sqrt{\varepsilon},&d=2,\\
|\ln\varepsilon|^{-1},&d=3,
\end{cases}
\end{align}
where $\mathcal{B}_{\beta}^{\ast}[\varphi]$ is given by
\begin{align}\label{KGF001}
\quad\mathcal{B}_{\beta}^{\ast}[\varphi]:=\int_{\partial D_{1}^{\ast}}\frac{\partial u_{b}^{\ast}}{\partial\nu_{0}}\Big|_{+}\cdot\psi_{\beta}.
\end{align}
Here $u_{b}^{\ast}$ is the solution of
\begin{align}\label{LCRD001}
\begin{cases}
\mathcal{L}_{\lambda,\mu}u_{b}^{\ast}=0,&\mathrm{in}\;\Omega^{\ast},\\
u_{b}^{\ast}=\sum\limits^{\frac{d(d+1)}{2}}_{\alpha=1}C^{\alpha}_{\ast}\psi_{\alpha},&\mathrm{on}\;(\partial D_{1}^{\ast}\setminus\{0\})\cup\partial D_{2},\\
%\int_{\partial D_{1}^{\ast}}\frac{\partial u_{b}^{\ast}}{\partial\nu_{0}}\big|_{+}\cdot\psi_{\beta}=0,&\beta=1,2,...,\frac{d(d+1)}{2},\\
u_{b}^{\ast}=\varphi,&\mathrm{on}\;\partial D,
\end{cases}
\end{align}
where the free constants $C_{\ast}^{\alpha}$, $\alpha=1,2,...,\frac{d(d+1)}{2}$ are determined by the following condition:
\begin{align}\label{LDCZ003}
\int_{\partial D_{1}^{\ast}\cup\partial D_{2}}\frac{\partial u_{b}^{\ast}}{\partial \nu_0}\big|_{+}\cdot\psi_{\beta}=0,\quad\beta=1,2,...,\frac{d(d+1)}{2}.
\end{align}
The key point to prove the convergence in \eqref{KGA001} lies in establishing the convergence that $C_{2}^{\alpha}\rightarrow C_{\ast}^{\alpha}$, as $\varepsilon\rightarrow0$. For that purpose, Li \cite{L2018} imposed some special symmetric conditions on the domain and the parity conditions on the boundary data, which leads to that $C_{1}^{\alpha}-C_{2}^{\alpha}=0$, $\alpha=d+1,...,\frac{d(d+1)}{2}$. Then using this fact, Li proved that for $\alpha=1,2,...,d$, $C_{\ast}^{\alpha}$ is the limit of $\frac{1}{2}(C_{1}^{\alpha}+C_{2}^{\alpha})$ as the distance $\varepsilon$ goes to zero, that is,
\begin{align*}
\frac{1}{2}(C_{1}^{\alpha}+C_{2}^{\alpha})=C_{\ast}^{\alpha}+O(\rho_{d}(\varepsilon)),\quad\alpha=1,2,...,d.
\end{align*}
This, together with the fact that $|C_{1}^{\alpha}-C_{2}^{\alpha}|\leq C\rho_{d}(\varepsilon)$, gives that
\begin{align}\label{LGAN}
C_{2}^{\alpha}=C_{\ast}^{\alpha}+O(\rho_{d}(\varepsilon)),\quad\alpha=1,2,...,d.
\end{align}
Since the establishment of \eqref{LGAN} depends on some strict symmetric conditions on the domain and the boundary data and the differences of $C_{1}^{\alpha}-C_{2}^{\alpha}$, $\alpha=1,2,...,d$, then the idea in \cite{L2018} cannot be used to deal with the generalized $m$-convex inclusions in all dimensions, especially when $m<d-1$. Then the main objective of this paper is to get rid of these strong assumed conditions in \cite{L2018} and establish a unified convergence result between $\mathcal{B}_{\beta}[\varphi]$ and $\mathcal{B}_{\beta}^{\ast}[\varphi]$ for the generalized $m$-convex inclusions in any dimension. As an immediate consequence of the stress concentration factors $\mathcal{B}_{\beta}^{\ast}[\varphi]$, $\alpha=1,2,...,\frac{d(d+1)}{2}$, the optimal upper and lower bound estimates and asymptotic expansions of the stress concentration are also established.

%The problem of estimating $|\nabla u|$ in the presence of closely located inclusions was first proposed in \cite{BASL1999} in relation to initiation and growth of damage in composites.
%This work is stimulated by previous numerical investigation in \cite{BASL1999} of Babu\u{s}ka et al.,
The problem of estimating $|\nabla u|$ in the presence of closely located inclusions was first proposed in \cite{BASL1999} of Babu\u{s}ka et al., where the Lam\'{e} system was assumed and they observed computationally that the gradient of the solution keeps bounded regardless of the distance between inclusions. Li and Nirenberg \cite{LN2003} demonstrated this numerical observation and established stronger $C^{1,\alpha}$ estimates for general second-order elliptic systems with piecewise H\"{o}lder continuous coefficients. For the corresponding results on second-order elliptic equations of divergence form, see \cite{LV2000,BV2000,DL2019}. It is worth emphasizing that the gradient estimates in \cite{DL2019} showed the explicit dependence on the ellipticity coefficients, which answered open problem $(b)$ in \cite{LV2000}. For more related open problems, we refer to page 94 of \cite{LV2000} and page 894 of \cite{LN2003}. In addition, Calo, Efendiev and Galvis \cite{CEG2014} gave an asymptotic expansion of a solution to elliptic equation for both high- and low-conductivity inclusions and presented a procedure to compute the terms in the expansion. By utilizing the single and double layer potentials with image line charges, Kim and Lim \cite{KL2019} recently obtained an asymptotic expression of the solution to the conductivity problem in the presence of the core-shell geometry with circular boundaries in two dimensions.

There appears no gradient blow-up in the aforementioned works related to elliptic equations and systems with finite coefficients. However, if the elliptic coefficients are allowed to deteriorate, the situation will be very different. For the scalar conductivity problem, when the conductivity $k$ of inclusions degenerates to infinity, we obtain its limit equation independent of the conductivity $k$, which is called the perfect conductivity equation. It has been discovered by many mathematicians that the electric field, which is the gradient of a solution to the perfect conductivity problem, blows up at the rate of $\varepsilon^{-1/2}$ in dimension two \cite{AKLLL2007,BC1984,BLY2009,AKL2005,Y2007,Y2009,K1993}, $|\varepsilon\ln\varepsilon|^{-1}$ in dimension three \cite{BLY2009,LY2009,BLY2010,L2012}, and $\varepsilon^{-1}$ in dimensions greater than three \cite{BLY2009}, respectively. Besides these aforementioned works related to the gradient estimates, there is a long list of papers \cite{ACKLY2013,KLY2013,KLY2014,KLY2015,LLY2019,Li2020,ZH202101} to pursue a precise characterization for the singular behavior of the gradient. Bonnetier and Triki \cite{BT2013} presented the asymptotic expansions of the eigenvalues for the Poincar\'{e} variational problem in the presence of two nearly touching inclusions as the distance between two inclusions approaches to zero. Kang and Yun \cite{KY201902,KY202002} gave a quantitative characterization for enhancement of the field induced by an emitter in the narrow regions between two circular and spherical inclusions, respectively. The techniques used in above-mentioned works are designed to solve only the linear problem. In order to deal with the nonlinear $p$-Laplace equation, Gorb and Novikov \cite{GN2012,G2015} utilized the method of barriers to obtain $L^{\infty}$ estimates of the gradient of the solution and revealed the explicit dependence of the gradient blow-up rate on nonlinear index $p$. Ciraolo and Sciammetta \cite{CS2019,CS20192} further extended their results to the Finsler $p$-Laplacian.

Although much progress has been made for the scalar equation mentioned above, it is not easy to extend to the linear systems of elasticity for the sake of some technical difficulties. For example, the maximum principle does not hold for the systems of equations. In \cite{LLBY2014} Li, Li, Bao and Yin created an ingenious iteration technique to overcome these difficulties and proved that the gradients for solutions to a class of elliptic systems decay exponentially fast in a thin gap when the same boundary data is imposed on the upper and bottom boundaries of this gap. Subsequently, the iterate technique was extensively used to study the singular behavior of the gradient of a solution to the Lam\'{e} system with partially infinity coefficients. Bao, Li and Li \cite{BLL2015,BLL2017} obtained the pointwise upper bounds on the gradient in all dimensions. Li \cite{L2018} then constructed a lower bound of the gradient by introducing a unified blow-up factor, which answers the optimality of the blow-up rate in dimensions two and three. Li and Xu \cite{LX2020} further improved the results in \cite{L2018} to be precise asymptotic expansions. The gradient estimates and asymptotics in \cite{L2018,LX2020} were established only when the domain and boundary data satisfied some strictly symmetric conditions. The interior estimates mentioned above were also generalized to the boundary case \cite{LZ2019,BJL2017}. Recently, Kang and Yu \cite{KY2019} gave a complete description for the singular behavior of the stress by using nuclei of strain to construct the explicit singular functions. It is worth to mention that the main tools used in \cite{KY2019} are the layer potential techniques and the variational principle, which is different from the iterate technique adopted in the previous works. For more related investigations, see \cite{HJL2018,LX2020,GB2005,G2015,BT2012,ABTV2015} and the references therein.

To list the main results in a precise manner, we further describe the domain and introduce some notations. Suppose that there exists a $\varepsilon$-independent positive constant $R$ such that the top and bottom boundaries of the narrow region between $D_{1}$ and $D_{2}$ are, respectively, the graphs of two $C^{2,\alpha}$ functions $\varepsilon+h_{1}(x')$ and $h_{2}(x')$, and $h_{i}$, $i=1,2$, satisfy the following conditions:
\begin{enumerate}
{\it\item[(\bf{H1})]
$\tau_{1}|x'|^{m}\leq h_{1}(x')-h_{2}(x')\leq\tau_{2}|x'|^{m},\;\mathrm{if}\;x'\in B_{2R}',$
\item[(\bf{H2})]
$|\nabla_{x'}^{j}h_{i}(x')|\leq \tau_{3}|x'|^{m-j},\;\mathrm{if}\;x'\in B_{2R}',\;i,j=1,2,$
\item[(\bf{H3})]
$\|h_{1}\|_{C^{2,\alpha}(B'_{2R})}+\|h_{2}\|_{C^{2,\alpha}(B'_{2R})}\leq \tau_{4},$}
\end{enumerate}
where $m\geq 2$ and $\tau_{i}>0,i=1,2,3,4$, are constants independent of $\varepsilon$. In addition, assume that for $i=1,...,d-1$, $h_{1}(x')-h_{2}(x')$ is even with respect to $x_{i}$ in $B_{R}'$. For $z'\in B'_{R}$ and $0<t\leq2R$, define the thin gap by
\begin{align*}
\Omega_{t}(z'):=&\left\{x\in \mathbb{R}^{d}~\big|~h_{2}(x')<x_{d}<\varepsilon+h_{1}(x'),~|x'-z'|<{t}\right\}.
\end{align*}
For notational simplicity, let the abbreviated notation $\Omega_{t}$ represent $\Omega_{t}(0')$ and denote its upper and lower boundaries by
\begin{align*}
\Gamma^{+}_{r}:=\left\{x\in\mathbb{R}^{d}|\,x_{d}=\varepsilon+h_{1}(x'),\;|x'|<r\right\},~~
\Gamma^{-}_{r}:=\left\{x\in\mathbb{R}^{d}|\,x_{d}=h_{2}(x'),\;|x'|<r\right\},
\end{align*}
respectively.

Define a scalar auxiliary function $\bar{v}\in C^{2}(\mathbb{R}^{d})$ satisfying that $\bar{v}=1$ on $\partial D_{1}$, $\bar{v}=0$ on $\partial D_{2}\cup\partial D$,
\begin{align}\label{zh001}
\bar{v}(x',x_{d}):=\frac{x_{d}-h_{2}(x')}{\varepsilon+h_{1}(x')-h_{2}(x')},\;\,\mathrm{in}\;\Omega_{2R},\quad\mbox{and}~\|\bar{v}\|_{C^{2}(\Omega\setminus\Omega_{R})}\leq C.
\end{align}
Then we introduce a family of vector-valued auxiliary functions as follows: for $\alpha=1,2,...,\frac{d(d+1)}{2}$,
\begin{align}\label{zzwz002}
\bar{u}_{1}^{\alpha}=&\psi_{\alpha}\bar{v}+\mathcal{F}_{\alpha},\quad \mathrm{in}\;\Omega_{2R},
\end{align}
where $\psi_{\alpha}$ is defined in (\ref{OPP}) and
\begin{align}\label{QLA001}
\mathcal{F}_{\alpha}=\frac{\lambda+\mu}{\mu}f(\bar{v})\psi^{d}_{\alpha}\sum^{d-1}_{i=1}\partial_{x_{i}}\delta\,e_{i}+\frac{\lambda+\mu}{\lambda+2\mu}f(\bar{v})\sum^{d-1}_{i=1}\psi^{i}_{\alpha}\partial_{x_{i}}\delta\,e_{d},
\end{align}
and
\begin{align}\label{deta}
\delta(x'):=\varepsilon+h_{1}(x')-h_{2}(x'),\quad f(\bar{v}):=\frac{1}{2}\left(\bar{v}-\frac{1}{2}\right)^{2}-\frac{1}{8}.
\end{align}
It is worth mentioning that the correction terms $\mathcal{F}_{\alpha}$, $\alpha=1,2,...,d$ were previously captured in \cite{LX2020}.

For $i,j=1,2,\,\alpha,\beta=1,2,...,\frac{d(d+1)}{2}$, denote
\begin{align}\label{LMZR}
a_{ij}^{\ast\alpha\beta}=\int_{\Omega^{\ast}}(\mathbb{C}^0e(v_{i}^{\ast\alpha}), e(v_j^{\ast\beta}))dx,\quad b_i^{\ast\alpha}=-\int_{\partial D}\frac{\partial v_{i}^{\ast\alpha}}{\partial \nu_0}\large\Big|_{+}\cdot v_{0}^{\ast},
\end{align}
where $v_{0}^{\ast}$ and $v_{i}^{\ast\alpha}\in{C}^{2}(\Omega^{\ast};\mathbb{R}^d)$, $i=1,2$, $\alpha=1,2,...,\frac{d(d+1)}{2}$, respectively, solve
\begin{align}\label{ZG001}
\begin{cases}
\mathcal{L}_{\lambda,\mu}v_{0}^{\ast}=0,&\mathrm{in}~\Omega^{\ast},\\
v_{0}^{\ast}=0,&\mathrm{on}~\partial{D}_{1}^{\ast}\cup\partial{D_{2}},\\
v_{0}^{\ast}=\varphi,&\mathrm{on}~\partial{D},
\end{cases}
\end{align}
and
\begin{equation}\label{qaz001111}
\begin{cases}
\mathcal{L}_{\lambda,\mu}v_{1}^{\ast\alpha}=0,&\mathrm{in}~\Omega^{\ast},\\
v_{1}^{\ast\alpha}=\psi^{\alpha},&\mathrm{on}~\partial{D}_{1}^{\ast}\setminus\{0\},\\
v_{1}^{\ast\alpha}=0,&\mathrm{on}~\partial{D_{2}}\cup\partial{D},
\end{cases}\quad
\begin{cases}
\mathcal{L}_{\lambda,\mu}v_{2}^{\ast\alpha}=0,&\mathrm{in}~\Omega^{\ast},\\
v_{2}^{\ast\alpha}=\psi^{\alpha},&\mathrm{on}~\partial{D}_{2},\\
v_{2}^{\alpha}=0,&\mathrm{on}~(\partial{D_{1}^{\ast}}\setminus\{0\})\cup\partial{D}.
\end{cases}
\end{equation}
Remark that the definition of $a_{ij}^{\ast\alpha\beta}$ is only valid under certain conditions, see Lemma \ref{lemmabc} below. Denote
\begin{align}
&\mathbb{A}^{\ast}=(a_{11}^{\ast\alpha\beta})_{\frac{d(d+1)}{2}\times\frac{d(d+1)}{2}},\quad \mathbb{B}^{\ast}=\bigg(\sum\limits^{2}_{i=1}a_{i1}^{\ast\alpha\beta}\bigg)_{\frac{d(d+1)}{2}\times\frac{d(d+1)}{2}},\label{WZW}\\
&\mathbb{C}^{\ast}=\bigg(\sum\limits^{2}_{j=1}a_{1j}^{\ast\alpha\beta}\bigg)_{\frac{d(d+1)}{2}\times\frac{d(d+1)}{2}},\quad \mathbb{D}^{\ast}=\bigg(\sum\limits^{2}_{i,j=1}a_{ij}^{\ast\alpha\beta}\bigg)_{\frac{d(d+1)}{2}\times\frac{d(d+1)}{2}}.\label{WEN002}
\end{align}
We now divide into three cases to introduce the blow-up factor matrices, which will be used to construct a family of unified stress concentration factors.

First, if $m\geq d+1$, then for $\alpha=1,2,...,\frac{d(d+1)}{2}$, by replacing the elements of $\alpha$-th column in the matrix $\mathbb{D}^{\ast}$ by column vector $\Big(\sum\limits_{i=1}^{2}b_{i}^{\ast1},...,\sum\limits_{i=1}^{2}b_{i}^{\ast\frac{d(d+1)}{2}}\Big)^{T}$, we get new matrix $\mathbb{D}^{\ast\alpha}$ as follows:
\begin{gather}\label{ZZ001}
\mathbb{D}^{\ast\alpha}=
\begin{pmatrix}
\sum\limits^{2}_{i,j=1}a_{ij}^{\ast11}&\cdots&\sum\limits_{i=1}^{2}b_{i}^{\ast1}&\cdots&\sum\limits^{2}_{i,j=1}a_{ij}^{\ast1\,\frac{d(d+1)}{2}} \\\\ \vdots&\ddots&\vdots&\ddots&\vdots\\\\ \sum\limits^{2}_{i,j=1}a_{ij}^{\ast\frac{d(d+1)}{2}\,1}&\cdots&\sum\limits_{i=1}^{2}b_{i}^{\ast\frac{d(d+1)}{2}}&\cdots&\sum\limits^{2}_{i,j=1}a_{ij}^{\ast\frac{d(d+1)}{2}\,\frac{d(d+1)}{2}}
\end{pmatrix}.
\end{gather}

Second, if $d-1\leq m<d+1$, then we define
\begin{gather}\mathbb{A}_{0}^{\ast}=\begin{pmatrix} a_{11}^{\ast d+1\,d+1}&\cdots&a_{11}^{\ast d+1\frac{d(d+1)}{2}} \\\\ \vdots&\ddots&\vdots\\\\a_{11}^{\ast\frac{d(d+1)}{2}d+1}&\cdots&a_{11}^{\ast\frac{d(d+1)}{2}\frac{d(d+1)}{2}}\end{pmatrix},\label{LAGT001}\\
\mathbb{B}_{0}^{\ast}=\begin{pmatrix} \sum\limits^{2}_{i=1}a_{i1}^{\ast d+1\,1}&\cdots&\sum\limits^{2}_{i=1}a_{i1}^{\ast d+1\,\frac{d(d+1)}{2}} \\\\ \vdots&\ddots&\vdots\\\\ \sum\limits^{2}_{i=1}a_{i1}^{\ast\frac{d(d+1)}{2}1}&\cdots&\sum\limits^{2}_{i=1}a_{i1}^{\ast\frac{d(d+1)}{2}\frac{d(d+1)}{2}}\end{pmatrix},\notag\\
\mathbb{C}_{0}^{\ast}=\begin{pmatrix} \sum\limits^{2}_{j=1}a_{1j}^{\ast1\,d+1}&\cdots&\sum\limits^{2}_{j=1}a_{1j}^{\ast1\frac{d(d+1)}{2}} \\\\ \vdots&\ddots&\vdots\\\\ \sum\limits^{2}_{j=1}a_{1j}^{\ast\frac{d(d+1)}{2}\,d+1}&\cdots&\sum\limits^{2}_{j=1}a_{1j}^{\ast\frac{d(d+1)}{2}\frac{d(d+1)}{2}}\end{pmatrix}.\notag
\end{gather}
For $\alpha=1,2,...,\frac{d(d+1)}{2}$, we replace the elements of $\alpha$-th column in the matrix $\mathbb{B}_{0}^{\ast}$ by column vector $\Big(b_{1}^{\ast d+1},...,b_{1}^{\ast\frac{d(d+1)}{2}}\Big)^{T}$, and then generate new matrix $\mathbb{B}_{0}^{\ast\alpha}$ as follows:
\begin{gather*}
\mathbb{B}_{0}^{\ast\alpha}=
\begin{pmatrix}
\sum\limits^{2}_{i=1}a_{i1}^{\ast d+1\,1}&\cdots&b_{1}^{\ast d+1}&\cdots&\sum\limits^{2}_{i=1}a_{i1}^{\ast d+1\,\frac{d(d+1)}{2}} \\\\ \vdots&\ddots&\vdots&\ddots&\vdots\\\\ \sum\limits^{2}_{i=1}a_{i1}^{\ast\frac{d(d+1)}{2}\,1}&\cdots&b_{1}^{\ast\frac{d(d+1)}{2}}&\cdots&\sum\limits^{2}_{i=1}a_{i1}^{\ast\frac{d(d+1)}{2}\frac{d(d+1)}{2}}
\end{pmatrix}.
\end{gather*}
Denote
\begin{align}\label{ZZ002}
\mathbb{F}_{0}^{\ast}=\begin{pmatrix} \mathbb{A}_{0}^{\ast}&\mathbb{B}_{0}^{\ast} \\  \mathbb{C}_{0}^{\ast}&\mathbb{D}^{\ast}
\end{pmatrix},\quad \mathbb{F}_{0}^{\ast\alpha}=\begin{pmatrix} \mathbb{A}_{0}^{\ast}&\mathbb{B}_{0}^{\ast\alpha} \\  \mathbb{C}_{0}^{\ast}&\mathbb{D}^{\ast\alpha}
\end{pmatrix},\quad\alpha=1,2,...,\frac{d(d+1)}{2}.
\end{align}

Third, if $m<d-1$, then by substituting column vector $\Big(b_{1}^{\ast1},...,b_{1}^{\ast\frac{d(d+1)}{2}}\Big)^{T}$ for the elements of $\alpha$-th column of the matrix $\mathbb{B}^{\ast}$, we obtain the new matrix $\mathbb{B}_{1}^{\ast\alpha}$ as follows:
\begin{gather*}
\mathbb{B}_{1}^{\ast\alpha}=
\begin{pmatrix}
\sum\limits^{2}_{i=1}a_{i1}^{\ast11}&\cdots&b_{1}^{\ast1}&\cdots&\sum\limits^{2}_{i=1}a_{i1}^{\ast1\,\frac{d(d+1)}{2}} \\\\ \vdots&\ddots&\vdots&\ddots&\vdots\\\\ \sum\limits^{2}_{i=1}a_{i1}^{\ast\frac{d(d+1)}{2}\,1}&\cdots&b_{1}^{\ast\frac{d(d+1)}{2}}&\cdots&\sum\limits^{2}_{i=1}a_{i1}^{\ast\frac{d(d+1)}{2}\frac{d(d+1)}{2}}
\end{pmatrix}.
\end{gather*}
Define
\begin{align}\label{ZZ003}
\mathbb{F}_{1}^{\ast}=\begin{pmatrix} \mathbb{A}^{\ast}&\mathbb{B}^{\ast} \\  \mathbb{C}^{\ast}&\mathbb{D}^{\ast}
\end{pmatrix},\quad\mathbb{F}^{\ast\alpha}_{1}=\begin{pmatrix} \mathbb{A}^{\ast}&\mathbb{B}_{1}^{\ast\alpha} \\  \mathbb{C}^{\ast}&\mathbb{D}^{\ast\alpha}
\end{pmatrix},\quad\alpha=1,2,...,\frac{d(d+1)}{2}.
\end{align}

Then using the blow-up factor matrices defined in \eqref{WEN002}--\eqref{ZZ001} and \eqref{ZZ002}--\eqref{ZZ003}, we construct the explicit values of $C_{\ast}^{\alpha}$, $\alpha=1,2,...,\frac{d(d+1)}{2}$ as follows:
\begin{align}\label{ZZWWWW}
C_{\ast}^{\alpha}=&
\begin{cases}
\frac{\det\mathbb{D}^{\ast\alpha}}{\det\mathbb{D}^{\ast}},&m\geq d+1,\\
\frac{\det\mathbb{F}_{0}^{\ast\alpha}}{\det \mathbb{F}_{0}^{\ast}},&d-1\leq m<d+1,\\
\frac{\det\mathbb{F}_{1}^{\ast\alpha}}{\det \mathbb{F}_{1}^{\ast}},&m<d-1.
\end{cases}
\end{align}
It is worth emphasizing that the free constants $C_{\ast}^{\alpha}$, $\alpha=1,2,...,\frac{d(d+1)}{2}$ are exactly constructed in \eqref{ZZWWWW}, but not determined by condition \eqref{LDCZ003} as in \cite{L2018,LX2020}. Moreover, by finding the exact values of $C_{\ast}^{\alpha}$, $\alpha=1,2,...,\frac{d(d+1)}{2}$, we get rid of condition \eqref{LDCZ003} in the definition of the stress concentration factors $\mathcal{B}_{\beta}^{\ast}[\varphi]$, $\beta=1,2,...,\frac{d(d+1)}{2}$, and thus give a more clear understanding for these stress concentration factors.

Throughout this paper, let $O(1)$ be some quantity such that $|O(1)|\leq\,C$, where $C$ denotes various positive constants whose value may differ from line to line and depend only on $m,d,R,\tau_{i},\,i=0,1,2,3,4$ and an upper bound of the $C^{2,\alpha}$ norms of $\partial D_{1}$, $\partial D_{2}$ and $\partial D$, but not on $\varepsilon$.
%Observe that by utilizing the standard elliptic theory (see Agmon et al. \cite{ADN1959,ADN1964}), we know
%\begin{align*}
%\|\nabla u\|_{L^{\infty}(\Omega\setminus\Omega_{R})}\leq\,C\|\varphi\|_{C^{2}(\partial D)}.
%\end{align*}

Denote
\begin{align}\label{JTD}
r_{\varepsilon}:=&
\begin{cases}
\varepsilon^{\min\{\frac{1}{4},\frac{m-d-1}{m}\}},&m>d+1,\\
|\ln\varepsilon|^{-1},&m=d+1,\\
\varepsilon^{\min\{\frac{d+1-m}{12m},\frac{m-d+1}{m}\}},&d-1<m<d+1,\\
|\ln\varepsilon|^{-1},&m=d-1,\\
\varepsilon^{\min\{\frac{1}{6},\frac{d-1-m}{12m}\}},&m<d-1.
\end{cases}
\end{align}
Then our main result in this paper is stated as follows.
\begin{theorem}\label{JGR}
Let $D_{1},D_{2}\subset D\subset\mathbb{R}^{d}\,(d\geq2)$ be defined as above, conditions $\mathrm{(}${\bf{H1}}$\mathrm{)}$--$\mathrm{(}${\bf{H3}}$\mathrm{)}$ hold. Then for a sufficiently small $\varepsilon>0$,
\begin{align}\label{AHNZ001}
\mathcal{B}_{\beta}[\varphi]=\mathcal{B}_{\beta}^{\ast}[\varphi]+O(r_{\varepsilon}),\quad\beta=1,2,...,\frac{d(d+1)}{2},
\end{align}
where the blow-up factors $\mathcal{B}_{\beta}[\varphi]$ and $\mathcal{B}_{\beta}^{\ast}[\varphi]$ are, respectively, defined in \eqref{LGBC} and \eqref{KGF001}, the remaining term $r_{\varepsilon}$ is defined by \eqref{JTD}.
\end{theorem}

\begin{remark}
For every $\beta=1,2,...,\frac{d(d+1)}{2}$, the stress concentration factor $\mathcal{B}_{\beta}^{\ast}[\varphi]$ of this paper is more explicit than that of \cite{L2018,LX2020}, since we give the exact values of $C^{\alpha}_{\ast}$, $\alpha=1,2,...,\frac{d(d+1)}{2}$ in the definition of $u_{b}^{\ast}$ in \eqref{LCRD001}. This is critical to the numerical computations and simulations for these stress concentration factors in future work.
%Moreover, by contrast with those complex blow-up factor matrices captured in \cite{ZH2021}, this stress concentration factor seems more concise and natural, which contributes to the significant reduction of computational work in numerical investigations.

\end{remark}

\begin{remark}
For $0<t\leq2R$, let $\Omega_{t}^{\ast}:=\Omega^{\ast}\cap\{|x'|<t\}$. It follows from \eqref{ZZWWWW}, \eqref{LFN001} and \eqref{KTG001} that for $\beta=1,2,...,\frac{d(d+1)}{2}$,
\begin{align*}
&\left|\int_{\partial D_{1}^{\ast}\cap\partial\Omega_{R}^{\ast}}\frac{\partial u_{b}^{\ast}}{\partial\nu_{0}}\Big|_{+}\cdot\psi_{\beta}\right|\notag\\
&=\left|\int_{\partial D_{1}^{\ast}\cap\partial\Omega_{R}^{\ast}}\sum^{\frac{d(d+1)}{2}}_{\alpha=1}C^{\alpha}_{\ast}\frac{\partial (v_{1}^{\ast\alpha}+v_{2}^{\ast\alpha})}{\partial\nu_{0}}\Big|_{+}\cdot\psi_{\beta}+\int_{\partial D_{1}^{\ast}\cap\partial\Omega_{R}^{\ast}}\frac{\partial v_{0}^{\ast}}{\partial\nu_{0}}\Big|_{+}\cdot\psi_{\beta}\right|\notag\\
&\leq C\int^{R}_{0}t^{d-2-\frac{md}{2}}e^{-\frac{1}{2Ct^{m-1}}}dt\leq CR^{d-1-\frac{md}{2}}e^{-\frac{1}{2CR^{m-1}}}.
\end{align*}
This implies that
\begin{align}\label{TGY}
\mathcal{B}_{\beta}^{\ast}[\varphi]=\mathcal{B}_{\beta}^{\ast}[\varphi]|_{\partial D_{1}^{\ast}\setminus\partial\Omega^{\ast}_{R}}+O(1)R^{d-1-\frac{md}{2}}e^{-\frac{1}{2CR^{m-1}}},
\end{align}
where $\mathcal{B}_{\beta}^{\ast}[\varphi]|_{\partial D_{1}^{\ast}\setminus\partial\Omega^{\ast}_{R}}:=\int_{\partial D_{1}^{\ast}\setminus\partial\Omega_{R}^{\ast}}\frac{\partial u_{b}^{\ast}}{\partial\nu_{0}}|_{+}\cdot\psi_{\beta}.$ From \eqref{TGY}, we conclude that owing to the exponentially decaying property of the gradient $\nabla u_{b}^{\ast}$ in the thin gap, it suffices to utilize regular meshes to numerically calculate the blow-up factor $\mathcal{B}_{\beta}^{\ast}[\varphi]$ outside the thin gap between these two inclusions. This computation method was previously expounded in \cite{KLY2015} for a similar stress concentration factor in the context of the perfect conductivity problem.

\end{remark}

The rest of the paper is organized as follows. Section \ref{SEC005} is devoted to the proof of Theorem \ref{JGR}. In Section \ref{KGRA90}, we apply the stress concentration factors captured in Theorem \ref{JGR} to establish the optimal gradient estimates and the asymptotic expansions in Theorems \ref{MGA001} and \ref{coro00389}, respectively. The corresponding results for the scalar perfect conductivity equation are presented in Section \ref{SECCC000}.

\section{Proof of Theorem \ref{JGR}}\label{SEC005}

For $\alpha=1,2,...,\frac{d(d+1)}{2}$, denote
\begin{align}\label{GAKL001}
\bar{u}_{2}^{\alpha}=&\psi_{\alpha}(1-\bar{v})-\mathcal{F}_{\alpha},
\end{align}
where $\mathcal{F}_{\alpha}$, $\alpha=1,2,...,\frac{d(d+1)}{2}$ are defined in \eqref{QLA001}. We first use the iterate technique built in \cite{LLBY2014} to demonstrate that for $i=1,2,\,\alpha=1,2,...,\frac{d(d+1)}{2}$, $\nabla\bar{u}_{i}^{\alpha}$ is the main term of $\nabla v_{i}^{\alpha}$.
\begin{prop}\label{thm86}
Assume as above. For $i=1,2,\,\alpha=1,2,...,\frac{d(d+1)}{2}$, let $v_{i}^{\alpha}\in H^{1}(\Omega;\mathbb{R}^{d})$ be a weak solution of \eqref{qaz001}. Then, for a sufficiently small $\varepsilon>0$ and $x\in\Omega_{R}$,
\begin{align}\label{Le2.025}
\nabla v_{i}^{\alpha}=\nabla\bar{u}_{i}^{\alpha}+O(1)
\begin{cases}
\delta^{\frac{m-2}{m}},&\alpha=1,2,...,d,\\
1,&\alpha=d+1,...,\frac{d(d+1)}{2},
\end{cases}
\end{align}
where $\delta$ is defined in \eqref{deta}, the main terms $\bar{u}_{i}^{\alpha}$, $i=1,2,\,\alpha=1,2,...,\frac{d(d+1)}{2}$ are defined in \eqref{zzwz002} and \eqref{GAKL001}, respectively.
\end{prop}
\begin{proof}
Take the case of $i=1$ for example. The case of $i=2$ is the same and thus omitted. To begin with, it follows from a straightforward computation that for $\alpha=1,2,...,\frac{d(d+1)}{2}$,
\begin{align}\label{DM001}
|\mathcal{L}_{\lambda,\mu}\bar{u}_{1}^{\alpha}|\leq&C\left(|\psi_{\alpha}|\delta^{-2/m}+|\nabla_{x'}\psi_{\alpha}|\delta^{-1}+1\right),\quad \mathrm{in}\;\Omega_{2R}.
\end{align}

For $x\in\Omega_{2R}$, denote
\begin{equation*}
w_{\alpha}:=v_{1}^{\alpha}-\bar{u}_{1}^{\alpha},\quad \alpha=1,2,...,\frac{d(d+1)}{2}.
\end{equation*}
Then $w_{\alpha}$ satisfies
\begin{align}\label{RNZ}
\begin{cases}
\mathcal{L}_{\lambda,\mu}w_{\alpha}=-\mathcal{L}_{\lambda,\mu}\bar{u}_{1}^{\alpha},&
\hbox{in}\  \Omega_{2R},  \\
w_{\alpha}=0, \quad&\hbox{on} \ \Gamma^{\pm}_{2R}.
\end{cases}
\end{align}

\noindent{\bf Step 1.}
Let $w_{\alpha}\in H^1(\Omega;\mathbb{R}^{d})$ be a weak solution of \eqref{RNZ}. Then
\begin{align}\label{YGZA001}
\int_{\Omega_{R}}|\nabla w_{\alpha}|^2dx\leq C,\quad \alpha=1,2,...,\frac{d(d+1)}{2}.
\end{align}

Multiplying equation \eqref{RNZ} by $w_{\alpha}$ and integrating by parts, we get
\begin{align}\label{LNTZ0010}
&\int_{\Omega_{R}}\left(\mathbb{C}^0e(w_{\alpha}),e(w_{\alpha})\right)dx\notag\\
&=\int_{\Omega_{R}}(\mathcal{L}_{\lambda,\mu}\bar{u}_{1}^{\alpha})w_{\alpha}dx+\int\limits_{\scriptstyle |x'|={R},\atop\scriptstyle
h_{2}(x')<x_{d}<\varepsilon+h_1(x')\hfill}\frac{\partial w_{\alpha}}{\partial\nu_{0}}\Big|_{+}\cdot w_{\alpha}.
\end{align}

On one hand, from \eqref{JFAMC001} and \eqref{ellip}, we deduce
\begin{align}\label{equ1}
\tau_{0}\int_{\Omega_{R}}|\nabla w_{\alpha}|^2dx
\leq\int_{\Omega_{R}}\left(\mathbb{C}^0e(w_{\alpha}),e(w_{\alpha})\right)dx.
\end{align}

On the other hand, we have
\begin{align}\label{MFD}
&\int_{\Omega_{R}}(\mathcal{L}_{\lambda,\mu}\bar{u}_{1}^{\alpha})w_{\alpha}dx+\int\limits_{\scriptstyle |x'|={R},\atop\scriptstyle
h_{2}(x')<x_{d}<\varepsilon+h_1(x')\hfill}\frac{\partial w_{\alpha}}{\partial\nu_{0}}\Big|_{+}\cdot w_{\alpha}\notag\\
&\leq C\int_{|x'|<R}\delta^{-1}(x')dx'\int_{h_{2}(x')}^{\varepsilon+h_{1}(x')}\left|\int^{x_{d}}_{h_{2}(x')}\partial_{x_{d}}w_{\alpha}(x',t)dt\right|dx_{d}\notag\\
&\quad+\int\limits_{\scriptstyle |x'|=R,\atop\scriptstyle
h_{2}(x')<x_{d}<\varepsilon+h_1(x')\hfill}C\left(|\nabla w_{\alpha}|^2+|w_{\alpha}|^2\right)\notag\\
&\leq C\int_{\Omega_{R}}|\nabla w_{\alpha}|dx+\int\limits_{\scriptstyle |x'|=R,\atop\scriptstyle
h_{2}(x')<x_{d}<\varepsilon+h_1(x')\hfill}C\left(|\nabla w_{\alpha}|^2+|w_{\alpha}|^2\right)\notag\\
&\leq\frac{\tau_{0}}{2}\|\nabla w_{\alpha}\|_{L^{2}(\Omega_{R})}^{2}+\int\limits_{\scriptstyle |x'|=R,\atop\scriptstyle
h_{2}(x')<x_{d}<\varepsilon+h_1(x')\hfill}C\left(|\nabla w_{\alpha}|^2+|w_{\alpha}|^2\right)+C,
\end{align}
where in the last line we used the Cauchy inequality. Due to the fact that $w_{\alpha}=0$ on $\Gamma_{2R}^{\pm}$ and $\overline{\Omega_{4/3R}}\setminus \Omega_{2/3R}\subset\left((\Omega_{2R}\setminus\overline{\Omega_{1/2R}})\cup (\Gamma_{2R}^{\pm}\setminus \Gamma_{1/2R}^{\pm})\right)$,
it follows from the standard elliptic estimates, the Sobolev embedding theorem and classical $W^{2, p}$ estimates for elliptic systems that for some $p>n$,
\begin{align*}
\|\nabla w_{\alpha}\|_{L^\infty(\Omega_{4/3R}\setminus \overline{\Omega_{2/3R}})}&\leq C\|w_{\alpha}\|_{W^{2,p}(\Omega_{4/3R}\setminus \overline{\Omega_{2/3R})}}\\
&\leq C\big(\|w_{\alpha}\|_{L^2(\Omega_{2R}\setminus \overline{\Omega_{1/2R}})}+\|\mathcal{L}_{\lambda,\mu}\bar{u}_{1}^{\alpha}\|_{L^\infty(\Omega_{2R}\setminus \overline{\Omega_{1/2R}})}\big)\nonumber\\
&\leq C\big(\|v_{1}^{\alpha}\|_{L^2(\Omega_{2R}\setminus \overline{\Omega_{1/2R}})}+1\big)\leq C,
\end{align*}
and for $x=(x',x_d)\in\Omega_{4/3R}\setminus \overline{\Omega_{2/3R}}$,
\begin{align*}
|w_{\alpha}(x',x_d)|&=|w_{\alpha}(x',x_d)-w_{\alpha}(x',h(x'))|\leq C\delta(x')\|\nabla w_{\alpha}\|_{L^\infty(\Omega_{4/3R}\setminus  \overline{\Omega_{2/3R}})}\leq C.
\end{align*}
Then, we obtain
\begin{align}\label{KLN0000}
&\int\limits_{\scriptstyle |x'|={R},\atop\scriptstyle
h_{2}(x')<x_{d}<\varepsilon+h_1(x')\hfill}\big(|w_{\alpha}|^{2}+|\nabla w_{\alpha}|^{2}\big)\leq C.
\end{align}

Therefore, substituting \eqref{equ1}--\eqref{KLN0000} into \eqref{LNTZ0010}, we obtain that \eqref{YGZA001} holds.

\noindent{\bf Step 2.}
Claim that for $\alpha=1,2,...,\frac{d(d+1)}{2}$,
\begin{align}\label{HN001}
 \int_{\Omega_\delta(z')}|\nabla w_{\alpha}|^2dx\leq& C\delta^{d}(|\psi_{\alpha}|^{2}\delta^{2-4/m}+|\nabla_{x'}\psi_{\alpha}|+\delta^{2}).
\end{align}
For $|z'|\leq R$, $\delta\leq t<s\leq\vartheta(\tau_{1},\tau_{3})\delta^{1/m}$, $\vartheta(\tau_{1},\tau_{3})=\frac{1}{2^{m+2}\tau_{3}\max\{1,\tau_{1}^{1/m-1}\}}$, it follows from conditions ({\bf{H1}}) and ({\bf{H2}}) that for  $(x',x_{d})\in\Omega_{s}(z')$,
\begin{align}\label{KHW01}
|\delta(x')-\delta(z')|\leq&|h_{1}(x')-h_{1}(z')|+|h_{2}(x')-h_{2}(z')|\notag\\
\leq&(|\nabla_{x'}h_{1}(x'_{\theta_{1}})|+|\nabla_{x'}h_{2}(x'_{\theta_{2}})|)|x'-z'|\notag\\
\leq&\tau_{3}s(|x'_{\theta_{1}}|^{m-1}+|x'_{\theta_{2}}|^{m-1})\notag\\
\leq&2^{m-1}\tau_{3}s(s^{m-1}+|z'|^{m-1})\notag\\
\leq&\frac{\delta(z')}{2}.
\end{align}
Then, we have
\begin{align}\label{QWN001}
\frac{1}{2}\delta(z')\leq\delta(x')\leq\frac{3}{2}\delta(z'),\quad\mathrm{in}\;\Omega_{s}(z').
\end{align}
Let $\eta\in C^{2}(\Omega_{2R})$ be a smooth cutoff function such that $\eta(x')=1$ if $|x'-z'|<t$, $\eta(x')=0$ if $|x'-z'|>s$, $0\leq\eta(x')\leq1$ if $t\leq|x'-z'|\leq s$, and $|\nabla_{x'}\eta(x')|\leq\frac{2}{s-t}$. Then multiplying equation \eqref{RNZ} by $w_{\alpha}\eta^{2}$ and using integration by parts, we obtain
\begin{align}\label{HN002}
\int_{\Omega_{s}(z')}\left(\mathbb{C}^0e(w_{\alpha}),e(w_{\alpha}\eta^{2})\right)dx=\int_{\Omega_{s}(z')}w_{\alpha}\eta^{2}\left(\mathcal{L}_{\lambda,\mu}\bar{u}_{1}^{\alpha}\right)dx.
\end{align}
On one hand, it follows from \eqref{JFAMC001}, \eqref{ellip} and the first Korn's inequality that
\begin{align}\label{HN003}
\int_{\Omega_{s}(z')}\left(\mathbb{C}^0e(w_{\alpha}),e(w_{\alpha}\eta^{2})\right)dx\geq\frac{1}{C}\int_{\Omega_{s}(z')}|\eta\nabla w_{\alpha}|^{2}-C\int_{\Omega_{s}(z')}|\nabla\eta|^{2}|w_{\alpha}|^{2}.
\end{align}
On the other hand, from the H\"{o}lder inequality and the Cauchy inequality, we have
\begin{align}\label{KDZ001}
&\left|\int_{\Omega_{s}(z')}w_{\alpha}\eta^{2}\left(\mathcal{L}_{\lambda,\mu}\bar{u}_{1}^{\alpha}\right)dx\right|\notag\\
&\leq \frac{C}{(s-t)^{2}}\int_{\Omega_{s}(z')}|w_{\alpha}|^{2}dx+C(s-t)^{2}\int_{\Omega_{s}(z')}|\mathcal{L}_{\lambda,\mu}\bar{u}_{1}^{\alpha}|^{2}dx.
\end{align}
Due to the fact that $w_{\alpha}=0$ on $\Gamma^{-}_{R}$, it follows from \eqref{QWN001} that
\begin{align}\label{RTYII001}
\int_{\Omega_{s}(z')}|w_{\alpha}|^{2}\leq C\delta^{2}\int_{\Omega_{s}(z')}|\nabla w_{\alpha}|^{2}.
\end{align}
From \eqref{DM001} and \eqref{QWN001}, we have
\begin{align}\label{RTYII002}
\int_{\Omega_{s}(z')}|\mathcal{L}_{\lambda,\mu}\bar{u}_{1}^{\alpha}|^{2}\leq&C\delta^{-1}s^{d-1}(|\psi_{\alpha}|^{2}\delta^{2-4/m}+|\nabla_{x'}\psi_{\alpha}|^{2}+\delta^{2}).
\end{align}
Then substituting \eqref{HN003}--\eqref{RTYII002} into \eqref{HN002}, we obtain the iteration formula as follows:
\begin{align}\label{GM001}
\int_{\Omega_{t}(z')}|\nabla w_{\alpha}|^{2}dx\leq&\left(\frac{c\delta}{s-t}\right)^{2}\int_{\Omega_{s}(z')}|\nabla w_{\alpha}|^{2}dx\notag\\
&+C\delta^{-1}s^{d-1}(s-t)^{2}(|\psi_{\alpha}|^{2}\delta^{2-4/m}+|\nabla_{x'}\psi_{\alpha}|^{2}+\delta^{2}).
\end{align}

Write
$$F(t):=\int_{\Omega_{t}(z')}|\nabla w_{\alpha}|^{2}.$$
Let $s=t_{i+1}$, $t=t_{i}$, $t_{i}=\delta+2ci\delta,\;i=0,1,2,...,\left[\frac{\vartheta(\tau_{1},\tau_{3})}{4c\delta^{1-1/m}}\right]+1$. Then \eqref{GM001} becomes
\begin{align*}
F(t_{i})\leq&\frac{1}{4}F(t_{i+1})+C(i+1)^{d-1}\delta^{d}(|\psi_{\alpha}|^{2}\delta^{2-4/m}+|\nabla_{x'}\psi_{\alpha}|^{2}+\delta^{2}).
\end{align*}
Hence, after $\left[\frac{\vartheta(\tau_{1},\tau_{3})}{4c\delta^{1-1/m}}\right]+1$ iterations, it follows from \eqref{YGZA001} that for a sufficiently small $\varepsilon>0$,
\begin{align*}
F(t_{0})\leq \delta^{d}(|\psi_{\alpha}|^{2}\delta^{2-4/m}+|\nabla_{x'}\psi_{\alpha}|^{2}+\delta^{2}).
\end{align*}
That is, \eqref{HN001} holds.

\noindent{\bf Step 3.}
Claim that for $\alpha=1,2,...,\frac{d(d+1)}{2}$,
\begin{align}\label{AQ3.052}
|\nabla w_{\alpha}|\leq&C(|\psi_{\alpha}|\delta^{1-2/m}+|\nabla_{x'}\psi_{\alpha}|+\delta),\quad\mathrm{in}\;\Omega_{R}.
\end{align}

First, by performing a change of variables as follows:
\begin{align*}
\begin{cases}
x'-z'=\delta y',\\
x_{d}=\delta y_{d}.
\end{cases},\quad (x',x_{d})\in\Omega_{\delta}(z'),
\end{align*}
we rescale $\Omega_{\delta}(z')$ into $Q_{1}$, where, for $0<r\leq 1$,
\begin{align*}
Q_{r}=\left\{y\in\mathbb{R}^{d}\,\Big|\,\frac{1}{\delta}h_{2}(\delta y'+z')<y_{d}<\frac{\varepsilon}{\delta}+\frac{1}{\delta}h_{1}(\delta y'+z'),\;|y'|<r\right\}.
\end{align*}
Denote by
\begin{align*}
\widehat{\Gamma}^{+}_{r}=&\left\{y\in\mathbb{R}^{d}\,\Big|\,y_{d}=\frac{\varepsilon}{\delta}+\frac{1}{\delta}h_{1}(\delta y'+z'),\;|y'|<r\right\}
\end{align*}
and
\begin{align*}
\widehat{\Gamma}^{-}_{r}=&\left\{y\in\mathbb{R}^{d}\,\Big|\,y_{d}=\frac{1}{\delta}h_{2}(\delta y'+z'),\;|y'|<r\right\}
\end{align*}
the top and bottom boundaries of $Q_{r}$, respectively. Analogous to \eqref{KHW01}, we get that for $x\in\Omega_{\delta}(z')$,
\begin{align*}
|\delta(x')-\delta(z')|
\leq&2^{m-1}\tau_{3}\delta(\delta^{m-1}+|z'|^{m-1})\leq2^{m+1}\tau_{3}\max\{1,\tau_{1}^{1/m-1}\}\delta^{2-1/m},
\end{align*}
which yields that
\begin{align*}
\left|\frac{\delta(x')}{\delta(z')}-1\right|\leq2^{m+2}\max\{\tau_{1}^{1-1/m},1\}\tau_{3}R^{m-1}.
\end{align*}
Then $Q_{1}$ is of nearly unit size as far as applications of Sobolev embedding theorems and classical $L^{p}$ estimates for elliptic systems are concerned, since $R$ is a small positive constant.

Let
\begin{align*}
W_{\alpha}(y',y_d):=w_{\alpha}(\delta y'+z',\delta y_d),\quad \bar{U}_{1}^{\alpha}(y',y_d):=\bar{u}_{1}^{\alpha}(\delta y'+z',\delta y_d).
\end{align*}
Then $W_{\alpha}(y)$ satisfies
\begin{align*}
\begin{cases}
\mathcal{L}_{\lambda,\mu}W_{\alpha}=-\mathcal{L}_{\lambda,\mu}\bar{U}_{1}^{\alpha},&
\hbox{in}\  Q_{1},  \\
W_{\alpha}=0, \quad&\hbox{on} \ \widehat{\Gamma}^{\pm}_{1}.
\end{cases}
\end{align*}
Then combining the Poincar\'{e} inequality, the Sobolev embedding theorem and classical $W^{2, p}$ estimates for elliptic systems, we obtain that for some $p>d$,
\begin{align*}
\|\nabla W_{\alpha}\|_{L^{\infty}(Q_{1/2})}\leq C\|W_{\alpha}\|_{W^{2,p}(Q_{1/2})}\leq C(\|\nabla W_{\alpha}\|_{L^{2}(Q_{1})}+\|\mathcal{L}_{\lambda,\mu}\bar{U}_{1}^{\alpha}\|_{L^{\infty}(Q_{1})}),
\end{align*}
which reads that
\begin{align}\label{ADQ601}
\|\nabla w_{\alpha}\|_{L^{\infty}(\Omega_{\delta/2}(z'))}\leq\frac{C}{\delta}\left(\delta^{1-\frac{d}{2}}\|\nabla w_{\alpha}\|_{L^{2}(\Omega_{\delta}(z'))}+\delta^{2}\|\mathcal{L}_{\lambda,\mu}\bar{u}_{1}^{\alpha}\|_{L^{\infty}(\Omega_{\delta}(z'))}\right).
\end{align}
Therefore, substituting \eqref{DM001} and \eqref{HN001} into \eqref{ADQ601}, we deduce that \eqref{AQ3.052} holds. That is, Proposition \ref{thm86} is established.

\end{proof}

Applying the proof of Proposition \ref{thm86} with minor modification, we have
\begin{corollary}\label{coro00z}
Assume as above. Let $v_{0},v_{i}^{\alpha}$ and $v_{0}^{\ast},v_{i}^{\ast\alpha}$, $i=1,2$, $\alpha=1,2,...,\frac{d(d+1)}{2}$ be the solutions of \eqref{qaz001} and \eqref{ZG001}--\eqref{qaz001111}, respectively. Then for a sufficiently small $\varepsilon>0$,
\begin{align*}
|\nabla v_{0}|+\left|\sum^{2}_{i=1}\nabla v_{i}^{\alpha}\right|\leq C\delta^{-\frac{d}{2}}e^{-\frac{1}{2C\delta^{1-1/m}}},\;\;\mathrm{in}\;\Omega_{R},
\end{align*}
and
\begin{align}\label{LFN001}
|\nabla v_{0}^{\ast}|+\left|\sum^{2}_{i=1}\nabla v_{i}^{\ast\alpha}\right|\leq C|x'|^{-\frac{md}{2}}e^{-\frac{1}{2C|x'|^{m-1}}},\;\;\mathrm{in}\;\Omega_{R}^{\ast}.
\end{align}
\end{corollary}

For $i,j=1,2$ and $\alpha, \beta=1,2,...,\frac{d(d+1)}{2}$, denote
\begin{align}\label{KAT001}
a_{ij}^{\alpha\beta}:=-\int_{\partial{D}_{j}}\frac{\partial v_{i}^{\alpha}}{\partial \nu_0}\large\Big|_{+}\cdot\psi_{\beta},\quad b_j^{\beta}:=\int_{\partial D_{j}}\frac{\partial v_{0}}{\partial \nu_0}\large\Big|_{+}\cdot \psi_{\beta}.
\end{align}
From \eqref{Le2.01222}, we obtain
\begin{align*}
a_{ij}^{\alpha\beta}=\int_{\Omega}(\mathbb{C}^{0}e(v_{i}^{\alpha}),e(v_{j}^{\beta})),\quad i,j=1,2,\;\alpha,\beta=1,2,...,\frac{d(d+1)}{2}.
\end{align*}
Applying the fourth line of \eqref{La.002} to decomposition \eqref{Decom}, we get
\begin{align}\label{AHNTW009}
\begin{cases}
\sum\limits_{\alpha=1}^{\frac{d(d+1)}{2}}(C_{1}^\alpha-C_{2}^{\alpha}) a_{11}^{\alpha\beta}+\sum\limits_{\alpha=1}^{\frac{d(d+1)}{2}}C_{2}^\alpha \sum\limits^{2}_{i=1}a_{i1}^{\alpha\beta}=b_1^\beta,\\
\sum\limits_{\alpha=1}^{\frac{d(d+1)}{2}}(C_{1}^\alpha-C_{2}^{\alpha}) a_{12}^{\alpha\beta}+\sum\limits_{\alpha=1}^{\frac{d(d+1)}{2}}C_{2}^\alpha \sum\limits^{2}_{i=1}a_{i2}^{\alpha\beta}=b_2^\beta.
\end{cases}
\end{align}
Then adding the first line of \eqref{AHNTW009} to the second line, we derive
\begin{align}\label{zzw002}
\begin{cases}
\sum\limits_{\alpha=1}^{\frac{d(d+1)}{2}}(C_{1}^\alpha-C_{2}^{\alpha}) a_{11}^{\alpha\beta}+\sum\limits_{\alpha=1}^{\frac{d(d+1)}{2}}C_{2}^\alpha \sum\limits^{2}_{i=1}a_{i1}^{\alpha\beta}=b_1^\beta,\\
\sum\limits_{\alpha=1}^{\frac{d(d+1)}{2}}(C_{1}^{\alpha}-C_{2}^{\alpha})\sum\limits^{2}_{j=1}a_{1j}^{\alpha\beta}+\sum\limits_{\alpha=1}^{\frac{d(d+1)}{2}}C_{2}^\alpha \sum\limits^{2}_{i,j=1}a_{ij}^{\alpha\beta}=\sum\limits^{2}_{i=1}b_{i}^{\beta}.
\end{cases}
\end{align}
For simplicity, denote
\begin{align*}
&X^{1}=\big(C_{1}^1-C_{2}^{1},...,C_{1}^\frac{d(d+1)}{2}-C_{2}^{\frac{d(d+1)}{2}}\big)^{T},\quad X^{2}=\big(C_{2}^{1},...,C_{2}^{\frac{d(d+1)}{2}}\big)^{T},\\
&Y^{1}=(b_1^1,...,b_{1}^{\frac{d(d+1)}{2}})^{T},\quad Y^{2}=\bigg(\sum\limits^{2}_{i=1}b_{i}^{1},...,\sum\limits^{2}_{i=1}b_{i}^{\frac{d(d+1)}{2}}\bigg)^{T},\notag
\end{align*}
and
%\begin{gather*}A=\begin{pmatrix} a_{11}^{11}&\cdots&a_{11}^{1\frac{d(d+1)}{2}} \\\\ \vdots&\ddots&\vdots\\\\a_{\frac{d(d+1)}{2}1}&\cdots&a_{\frac{d(d+1)}{2}\frac{d(d+1)}{2}}\end{pmatrix}  ,\;\,
%B=\begin{pmatrix} \sum\limits^{2}_{i=1}a_{i1}^{11}&\cdots&\sum\limits^{2}_{i=1}a_{i1}^{1\frac{d(d+1)}{2}} \\\\ \vdots&\ddots&\vdots\\\\ \sum\limits^{2}_{i=1}a_{i1}^{\frac{d(d+1)}{2}1}&\cdots&\sum\limits^{2}_{i=1}a_{i1}^{\frac{d(d+1)}{2}\frac{d(d+1)}{2}}\end{pmatrix} ,\end{gather*}\begin{gather*}
%C=\begin{pmatrix} \sum\limits^{2}_{j=1}a_{1j}^{11}&\cdots&\sum\limits^{2}_{j=1}a_{1j}^{1\frac{d(d+1)}{2}} \\\\ \vdots&\ddots&\vdots\\\\ \sum\limits^{2}_{j=1}a_{1j}^{\frac{d(d+1)}{2}1}&\cdots&\sum\limits^{2}_{j=1}a_{1j}^{\frac{d(d+1)}{2}\frac{d(d+1)}{2}}\end{pmatrix},\;\,
%D=\begin{pmatrix} \sum\limits^{2}_{i,j=1}a_{ij}^{11}&\cdots&\sum\limits^{2}_{i,j=1}a_{ij}^{1\frac{d(d+1)}{2}} \\\\ \vdots&\ddots&\vdots\\\\ \sum\limits^{2}_{i,j=1}a_{ij}^{\frac{d(d+1)}{2}1}&\cdots&\sum\limits^{2}_{i,j=1}a_{ij}^{\frac{d(d+1)}{2}\frac{d(d+1)}{2}}\end{pmatrix}.
%\end{gather*}
\begin{align}
&\mathbb{A}=(a_{11}^{\alpha\beta})_{\frac{d(d+1)}{2}\times\frac{d(d+1)}{2}},\quad \mathbb{B}=\bigg(\sum\limits^{2}_{i=1}a_{i1}^{\alpha\beta}\bigg)_{\frac{d(d+1)}{2}\times\frac{d(d+1)}{2}},\label{GGDA01}\\
&\mathbb{C}=\bigg(\sum\limits^{2}_{j=1}a_{1j}^{\alpha\beta}\bigg)_{\frac{d(d+1)}{2}\times\frac{d(d+1)}{2}},\quad \mathbb{D}=\bigg(\sum\limits^{2}_{i,j=1}a_{ij}^{\alpha\beta}\bigg)_{\frac{d(d+1)}{2}\times\frac{d(d+1)}{2}}.\label{GGDA}
\end{align}
Then \eqref{zzw002} can be rewritten as
\begin{gather}\label{PLA001}
\begin{pmatrix} \mathbb{A}&\mathbb{B} \\  \mathbb{C}&\mathbb{D}
\end{pmatrix}
\begin{pmatrix}
X^{1}\\
X^{2}
\end{pmatrix}=
\begin{pmatrix}
Y^{1}\\
Y^{2}
\end{pmatrix}.
\end{gather}
Due to the symmetry of $a_{ij}^{\alpha\beta}=a_{ji}^{\beta\alpha}$, we see that $\mathbb{C}=\mathbb{B}^{T}$.

Write
\begin{align}
&(\mathcal{L}_{2}^{1},\mathcal{L}_{2}^{2},\mathcal{L}_{2}^{3})=(\mu,\lambda+2\mu,\lambda+2\mu),\quad d=2,\label{AZM001}\\
&(\mathcal{L}_{d}^{1},...,\mathcal{L}_{d}^{d-1},\mathcal{L}_{d}^{d},...,\mathcal{L}_{d}^{2d-1},\mathcal{L}_{d}^{2d},...,\mathcal{L}_{d}^{\frac{d(d+1)}{2}})\notag\\
&=(\mu,...,\mu,\lambda+2\mu,...,\lambda+2\mu,2\mu,...,2\mu),\quad d\geq3.\label{AZ110}
\end{align}
For $i=0,2$, define
\begin{align}\label{rate00}
\rho_{i}(d,m;\varepsilon):=&
\begin{cases}
\varepsilon^{\frac{d+i-1}{m}-1},&m>n+i-1,\\
|\ln\varepsilon|,&m=d+i-1,\\
1,&m<d+i-1.
\end{cases}
\end{align}

For the purpose of solving $X^{2}=(C_{2}^{1},...,C_{2}^{\frac{d(d+1)}{2}})$ in \eqref{PLA001}, we need the following two lemmas.

\begin{lemma}\label{lemmabc}
Assume as above. Then, for a sufficiently small $\varepsilon>0$,

$(\rm{i})$ for $\alpha=1,2,...,d$, then
\begin{align}\label{LMC}
\begin{cases}
\tau_{2}^{-\frac{d-1}{m}}\lesssim\frac{a_{11}^{\alpha\alpha}}{\mathcal{L}_{d}^{\alpha}\rho_{0}(d,m;\varepsilon)}\lesssim\tau_{1}^{-\frac{d-1}{m}},&m\geq d-1,\\
a_{11}^{\alpha\alpha}=a^{\ast\alpha\alpha}_{11}+O(\varepsilon^{\min\{\frac{1}{6},\frac{d-1-m}{12m}\}}),&m<d-1;
\end{cases}
\end{align}

$(\rm{ii})$ for $\alpha=d+1,...,\frac{d(d+1)}{2}$, then
\begin{align}\label{LMC1}
\begin{cases}
\tau_{2}^{-\frac{d+1}{m}}\lesssim\frac{a_{11}^{\alpha\alpha}}{\mathcal{L}_{d}^{\alpha}\rho_{2}(d,m;\varepsilon)}\lesssim\tau_{1}^{-\frac{d+1}{m}},&m\geq d+1,\\
a_{11}^{\alpha\alpha}=a_{11}^{\ast\alpha\alpha}+O(\varepsilon^{\min\{\frac{1}{6},\frac{d+1-m}{12m}\}}),&m<d+1;
\end{cases}
\end{align}

$(\rm{iii})$ if $d=2$, for $\alpha,\beta=1,2,\alpha\neq\beta$, then
\begin{align}\label{M001}
a_{11}^{12}=a_{11}^{21}=O(1)|\ln\varepsilon|,
\end{align}
and if $d\geq3$, for $\alpha,\beta=1,2,...,d,\,\alpha\neq\beta$, then
\begin{align}\label{M002}
a_{11}^{\alpha\beta}=a_{11}^{\beta\alpha}=a_{11}^{\ast\alpha\beta}+O(1)\varepsilon^{\min\{\frac{1}{6},\frac{d-2}{12m}\}},
\end{align}
and if $d\geq2$, for $\alpha=1,2,...,d,\,\beta=d+1,...,\frac{d(d+1)}{2},$ then
\begin{align}\label{M003}
a_{11}^{\alpha\beta}=a_{11}^{\beta\alpha}=a_{11}^{\ast\alpha\beta}+O(1)\varepsilon^{\min\{\frac{1}{6},\frac{d-1}{12m}\}},
\end{align}
and if $d\geq3$, for $\alpha,\beta=d+1,...,\frac{d(d+1)}{2},\,\alpha\neq\beta$, then
\begin{align}\label{M005}
a_{11}^{\alpha\beta}=a_{11}^{\beta\alpha}=a_{11}^{\ast\alpha\beta}+O(1)\varepsilon^{\min\{\frac{1}{6},\frac{d}{12m}\}};
\end{align}

$(\rm{iv})$ for $\alpha,\beta=1,2,...,\frac{d(d+1)}{2}$,
\begin{align}\label{AZQ001}
\sum\limits^{2}_{i=1}a_{i1}^{\alpha\beta}=&\sum\limits^{2}_{i=1}a_{i1}^{\ast\alpha\beta}+O(\varepsilon^{\frac{1}{4}}),\quad\sum\limits^{2}_{j=1}a_{1j}^{\alpha\beta}=\sum\limits^{2}_{j=1}a_{1j}^{*\alpha\beta}+O(\varepsilon^{\frac{1}{4}}),
\end{align}
and thus
\begin{align*}
\sum\limits^{2}_{i,j=1}a_{ij}^{\alpha\beta}=\sum\limits^{2}_{i,j=1}a_{ij}^{\ast\alpha\beta}+O(\varepsilon^{\frac{1}{4}}).
\end{align*}
\end{lemma}

\begin{proof}
\noindent{\bf Step 1.} Proofs of \eqref{LMC}--\eqref{LMC1}. Introduce a family of auxiliary functions as follows: for $\alpha=1,2,...,\frac{d(d+1)}{2}$,
\begin{align*}
\bar{u}^{\ast\alpha}_{1}=&\psi_{\alpha}\bar{v}^{\ast}+\mathcal{F}_{\alpha}^{\ast},\;\;\mathcal{F}_{\alpha}^{\ast}=\frac{\lambda+\mu}{\mu}f(\bar{v}^{\ast})\psi^{d}_{\alpha}\sum^{d-1}_{i=1}\partial_{x_{i}}\delta\,e_{i}+\frac{\lambda+\mu}{\lambda+2\mu}f(\bar{v}^{\ast})\sum^{d-1}_{i=1}\psi^{i}_{\alpha}\partial_{x_{i}}\delta\,e_{d},
%\bar{u}_{2}^{\ast\beta}=&\psi_{\beta}(1-\bar{v}^{\ast})-\frac{\lambda+\mu}{\mu}f(\bar{v}^{\ast})\psi^{d}_{\beta}\sum^{d-1}_{i=1}\partial_{x_{i}}\delta\,e_{i}-\frac{\lambda+\mu}{\lambda+2\mu}f(\bar{v}^{\ast})\sum^{d-1}_{i=1}\psi^{i}_{\beta}\partial_{x_{i}}\delta\,e_{d}\label{ZWZ002}.
\end{align*}
where $\bar{v}^{\ast}$ verifies that $\bar{v}^{\ast}=1$ on $\partial D_{1}^{\ast}\setminus\{0\}$, $\bar{v}^{\ast}=0$ on $\partial D_{2}\cup\partial D$, and
\begin{align*}
\bar{v}^{\ast}(x',x_{d})=\frac{x_{d}-h_{2}(x')}{h_{1}(x')-h_{2}(x')},\;\,\mathrm{in}\;\Omega_{2R}^{\ast},\quad\|\bar{v}^{\ast}\|_{C^{2}(\Omega^{\ast}\setminus\Omega^{\ast}_{R})}\leq C.
\end{align*}
From ({\bf{H1}})--({\bf{H2}}), we deduce that for $x\in\Omega_{R}^{\ast}$,
\begin{align}\label{LKT6.003}
|\partial_{x_{d}}(\bar{u}_{1}^{\alpha}-\bar{u}_{1}^{\ast\alpha})|\leq&
\begin{cases}
\frac{C\varepsilon}{|x'|^{m}(\varepsilon+|x'|^{m})},&\alpha=1,2,...,d,\\
\frac{C\varepsilon}{|x'|^{m-1}(\varepsilon+|x'|^{m})},&\alpha=d+1,...,\frac{d(d+1)}{2}.
\end{cases}
\end{align}
Using Proposition \ref{thm86} for $v_{1}^{\ast\alpha}$, it follows that for $x\in\Omega_{R}^{\ast}$,
\begin{align}\label{GZIJ}
\nabla v_{1}^{\ast\alpha}=&\nabla\bar{u}_{1}^{\ast\alpha}+O(1)
\begin{cases}
|x'|^{m-2},&\alpha=1,2,...,d,\\
1,&\alpha=d+1,...,\frac{d(d+1)}{2}.
\end{cases}
\end{align}
For $0<t<R$, denote
\begin{align*}
\mathcal{C}_{t}:=\left\{x\in\mathbb{R}^{d}\Big|\;2\min_{|x'|\leq t}h_{2}(x')\leq x_{d}\leq\varepsilon+2\max_{|x'|\leq t}h_{1}(x'),\;|x'|<t\right\}.
\end{align*}

Observe that for $\alpha=1,2,...,\frac{d(d+1)}{2}$, $v_{1}^{\alpha}-v_{1}^{\ast\alpha}$ verifies
\begin{align*}
\begin{cases}
\mathcal{L}_{\lambda,\mu}(v_{1}^{\alpha}-v_{1}^{\ast\alpha})=0,&\mathrm{in}\;\,D\setminus(\overline{D_{1}\cup D_{1}^{\ast}\cup D_{2}}),\\
v_{1}^{\alpha}-v_{1}^{\ast\alpha}=\psi_{\alpha}-v_{1}^{\ast\alpha},&\mathrm{on}\;\,\partial D_{1}\setminus D_{1}^{\ast},\\
v_{1}^{\alpha}-v_{1}^{\ast\alpha}=v_{1}^{\alpha}-\psi_{\alpha},&\mathrm{on}\;\,\partial D_{1}^{\ast}\setminus(D_{1}\cup\{0\}),\\
v_{1}^{\alpha}-v_{1}^{\ast\alpha}=0,&\mathrm{on}\;\,\partial D_{2}\cup\partial D.
\end{cases}
\end{align*}
First, following the standard boundary and interior estimates of elliptic systems, we get that for $x\in\partial D_{1}\setminus D_{1}^{\ast}$,
\begin{align}\label{LKT6.007}
|(v_{1}^{\alpha}-v_{1}^{\ast\alpha})(x',x_{d})|=|v_{1}^{\ast\alpha}(x',x_{d}-\varepsilon)-v_{1}^{\ast\alpha}(x',x_{d})|\leq C\varepsilon.
\end{align}
Making use of \eqref{Le2.025}, we deduce that for $x\in\partial D_{1}^{\ast}\setminus(D_{1}\cup\mathcal{C}_{\varepsilon^{\gamma}})$,
\begin{align}\label{LKT6.008}
|(v_{1}^{\alpha}-v_{1}^{\ast\alpha})(x',x_{d})|=|v_{1}^{\alpha}(x',x_{d})-v_{1}^{\alpha}(x',x_{d}+\varepsilon)|\leq C\varepsilon^{1-m\gamma},
\end{align}
where $0<\gamma<\frac{1}{2}$ is determined later. In light of \eqref{Le2.025} and \eqref{LKT6.003}--\eqref{GZIJ}, we obtain that for $x\in\Omega_{R}^{\ast}\cap\{|x'|=\varepsilon^{\gamma}\}$,
\begin{align*}
|\partial_{x_{d}}(v_{1}^{\alpha}-v_{1}^{\ast\alpha})|\leq&|\partial_{x_{d}}(v_{1}^{\alpha}-\bar{u}_{1}^{\alpha})|+|\partial_{x_{d}}(\bar{u}_{1}^{\alpha}-\bar{u}_{1}^{\ast\alpha}|+|\partial_{x_{d}}(v_{1}^{\ast\alpha}-\bar{u}_{1}^{\ast\alpha})|\notag\\
\leq&C\Big(\frac{1}{\varepsilon^{2m\gamma-1}}+1\Big),
\end{align*}
which, in combination with $v_{1}^{\alpha}-v_{1}^{\ast\alpha}=0$ on $\partial D_{2}$, leads to that
\begin{align}
|(v_{1}^{\alpha}-v_{1}^{\ast\alpha})(x',x_{d})|=&|(v_{1}^{\alpha}-v_{1}^{\ast\alpha})(x',x_{d})-(v_{1}^{\alpha}-v_{1}^{\ast\alpha})(x',h_{2}(x'))|\notag\\
\leq& C(\varepsilon^{1-m\gamma}+\varepsilon^{m\gamma}).\label{LKT6.009}
\end{align}
Let $\gamma=\frac{1}{2m}$. Then it follows from (\ref{LKT6.007})--(\ref{LKT6.009}) that
$$|v_{1}^{\alpha}-v_{1}^{\ast\alpha}|\leq C\varepsilon^{\frac{1}{2}},\quad\;\,\mathrm{on}\;\,\partial\big(D\setminus\big(\overline{D_{1}\cup D_{1}^{\ast}\cup D_{2}\cup\mathcal{C}_{\varepsilon^{\frac{1}{2m}}}}\big)\big).$$
This, together with the maximum principle for the Lam\'{e} system in \cite{MMN2007}, reads that for $\alpha=1,2,...,\frac{d(d+1)}{2}$,
\begin{align}\label{AST123}
|v_{1}^{\alpha}-v_{1}^{\ast\alpha}|\leq C\varepsilon^{\frac{1}{2}},\quad\;\,\mathrm{in}\;\,D\setminus\big(\overline{D_{1}\cup D_{1}^{\ast}\cup D_{2}\cup\mathcal{C}_{\varepsilon^{\frac{1}{2m}}}}\big).
\end{align}
For $\varepsilon^{\frac{1}{12m}}\leq|z'|\leq R$, by using the change of variable as follows:
\begin{align*}
\begin{cases}
x'-z'=|z'|^{m}y',\\
x_{d}=|z'|^{m}y_{d},
\end{cases}
\end{align*}
we rescale $\Omega_{|z'|+|z'|^{m}}\setminus\Omega_{|z'|}$ and $\Omega_{|z'|+|z'|^{m}}^{\ast}\setminus\Omega_{|z'|}^{\ast}$ into two approximate unit-size squares (or cylinders) $Q_{1}$ and $Q_{1}^{\ast}$, respectively. Denote
\begin{align*}
V_{1}^{\alpha}(y)=v_{1}^{\alpha}(z'+|z'|^{m}y',|z'|^{m}y_{d}),\quad\mathrm{in}\;Q_{1},
\end{align*}
and
\begin{align*}
V_{1}^{\ast\alpha}(y)=v_{1}^{\ast\alpha}(z'+|z'|^{m}y',|z'|^{m}y_{d}),\quad\mathrm{in}\;Q_{1}^{\ast}.
\end{align*}
Then we deduce from the standard elliptic estimates that
\begin{align*}
|\nabla^{2}V_{1}^{\alpha}|\leq C,\quad\mathrm{in}\;Q_{1},\quad\mathrm{and}\;|\nabla^{2}V_{1}^{\ast\alpha}|\leq C,\quad\mathrm{in}\;Q_{1}^{\ast}.
\end{align*}
Interpolating it with \eqref{AST123}, we get
\begin{align*}
|\nabla(V_{1}^{\alpha}-V_{1}^{\ast\alpha})|\leq C\varepsilon^{\frac{1}{2}(1-\frac{1}{2})}\leq C\varepsilon^{\frac{1}{4}}.
\end{align*}
Then back to $v_{1}^{\alpha}-v_{1}^{\ast\alpha}$, we obtain that for $\varepsilon^{\frac{1}{12m}}\leq|z'|\leq R$,
\begin{align*}
|\nabla(v_{1}^{\alpha}-v_{1}^{\ast\alpha})(x)|\leq C\varepsilon^{\frac{1}{4}}|z'|^{-m}\leq C\varepsilon^{\frac{1}{6}},\quad x\in\Omega^{\ast}_{|z'|+|z'|^{m}}\setminus\Omega_{|z'|}^{\ast}.
\end{align*}
Therefore, for $\alpha=1,2,...,\frac{d(d+1)}{2}$,
\begin{align}\label{con035}
|\nabla(v_{1}^{\alpha}-v_{1}^{\ast\alpha})|\leq C\varepsilon^{\frac{1}{6}},\quad\;\,\mathrm{in}\;\,D\setminus\big(\overline{D_{1}\cup D_{1}^{\ast}\cup D_{2}\cup\mathcal{C}_{\varepsilon^{\frac{1}{12m}}}}\big).
\end{align}

Denote $\bar{\gamma}=\frac{1}{12m}$. For $\alpha=1,2,...,\frac{d(d+1)}{2}$, we split $a_{11}^{\alpha\alpha}$ as follows:
\begin{align}\label{a1111}
a_{11}^{\alpha\alpha}=&\int_{\Omega\setminus\Omega_{R}}(\mathbb{C}^{0}e(v_{1}^{\alpha}),e(v_{1}^{\alpha}))+\int_{\Omega_{\varepsilon^{\bar{\gamma}}}}(\mathbb{C}^{0}e(v_{1}^{\alpha}),e(v_{1}^{\alpha}))+\int_{\Omega_{R}\setminus\Omega_{\varepsilon^{\bar{\gamma}}}}(\mathbb{C}^{0}e(v_{1}^{\alpha}),e(v_{1}^{\alpha}))\nonumber\\
=&:\mathrm{I}+\mathrm{II}+\mathrm{III}.
\end{align}
Due to the fact that $|D_{1}^{\ast}\setminus(D_{1}\cup\Omega_{R})|$ and $|D_{1}\setminus D_{1}^{\ast}|$ are less than $C\varepsilon$ and $|\nabla v_{1}^{\alpha}|$ keeps bounded in $(D_{1}^{\ast}\setminus(D_{1}\cup\Omega_{R}))\cup(D_{1}\setminus D_{1}^{\ast})$, then we obtain from \eqref{con035} that
\begin{align}\label{KKAA123}
\mathrm{I}=&\int_{D\setminus(D_{1}\cup D_{1}^{\ast}\cup D_{2}\cup\Omega_{R})}(\mathbb{C}^{0}e(v_{1}^{\alpha}),e(v_{1}^{\alpha}))+O(1)\varepsilon\notag\\
=&\int_{D\setminus(D_{1}\cup D_{1}^{\ast}\cup D_{2}\cup\Omega_{R})}\left((\mathbb{C}^{0}e(v_{1}^{\ast\alpha}),e(v_{1}^{\ast\alpha}))+2(\mathbb{C}^{0}e(v_{1}^{\alpha}-v_{1}^{\ast\alpha}),e(v_{1}^{\ast\alpha}))\right)\notag\\
&+\int_{D\setminus(D_{1}\cup D_{1}^{\ast}\cup D_{2}\cup\Omega_{R})}(\mathbb{C}^{0}e(v_{1}^{\alpha}-v_{1}^{\ast\alpha}),e(v_{1}^{\alpha}-v_{1}^{\ast\alpha}))+O(1)\varepsilon\notag\\
=&\int_{\Omega^{\ast}\setminus\Omega^{\ast}_{R}}(\mathbb{C}^{0}e(v_{1}^{\ast\alpha}),e(v_{1}^{\ast\alpha}))+O(1)\varepsilon^{\frac{1}{6}}.
\end{align}

Note that for $\alpha=1,2,...,d$, $\psi_{\alpha}=e_{\alpha}$; for $\alpha=d+1,...,\frac{d(d+1)}{2}$, there exist two indices $1\leq i_{\alpha}<j_{\alpha}\leq d$ such that
$\psi_{\alpha}=(0,...,0,x_{j_{\alpha}},0,...,0,-x_{i_{\alpha}},0,...,0)$. Especially when $d+1\leq\alpha\leq2d-1$, we know $i_{\alpha}=\alpha-d,\,j_{\alpha}=d$ and thus $\psi_{\alpha}=(0,...,0,x_{d},0,...,0,-x_{\alpha-d})$. Recalling the definitions of $\bar{u}_{1}^{\alpha}$ and $\mathbb{C}^{0}$, it follows from a straightforward calculation that
\begin{align*}
(\mathbb{C}^{0}e(\bar{u}_{1}^{\alpha}),e(\bar{u}_{1}^{\alpha}))=&
\begin{cases}
(\lambda+\mu)(\partial_{x_{\alpha}}\bar{v})^{2}+\mu\sum\limits^{d}_{i=1}(\partial_{x_{i}}\bar{v})^{2}+(\mathbb{C}^{0}e(\mathcal{F}_{\alpha}),e(\mathcal{F}_{\alpha}))\\
+2(\mathbb{C}^{0}e(\psi_{\alpha}\bar{v}),e(\mathcal{F}_{\alpha})),\;\,\alpha=1,2,...,d,\\
\mu(x_{i_{\alpha}}^{2}+x_{j_{\alpha}}^{2})\sum\limits^{d}_{k=1}(\partial_{x_{k}}\bar{v})^{2}+(\lambda+\mu)(x_{j_{\alpha}}\partial_{x_{i_{\alpha}}}\bar{v}-x_{i_{\alpha}}\partial_{x_{j_{\alpha}}}\bar{v})^{2}\notag\\
+2(\mathbb{C}^{0}e(\psi_{\alpha}\bar{v}),e(\mathcal{F}_{\alpha}))+(\mathbb{C}^{0}e(\mathcal{F}_{\alpha}),e(\mathcal{F}_{\alpha})),\;\,\alpha=d+1,...,\frac{d(d+1)}{2},
\end{cases}
\end{align*}
where the correction term $\mathcal{F}_{\alpha}$ is defined by \eqref{QLA001}. Then it follows from Corollary \ref{thm86} that
\begin{align}\label{con03365}
\mathrm{II}=&\int_{\Omega_{\varepsilon^{\bar{\gamma}}}}(\mathbb{C}^{0}e(\bar{u}_{1}^{\alpha}),e(\bar{u}_{1}^{\alpha}))+2\int_{\Omega_{\varepsilon^{\bar{\gamma}}}}(\mathbb{C}^{0}e(v_{1}^{\alpha}-\bar{u}_{1}^{\alpha}),e(\bar{u}_{1}^{\alpha}))\notag\\
&+\int_{\Omega_{\varepsilon^{\bar{\gamma}}}}(\mathbb{C}^{0}e(v_{1}^{\alpha}-\bar{u}_{1}^{\alpha}),e(v_{1}^{\alpha}-\bar{u}_{1}^{\alpha}))\notag\\
=&
\begin{cases}
\mathcal{L}_{d}^{\alpha}\int_{|x'|<\varepsilon^{\bar{\gamma}}}\frac{1}{\varepsilon+h_{1}(x')-h_{2}(x')}+O(1)\varepsilon^{\frac{d-1}{12m}},&\alpha=1,2,...,d,\\
\frac{\mathcal{L}_{d}^{\alpha}}{d-1}\int_{|x'|<\varepsilon^{\bar{\gamma}}}\frac{|x'|^{2}}{\varepsilon+h_{1}(x')-h_{2}(x')}+O(1)\varepsilon^{\frac{d-1}{12m}},&\alpha=d+1,...,\frac{d(d+1)}{2},
\end{cases}
\end{align}
where $\mathcal{L}_{d}^{\alpha}$, $\alpha=1,2,...,\frac{d(d+1)}{2}$ are given by \eqref{AZM001}--\eqref{AZ110}.

With regard to the last term $\mathrm{III}$ in \eqref{a1111}, we further decompose it as follows:
\begin{align*}
\mathrm{III}_{1}=&\int_{\Omega^{\ast}_{R}\setminus\Omega^{\ast}_{\varepsilon^{\bar{\gamma}}}}(\mathbb{C}^{0}e(v_{1}^{\alpha}-v_{1}^{\ast\alpha}),e(v_{1}^{\alpha}-v_{1}^{\ast\alpha}))+2\int_{\Omega^{\ast}_{R}\setminus\Omega^{\ast}_{\varepsilon^{\bar{\gamma}}}}(\mathbb{C}^{0}e(v_{1}^{\alpha}-v_{1}^{\ast\alpha}),e(v_{1}^{\ast\alpha})),\\
\mathrm{III}_{2}=&\int_{(\Omega_{R}\setminus\Omega_{\varepsilon^{\bar{\gamma}}})\setminus(\Omega^{\ast}_{R}\setminus\Omega^{\ast}_{\varepsilon^{\bar{\gamma}}})}(\mathbb{C}^{0}e(v_{1}^{\alpha}),e(v_{1}^{\alpha})),\\
\mathrm{III}_{3}=&\int_{\Omega^{\ast}_{R}\setminus\Omega^{\ast}_{\varepsilon^{\bar{\gamma}}}}(\mathbb{C}^{0}e(v_{1}^{\ast\alpha}),e(v_{1}^{\ast\alpha})).
\end{align*}
A consequence of \eqref{GZIJ} and \eqref{con035} gives
\begin{align}\label{con036666}
|\mathrm{III}_{1}|\leq\,C\varepsilon^{\frac{1}{6}}.
\end{align}
Since the thickness of $(\Omega_{R}\setminus\Omega_{\varepsilon^{\bar{\gamma}}})\setminus(\Omega^{\ast}_{R}\setminus\Omega^{\ast}_{\varepsilon^{\bar{\gamma}}})$ is $\varepsilon$, then it follows from \eqref{Le2.025} that for $\alpha=1,2,...,d,$
\begin{align}\label{con0333355}
|\mathrm{III}_{2}|\leq\,C\varepsilon\int_{\varepsilon^{\bar{\gamma}}<|x'|<R}\frac{dx'}{|x'|^{2m}}\leq& C
\begin{cases}
\varepsilon,&m<\frac{d-1}{2},\\
\varepsilon|\ln\varepsilon|,&m=\frac{d-1}{2},\\
\varepsilon^{\frac{10m+d-1}{12m}},&m>\frac{d-1}{2},
\end{cases}
\end{align}
and, for $\alpha=d+1,...,\frac{d(d+1)}{2}$,
\begin{align}\label{con0333355KL}
|\mathrm{III}_{2}|\leq\,C\varepsilon\int_{\varepsilon^{\bar{\gamma}}<|x'|<R}\frac{dx'}{|x'|^{2m-2}}\leq& C
\begin{cases}
\varepsilon,&m<\frac{d+1}{2},\\
\varepsilon|\ln\varepsilon|,&m=\frac{d+1}{2},\\
\varepsilon^{\frac{10m+d+1}{12m}},&m>\frac{d+1}{2}.
\end{cases}
\end{align}

For $\mathrm{III}_{3}$, using \eqref{GZIJ} again, we derive
\begin{align*}
\mathrm{III}_{3}=&\int_{\Omega_{R}^{\ast}\setminus\Omega^{\ast}_{\varepsilon^{\bar{\gamma}}}}(\mathbb{C}^{0}e(\bar{u}_{1}^{\ast\alpha}),e(\bar{u}_{1}^{\ast\alpha}))+2\int_{\Omega_{R}^{\ast}\setminus\Omega^{\ast}_{\varepsilon^{\bar{\gamma}}}}(\mathbb{C}^{0}e(v_{1}^{\ast\alpha}-\bar{u}_{1}^{\ast\alpha}),e(\bar{u}_{1}^{\ast\alpha}))\notag\\
&+\int_{\Omega_{R}^{\ast}\setminus\Omega^{\ast}_{\varepsilon^{\bar{\gamma}}}}(\mathbb{C}^{0}e(v_{1}^{\ast\alpha}-\bar{u}_{1}^{\ast\alpha}),e(v_{1}^{\ast\alpha}-\bar{u}_{1}^{\ast\alpha}))\notag\\
=&
\begin{cases}
\mathcal{L}_{d}^{\alpha}\int_{\varepsilon^{\bar{\gamma}}<|x'|<R}\frac{dx'}{h_{1}-h_{2}}-\int_{\Omega^{\ast}\setminus\Omega^{\ast}_{R}}(\mathbb{C}^{0}e(v_{1}^{\ast\alpha}),e(v_{1}^{\ast\alpha}))\notag\\
+M_{d}^{\ast\alpha}+O(1)\varepsilon^{\min\{\frac{1}{6},\frac{d-1}{12m}\}},\quad\alpha=1,2,...,d,\notag\\
\frac{\mathcal{L}_{d}^{\alpha}}{d-1}\int_{\varepsilon^{\bar{\gamma}}<|x'|<R}\frac{|x'|^{2}}{h_{1}(x')-h_{2}(x')}dx'-\int_{\Omega^{\ast}\setminus\Omega^{\ast}_{R}}(\mathbb{C}^{0}e(v_{1}^{\ast\alpha}),e(v_{1}^{\ast\alpha}))\notag\\
+M_{d}^{\ast\alpha}+O(1)\varepsilon^{\min\{\frac{1}{6},\frac{d-1}{12m}\}},\quad\alpha=d+1,...,\frac{d(d+1)}{2},
\end{cases}
\end{align*}
where
\begin{align}\label{FT001}
M_{d}^{\ast\alpha}=&\int_{\Omega^{\ast}\setminus\Omega^{\ast}_{R}}(\mathbb{C}^{0}e(\bar{u}_{1}^{\ast\alpha}),e(\bar{u}_{1}^{\ast\alpha}))+\int_{\Omega_{R}^{\ast}}(\mathbb{C}^{0}e(v_{1}^{\ast\alpha}-\bar{u}_{1}^{\ast\alpha}),e(v_{1}^{\ast\alpha}-\bar{u}_{1}^{\ast\alpha}))\notag\\
&+\int_{\Omega_{R}^{\ast}}\big[2(\mathbb{C}^{0}e(v_{1}^{\ast\alpha}-\bar{u}_{1}^{\ast\alpha}),e(\bar{u}_{1}^{\ast\alpha}))+2(\mathbb{C}^{0}e(\psi_{\alpha}\bar{v}^{\ast}),e(\mathcal{F}_{\alpha}^{\ast}))+(\mathbb{C}^{0}e(\mathcal{F}_{\alpha}^{\ast}),e(\mathcal{F}_{\alpha}^{\ast}))\big]\notag\\
&+
\begin{cases}
\int_{\Omega_{R}^{\ast}}(\lambda+\mu)(\partial_{x_{\alpha}}\bar{v}^{\ast})^{2}+\mu\sum\limits^{d-1}_{i=1}(\partial_{x_{i}}\bar{v}^{\ast})^{2},\quad\alpha=1,...,d-1,\\
\int_{\Omega_{R}^{\ast}}\mu\sum\limits^{d-1}_{i=1}(\partial_{x_{i}}\bar{v}^{\ast})^{2},\quad\alpha=d,\\
\int_{\Omega_{R}^{\ast}}\big[\mu(x_{\alpha-d}^{2}+x_{d}^{2})\sum\limits^{d-1}_{k=1}(\partial_{x_{k}}\bar{v}^{\ast})^{2}+\mu(x_{d}\partial_{x_{d}}\bar{v}^{\ast})^{2}+(\lambda+\mu)(x_{d}\partial_{x_{\alpha-d}}\bar{v}^{\ast})^{2}\\
\quad\quad\;-2(\lambda+\mu)x_{\alpha-d}x_{d}\partial_{x_{\alpha-d}}\bar{v}^{\ast}\partial_{x_{d}}\bar{v}^{\ast}\big],\;\quad\alpha=d+1,...,2d-1,\\
\int_{\Omega_{R}^{\ast}}\big[\mu(x_{i_{\alpha}}^{2}+x_{j_{\alpha}}^{2})\sum\limits^{d-1}_{k=1}(\partial_{x_{k}}\bar{v}^{\ast})^{2}\\
\quad\quad\;+(\lambda+\mu)(x_{j_{\alpha}}\partial_{x_{i_{\alpha}}}\bar{v}^{\ast}-x_{i_{\alpha}}\partial_{x_{j_{\alpha}}}\bar{v}^{\ast})^{2}\big],\;\quad\alpha=2d,...,\frac{d(d+1)}{2},\;d\geq3.
\end{cases}
\end{align}
This, in combination with \eqref{KKAA123}--\eqref{con0333355KL}, gives that

$(\rm{i})$ for $\alpha=1,2,...,d$, we have
\begin{align}\label{KATZ001}
a_{11}^{\alpha\alpha}=&\mathcal{L}_{d}^{\alpha}\left(\int_{\varepsilon^{\bar{\gamma}}<|x'|<R}\frac{dx'}{h_{1}(x')-h_{2}(x')}+\int_{|x'|<\varepsilon^{\bar{\gamma}}}\frac{dx'}{\varepsilon+h_{1}(x')-h_{2}(x')}\right)\notag\\
&+M_{d}^{\ast\alpha}+O(1)\varepsilon^{\min\{\frac{1}{6},\frac{d-1}{12m}\}}.
\end{align}
On one hand, if $m\geq d-1$, then
\begin{align*}
&\int_{\varepsilon^{\bar{\gamma}}<|x'|<R}\frac{1}{h_{1}-h_{2}}+\int_{|x'|<\varepsilon^{\bar{\gamma}}}\frac{1}{\varepsilon+h_{1}-h_{2}}\notag\\
=&\int_{|x'|<R}\frac{1}{\varepsilon+h_{1}-h_{2}}+\int_{\varepsilon^{\bar{\gamma}}<|x'|<R}\frac{\varepsilon}{(h_{1}-h_{2})(\varepsilon+h_{1}-h_{2})}\notag\\
=&\int_{|x'|<R}\frac{1}{\varepsilon+h_{1}-h_{2}}+O(1)
\begin{cases}
\varepsilon,&m<\frac{d-1}{2},\\
\varepsilon|\ln\varepsilon|,&m=\frac{d-1}{2},\\
\varepsilon^{\frac{10m+d-1}{12m}},&m>\frac{d-1}{2},
\end{cases}
\end{align*}
which yields that
\begin{align}\label{LNQ001}
a_{11}^{\alpha\alpha}\lesssim&\mathcal{L}_{d}^{\alpha}\int_{|x'|<R}\frac{1}{\varepsilon+h_{1}-h_{2}}\lesssim\mathcal{L}_{d}^{\alpha}\int_{|x'|<R}\frac{1}{\varepsilon+\tau_{1}|x'|^{m}}\notag\\
\lesssim&\mathcal{L}_{d}^{\alpha}\int_{0}^{R}\frac{s^{d-2}}{\varepsilon+\tau_{1}s^{m}}\lesssim\frac{\mathcal{L}_{d}^{\alpha}}{\tau_{1}^{\frac{d-1}{m}}}\rho_{0}(d,m;\varepsilon),
\end{align}
and
\begin{align}\label{LAU001}
a_{11}^{\alpha\alpha}\gtrsim&\mathcal{L}_{d}^{\alpha}\int_{|x'|<R}\frac{1}{\varepsilon+\tau_{2}|x'|^{m}}\gtrsim\frac{\mathcal{L}_{d}^{\alpha}}{\tau_{2}^{\frac{d-1}{m}}}\rho_{0}(d,m;\varepsilon).
\end{align}
On the other hand, if $m<d-1$, then
\begin{align}\label{LNQ002}
a_{11}^{\alpha\alpha}=&\mathcal{L}_{d}^{\alpha}\left(\int_{|x'|<R}\frac{1}{h_{1}-h_{2}}-\int_{|x'|<\varepsilon^{\bar{\gamma}}}\frac{\varepsilon}{(h_{1}-h_{2})(\varepsilon+h_{1}-h_{2})}\right)\notag\\
&+M_{d}^{\ast\alpha}+O(1)\varepsilon^{\min\{\frac{1}{6},\frac{d-1}{12m}\}}\notag\\
=&\mathcal{L}_{d}^{\alpha}\int_{\Omega_{R}^{\ast}}|\partial_{x_{d}}\bar{u}_{1}^{\ast\alpha}|^{2}+M_{d}^{\ast\alpha}+O(1)\varepsilon^{\min\{\frac{1}{6},\frac{d-1-m}{12m}\}}\notag\\
=&a_{11}^{\ast\alpha\alpha}+O(1)\varepsilon^{\min\{\frac{1}{6},\frac{d-1-m}{12m}\}};
\end{align}

$(\rm{ii})$ for $\alpha=d+1,...,\frac{d(d+1)}{2}$, on one hand, if $m\geq d+1$, then
\begin{align}\label{QZH001}
a_{11}^{\alpha\alpha}=&\frac{\mathcal{L}_{d}^{\alpha}}{d-1}\left(\int_{\varepsilon^{\bar{\gamma}}<|x'|<R}\frac{|x'|^{2}}{h_{1}(x')-h_{2}(x')}+\int_{|x'|<\varepsilon^{\bar{\gamma}}}\frac{|x'|^{2}}{\varepsilon+h_{1}(x')-h_{2}(x')}\right)\nonumber\\
&+M_{d}^{\ast\alpha}+O(1)\varepsilon^{\min\{\frac{1}{6},\frac{d-1}{12 m}\}}.
\end{align}
In light of the fact that
\begin{align*}
&\int_{\varepsilon^{\bar{\gamma}}<|x'|<R}\frac{|x'|^{2}}{h_{1}-h_{2}}+\int_{|x'|<\varepsilon^{\bar{\gamma}}}\frac{|x'|^{2}}{\varepsilon+h_{1}-h_{2}}\notag\\
=&\int_{|x'|<R}\frac{|x'|^{2}}{\varepsilon+h_{1}-h_{2}}+\int_{\varepsilon^{\bar{\gamma}}<|x'|<R}\frac{\varepsilon|x'|^{2}}{(h_{1}-h_{2})(\varepsilon+h_{1}-h_{2})}\notag\\
=&\int_{|x'|<R}\frac{|x'|^{2}}{\varepsilon+h_{1}-h_{2}}+O(1)
\begin{cases}
\varepsilon,&m<\frac{d+1}{2},\\
\varepsilon|\ln\varepsilon|,&m=\frac{d+1}{2},\\
\varepsilon^{\frac{10m+d+1}{12m}},&m>\frac{d+1}{2},
\end{cases}
\end{align*}
we derive
\begin{align}\label{QN001}
a_{11}^{\alpha\alpha}\lesssim&\frac{\mathcal{L}_{d}^{\alpha}}{d-1}\int_{|x'|<R}\frac{|x'|^{2}}{\varepsilon+h_{1}(x')-h_{2}(x')}\lesssim\frac{\mathcal{L}_{d}^{\alpha}}{d-1}\int_{|x'|<R}\frac{|x'|^{2}}{\varepsilon+\tau_{1}|x'|^{m}}\notag\\
\lesssim&\mathcal{L}_{d}^{\alpha}\int_{0}^{R}\frac{s^{d}}{\varepsilon+\tau_{1}s^{m}}\lesssim\frac{\mathcal{L}_{d}^{\alpha}}{\tau_{1}^{\frac{d+1}{m}}}\rho_{2}(d,m;\varepsilon),
\end{align}
and
\begin{align}\label{QN002}
a_{11}^{\alpha\alpha}\gtrsim&\frac{\mathcal{L}_{d}^{\alpha}}{d-1}\int_{|x'|<R}\frac{|x'|^{2}}{\varepsilon+\tau_{2}|x'|^{m}}\gtrsim\frac{\mathcal{L}_{d}^{\alpha}}{\tau_{2}^{\frac{d+1}{m}}}\rho_{2}(d,m;\varepsilon).
\end{align}
On the other hand, if $m<d+1$, then for $\alpha=d+1,...,2d-1$,
\begin{align}\label{QZH002}
a_{11}^{\alpha\alpha}=&\mathcal{L}_{d}^{\alpha}\left(\int_{\varepsilon^{\bar{\gamma}}<|x'|<R}\frac{x_{\alpha-d}^{2}}{h_{1}-h_{2}}+\int_{|x'|<\varepsilon^{\bar{\gamma}}}\frac{x_{\alpha-d}^{2}}{\varepsilon+h_{1}-h_{2}}\right)+M_{d}^{\ast\alpha}+O(1)\varepsilon^{\min\{\frac{1}{6},\frac{d-1}{12 m}\}}\notag\\
=&\mathcal{L}_{d}^{\alpha}\left(\int_{|x'|<R}\frac{x_{\alpha-d}^{2}}{h_{1}-h_{2}}-\int_{|x'|<\varepsilon^{\bar{\gamma}}}\frac{\varepsilon x_{\alpha-d}^{2}}{(h_{1}-h_{2})(\varepsilon+h_{1}-h_{2})}\right)\notag\\
&+M_{d}^{\ast\alpha}+O(1)\varepsilon^{\min\{\frac{1}{6},\frac{d-1}{12 m}\}}\notag\\
=&\mathcal{L}_{d}^{\alpha}\int_{\Omega_{R}^{\ast}}|x_{\alpha-d}\partial_{x_{d}}\bar{v}^{\ast}|^{2}+M_{d}^{\ast\alpha}+O(1)\varepsilon^{\min\{\frac{1}{6},\frac{d+1-m}{12 m}\}}\notag\\
=&a_{11}^{\ast\alpha\alpha}+O(1)\varepsilon^{\min\{\frac{1}{6},\frac{d+1-m}{12 m}\}},
\end{align}
while, for $\alpha=2d,...,\frac{d(d+1)}{2},\,d\geq3$,
\begin{align*}
a_{11}^{\alpha\alpha}=&\frac{\mathcal{L}_{d}^{\alpha}}{2}\left(\int_{\varepsilon^{\bar{\gamma}}<|x'|<R}\frac{x_{i_{\alpha}}^{2}+x_{j_{\alpha}}^{2}}{h_{1}-h_{2}}+\int_{|x'|<\varepsilon^{\bar{\gamma}}}\frac{x_{i_{\alpha}}^{2}+x_{j_{\alpha}}^{2}}{\varepsilon+h_{1}-h_{2}}\right)\notag\\
&+M_{d}^{\ast\alpha}+O(1)\varepsilon^{\min\{\frac{1}{6},\frac{d-1}{12m}\}}\\
=&\frac{\mathcal{L}_{d}^{\alpha}}{2}\int_{|x'|<R}\frac{x_{i_{\alpha}}^{2}+x_{j_{\alpha}}^{2}}{h_{1}-h_{2}}+M_{d}^{\ast\alpha}+O(1)\varepsilon^{\min\{\frac{1}{6},\frac{d+1-m}{12m}\}}\\
=&\frac{\mathcal{L}_{d}^{\alpha}}{2}\int_{\Omega_{R}^{\ast}}(x_{i_{\alpha}}^{2}+x_{j_{\alpha}}^{2})|\partial_{x_{d}}\bar{v}^{\ast}|^{2}+M_{d}^{\ast\alpha}+O(1)\varepsilon^{\min\{\frac{1}{6},\frac{d+1-m}{12m}\}}\\
=&a_{11}^{\ast\alpha\alpha}+O(1)\varepsilon^{\min\{\frac{1}{6},\frac{d+1-m}{12 m}\}}.
\end{align*}
This, in combination with \eqref{LNQ001}--\eqref{LNQ002} and \eqref{QN001}--\eqref{QZH002}, yields that \eqref{LMC}--\eqref{LMC1} hold.

\noindent{\bf Step 2.} Proofs of \eqref{M001}--\eqref{M005}. Because of symmetry, it is enough to consider the case of $\alpha<\beta$. Similarly as above, for $\alpha,\beta=1,2,...,\frac{d(d+1)}{2}$, $\alpha\neq\beta,$ we decompose $a_{11}^{\alpha\beta}$ as follows:
\begin{align*}
a_{11}^{\alpha\beta}=&\int_{\Omega\setminus\Omega_{R}}(\mathbb{C}^{0}e(v_{1}^{\alpha}),e(v_{1}^{\beta}))+\int_{\Omega_{R}\setminus\Omega_{\varepsilon^{\bar{\gamma}}}}(\mathbb{C}^{0}e(v_{1}^{\alpha}),e(v_{1}^{\beta}))+\int_{\Omega_{\varepsilon^{\bar{\gamma}}}}(\mathbb{C}^{0}e(v_{1}^{\alpha}),e(v_{1}^{\beta}))\nonumber\\
=&:\mathrm{I}+\mathrm{II}+\mathrm{III},
\end{align*}
where $\bar{\gamma}=\frac{1}{12m}$. To begin with, following the same argument as in \eqref{KKAA123}, we obtain
\begin{align}\label{KKAA1233333}
\mathrm{I}=&\int_{D\setminus(D_{1}\cup D_{1}^{\ast}\cup D_{2}\cup\Omega_{R})}(\mathbb{C}^{0}e(v_{1}^{\alpha}),e(v_{1}^{\beta}))+O(1)\varepsilon\notag\\
=&\int_{D\setminus(D_{1}\cup D_{1}^{\ast}\cup D_{2}\cup\Omega_{R})}\big[(\mathbb{C}^{0}e(v_{1}^{\ast\alpha}),e(v_{1}^{\ast\beta}))+(\mathbb{C}^{0}e(v_{1}^{\alpha}-v_{1}^{\ast\alpha}),e(v_{1}^{\beta}-v_{1}^{\ast\beta}))\big]\notag\\
&+\int_{D\setminus(D_{1}\cup D_{1}^{\ast}\cup D_{2}\cup\Omega_{R})}\big[(\mathbb{C}^{0}e(v_{1}^{\ast\alpha}),e(v_{1}^{\beta}-v_{1}^{\ast\beta}))+(\mathbb{C}^{0}e(v_{1}^{\alpha}-v_{1}^{\ast\alpha}),e(v_{1}^{\ast\beta}))\big]\notag\\
=&\int_{\Omega^{\ast}\setminus\Omega^{\ast}_{R}}(\mathbb{C}^{0}e(v_{1}^{\ast\alpha}),e(v_{1}^{\ast\beta}))+O(1)\varepsilon^{\frac{1}{6}}.
\end{align}

For the second term $\mathrm{II}$, we split it as follows:
\begin{align*}
\mathrm{II}_{1}=&\int_{(\Omega_{R}\setminus\Omega_{\varepsilon^{\bar{\gamma}}})\setminus(\Omega^{\ast}_{R}\setminus\Omega^{\ast}_{\varepsilon^{\bar{\gamma}}})}(\mathbb{C}^{0}e(v_{1}^{\alpha}),e(v_{1}^{\beta}))+\int_{\Omega^{\ast}_{R}\setminus\Omega^{\ast}_{\varepsilon^{\bar{\gamma}}}}(\mathbb{C}^{0}e(v_{1}^{\ast\alpha}),e(v_{1}^{\beta}-v_{1}^{\ast\beta}))\notag\\
&+\int_{\Omega^{\ast}_{R}\setminus\Omega^{\ast}_{\varepsilon^{\bar{\gamma}}}}(\mathbb{C}^{0}e(v_{1}^{\alpha}-v_{1}^{\ast\alpha}),e(v_{1}^{\ast\beta}))+\int_{\Omega^{\ast}_{R}\setminus\Omega^{\ast}_{\varepsilon^{\bar{\gamma}}}}(\mathbb{C}^{0}e(v_{1}^{\alpha}-v_{1}^{\ast\alpha}),e(v_{1}^{\beta}-v_{1}^{\ast\beta})),\\
\mathrm{II}_{2}=&\int_{\Omega^{\ast}_{R}\setminus\Omega^{\ast}_{\varepsilon^{\bar{\gamma}}}}(\mathbb{C}^{0}e(v_{1}^{\ast\alpha}),e(v_{1}^{\ast\beta})).
\end{align*}
In view of the fact that the thickness of $(\Omega_{R}\setminus\Omega_{\varepsilon^{\bar{\gamma}}})\setminus(\Omega^{\ast}_{R}\setminus\Omega^{\ast}_{\varepsilon^{\bar{\gamma}}})$ is $\varepsilon$, it follows from \eqref{Le2.025} and \eqref{con035} that
\begin{align}\label{con036}
\mathrm{II}_{1}=O(1)\varepsilon^{\frac{1}{6}}.
\end{align}
With regard to $\mathrm{II}_{2}$, we further discuss as follows:

{\bf case 1.} if $d=2$, $\alpha=1,\beta=2$, then it follows from \eqref{GZIJ} that
\begin{align}\label{PAHN}
\mathrm{II}_{2}=&\int_{\Omega_{R}^{\ast}\setminus\Omega^{\ast}_{\varepsilon^{\bar{\gamma}}}}(\mathbb{C}^{0}e(\bar{u}_{1}^{\ast1}),e(\bar{u}_{1}^{\ast 2}))+\int_{\Omega_{R}^{\ast}\setminus\Omega^{\ast}_{\varepsilon^{\bar{\gamma}}}}(\mathbb{C}^{0}e(v_{1}^{\ast1}-\bar{u}_{1}^{\ast1}),e(v_{1}^{\ast 2}-\bar{u}_{1}^{\ast2}))\notag\\
&+\int_{\Omega_{R}^{\ast}\setminus\Omega^{\ast}_{\varepsilon^{\bar{\gamma}}}}(\mathbb{C}^{0}e(v_{1}^{\ast1}-\bar{u}_{1}^{\ast1}),e(\bar{u}_{1}^{\ast2}))+\int_{\Omega_{R}^{\ast}\setminus\Omega^{\ast}_{\varepsilon^{\bar{\gamma}}}}(\mathbb{C}^{0}e(\bar{u}_{1}^{\ast1}),e(v_{1}^{\ast2}-\bar{u}_{1}^{\ast2}))\notag\\
=&\int_{\Omega_{R}^{\ast}\setminus\Omega^{\ast}_{\varepsilon^{\bar{\gamma}}}}(\lambda+\mu)\partial_{x_{1}}\bar{v}^{\ast}\partial_{x_{2}}\bar{v}^{\ast}+O(1)\notag\\
=&O(1)|\ln\varepsilon|;
\end{align}

{\bf case 2.} if $d\geq3,\,\alpha,\beta=1,2,...,d,\,\alpha<\beta$, if $d\geq2,\alpha=1,2,...,d,\,\beta=d+1,...,\frac{d(d+1)}{2},\alpha<\beta$, or if $d\geq3,\,\alpha,\beta=d+1,...,\frac{d(d+1)}{2},\,\alpha<\beta$, then we deduce from \eqref{GZIJ} again that
\begin{align}\label{QKL}
&\mathrm{II}_{2}-\int_{\Omega^{\ast}_{R}}(\mathbb{C}^{0}e(v_{1}^{\ast\alpha}),e(v_{1}^{\ast\beta}))=-\int_{\Omega^{\ast}_{\varepsilon^{\bar{\gamma}}}}(\mathbb{C}^{0}e(v_{1}^{\ast\alpha}),e(v_{1}^{\ast\beta}))\notag\\
=&\int_{\Omega^{\ast}_{\varepsilon^{\bar{\gamma}}}}\big[(\mathbb{C}^{0}e(\bar{u}_{1}^{\ast\alpha}),e(\bar{u}_{1}^{\ast\beta}))+(\mathbb{C}^{0}e(v_{1}^{\ast\alpha}-\bar{u}_{1}^{\ast\alpha}),e(v_{1}^{\ast\beta}-\bar{u}_{1}^{\ast\beta}))\notag\\
&\quad\quad\quad+(\mathbb{C}^{0}e(\bar{u}_{1}^{\ast\alpha}),e(v_{1}^{\ast\beta}-\bar{u}_{1}^{\ast\beta}))+(\mathbb{C}^{0}e(v_{1}^{\ast\alpha}-\bar{u}_{1}^{\ast\alpha}),e(\bar{u}_{1}^{\ast\beta}))\big]\notag\\
=&\int_{\Omega^{\ast}_{\varepsilon^{\bar{\gamma}}}}(\mathbb{C}^{0}e(\psi_{\alpha}\bar{v}^{\ast}),e(\psi_{\beta}\bar{v}^{\ast}))\notag\\
&+O(1)
\begin{cases}
\varepsilon^{(d-1)\bar{\gamma}},&\alpha=1,2,...,d,\,\beta=1,2,...,\frac{d(d+1)}{2},\,\alpha<\beta,\\
\varepsilon^{d\bar{\gamma}},&\alpha,\beta=d+1,...,\frac{d(d+1)}{2},\,\alpha<\beta,
\end{cases}
\end{align}
where we used the fact that
\begin{align*}
(\mathbb{C}^{0}e(\bar{u}_{1}^{\ast\alpha}),e(\bar{u}_{1}^{\ast\beta}))=&(\mathbb{C}^{0}e(\psi_{\alpha}\bar{v}^{\ast}),e(\psi_{\beta}\bar{v}^{\ast}))+(\mathbb{C}^{0}e(\mathcal{F}_{\alpha}^{\ast}),e(\psi_{\beta}\bar{v}^{\ast}))\notag\\
&+(\mathbb{C}^{0}e(\psi_{\alpha}\bar{v}^{\ast}),e(\mathcal{F}_{\beta}^{\ast}))+(\mathbb{C}^{0}e(\mathcal{F}_{\alpha}^{\ast}),e(\mathcal{F}_{\beta}^{\ast})).
\end{align*}
Let $E_{\alpha\beta}(\bar{v}^{\ast})=(\mathbb{C}^{0}e(\psi_{\alpha}\bar{v}^{\ast}),e(\psi_{\beta}\bar{v}^{\ast}))$. Straightforward computations give that

$(1)$ for $\alpha,\beta=1,2,...,d,$ $\alpha<\beta$,
\begin{align}\label{ZH0000}
E_{\alpha\beta}(\bar{v}^{\ast})=(\lambda+\mu)\partial_{x_{\alpha}}\bar{v}^{\ast}\partial_{x_{\beta}}\bar{v}^{\ast};
\end{align}

$(2)$ for $\alpha=1,2,...,d$, $\beta=d+1,...,\frac{d(d+1)}{2}$, there exist two indices $1\leq i_{\beta}<j_{\beta}\leq d$ such that
$\psi_{\beta}\bar{v}^{\ast}=(0,...,0,x_{j_{\beta}}\bar{v}^{\ast},0,...,0,-x_{i_{\beta}}\bar{v}^{\ast},0,...,0)$. If $i_{\beta}\neq\alpha,\,j_{\beta}\neq\alpha$, then
\begin{align}\label{ZH000}
E_{\alpha\beta}(\bar{v}^{\ast})=(\lambda+\mu)\partial_{x_{\alpha}}\bar{v}^{\ast}(x_{j_{\beta}}\partial_{i_{\beta}}\bar{v}^{\ast}-x_{i_{\beta}}\partial_{x_{j_{\beta}}}\bar{v}^{\ast}),
\end{align}
and if $i_{\beta}=\alpha,\,j_{\beta}\neq\alpha$, then
\begin{align}\label{ZH001}
E_{\alpha\beta}(\bar{v}^{\ast})=\mu x_{j_{\beta}}\sum^{d}_{k=1}(\partial_{x_{k}}\bar{v}^{\ast})^{2}+(\lambda+\mu)\partial_{x_{\alpha}}\bar{v}^{\ast}(x_{j_{\beta}}\partial_{i_{\beta}}\bar{v}^{\ast}-x_{i_{\beta}}\partial_{x_{j_{\beta}}}\bar{v}^{\ast}),
\end{align}
and if $i_{\beta}\neq\alpha,\,j_{\beta}=\alpha$, then
\begin{align}\label{ZH002}
E_{\alpha\beta}(\bar{v}^{\ast})=&-\mu x_{i_{\beta}}\sum^{d}_{k=1}(\partial_{x_{k}}\bar{v}^{\ast})^{2}+(\lambda+\mu)\partial_{x_{\alpha}}\bar{v}^{\ast}(x_{j_{\beta}}\partial_{i_{\beta}}\bar{v}^{\ast}-x_{i_{\beta}}\partial_{x_{j_{\beta}}}\bar{v}^{\ast});
\end{align}

$(3)$ for $\alpha,\beta=d+1,...,\frac{d(d+1)}{2}$, $\alpha<\beta$, there exist four indices $1\leq i_{\alpha}<j_{\alpha}\leq d$ and $1\leq i_{\beta}<j_{\beta}\leq d$ such that $\psi_{\alpha}\bar{v}^{\ast}=(0,...,0,x_{j_{\alpha}}\bar{v}^{\ast},0,...,0,-x_{i_{\alpha}}\bar{v}^{\ast},0,...,0)$
and
$\psi_{\beta}\bar{v}^{\ast}=(0,...,0,x_{j_{\beta}}\bar{v}^{\ast},0,...,0,-x_{i_{\beta}}\bar{v}^{\ast},0,...,0)$. In light of $\alpha<\beta$, we know that $j_{\beta}\leq j_{\alpha}$. If $i_{\alpha}\neq i_{\beta},\,j_{\alpha}\neq j_{\beta},\,i_{\alpha}\neq j_{\beta}$, then
\begin{align}\label{ZH003}
E_{\alpha\beta}(\bar{v}^{\ast})=(\lambda+\mu)(x_{j_{\alpha}}\partial_{x_{i_{\alpha}}}\bar{v}^{\ast}-x_{i_{\alpha}}\partial_{x_{j_{\alpha}}}\bar{v}^{\ast})(x_{j_{\beta}}\partial_{x_{i_{\beta}}}\bar{v}^{\ast}-x_{i_{\beta}}\partial_{x_{j_{\beta}}}\bar{v}^{\ast}),
\end{align}
and if $i_{\alpha}=i_{\beta},\,j_{\alpha}\neq j_{\beta}$, then
\begin{align}\label{ZH004}
E_{\alpha\beta}(\bar{v}^{\ast})=&\mu x_{j_{\alpha}}x_{j_{\beta}}\sum^{d}_{k=1}(\partial_{x_{k}}\bar{v}^{\ast})^{2}\notag\\
&+(\lambda+\mu)(x_{j_{\alpha}}\partial_{x_{i_{\alpha}}}\bar{v}^{\ast}-x_{i_{\alpha}}\partial_{x_{j_{\alpha}}}\bar{v}^{\ast})(x_{j_{\beta}}\partial_{x_{i_{\beta}}}\bar{v}^{\ast}-x_{i_{\beta}}\partial_{x_{j_{\beta}}}\bar{v}^{\ast}),
\end{align}
and if $i_{\alpha}\neq i_{\beta},\,j_{\alpha}=j_{\beta}$, then
\begin{align}\label{ZH005}
E_{\alpha\beta}(\bar{v}^{\ast})=&\mu x_{i_{\alpha}}x_{i_{\beta}}\sum^{d}_{k=1}(\partial_{x_{k}}\bar{v}^{\ast})^{2}\notag\\
&+(\lambda+\mu)(x_{j_{\alpha}}\partial_{x_{i_{\alpha}}}\bar{v}^{\ast}-x_{i_{\alpha}}\partial_{x_{j_{\alpha}}}\bar{v}^{\ast})(x_{j_{\beta}}\partial_{x_{i_{\beta}}}\bar{v}^{\ast}-x_{i_{\beta}}\partial_{x_{j_{\beta}}}\bar{v}^{\ast}),
\end{align}
and if $i_{\beta}<j_{\beta}=i_{\alpha}<j_{\alpha}$, then
\begin{align}\label{ZH006}
E_{\alpha\beta}(\bar{v}^{\ast})=&-\mu x_{i_{\beta}}x_{j_{\alpha}}\sum^{d}_{k=1}(\partial_{x_{k}}\bar{v}^{\ast})^{2}\notag\\
&+(\lambda+\mu)(x_{j_{\alpha}}\partial_{x_{i_{\alpha}}}\bar{v}^{\ast}-x_{i_{\alpha}}\partial_{x_{j_{\alpha}}}\bar{v}^{\ast})(x_{j_{\beta}}\partial_{x_{i_{\beta}}}\bar{v}^{\ast}-x_{i_{\beta}}\partial_{x_{j_{\beta}}}\bar{v}^{\ast}).
\end{align}

Therefore, due to the fact that
\begin{align*}
\left|\int^{h_{1}(x')}_{h_{2}(x')}x_{d}\,dx_{d}\right|\leq\frac{1}{2}|(h_{1}+h_{2})(x')(h_{1}-h_{2})(x')|\leq C|x'|^{2m},\quad \mathrm{in}\; B'_{R},
\end{align*}
it follows from \eqref{QKL}--\eqref{ZH006}, the parity of integrand and the symmetry of integral region that
\begin{align*}
&\mathrm{II}_{2}-\int_{\Omega^{\ast}_{R}}(\mathbb{C}^{0}e(v_{1}^{\ast\alpha}),e(v_{1}^{\ast\beta}))\notag\\
=&O(1)
\begin{cases}
\varepsilon^{\frac{d-2}{12m}},&d\geq3,\,\alpha,\beta=1,2,...,d,\,\alpha<\beta,\\
\varepsilon^{\frac{d-1}{12m}},&d\geq2,\,\alpha=1,2,...,d,\,\beta=d+1,...,\frac{d(d+1)}{2},\\
\varepsilon^{\frac{d}{12m}},&d\geq3,\,\alpha,\beta=d+1,...,\frac{d(d+1)}{2},\,\alpha<\beta.
\end{cases}
\end{align*}
This, in combination with \eqref{con036}--\eqref{PAHN}, yields that
\begin{align}\label{FATL001}
\mathrm{II}=&O(1)|\ln\varepsilon|,\quad d=2,\alpha=1,\beta=2,
\end{align}
and
\begin{align}\label{FATL002}
&\mathrm{II}-\int_{\Omega^{\ast}_{R}}(\mathbb{C}^{0}e(v_{1}^{\ast\alpha}),e(v_{1}^{\ast\beta}))\notag\\
=&O(1)
\begin{cases}
\varepsilon^{\min\{\frac{1}{6},\frac{d-1}{12m}\}},&d\geq3,\alpha,\beta=1,2,...,d,\,\alpha<\beta\\
\varepsilon^{\min\{\frac{1}{6},\frac{d-2}{12m}\}},&d\geq2,\,\alpha=1,2,...,d,\,\beta=d+1,...,\frac{d(d+1)}{2},\\
\varepsilon^{\min\{\frac{1}{6},\frac{d}{12m}\}},&d\geq3,\,\alpha,\beta=d+1,...,\frac{d(d+1)}{2},\,\alpha<\beta.
\end{cases}
\end{align}

In exactly the same way as in \eqref{QKL}, we obtain
\begin{align}\label{KKTA}
\mathrm{III}=&\int_{\Omega_{\varepsilon^{\bar{\gamma}}}}(\mathbb{C}^{0}e(v_{1}^{\alpha}),e(v_{1}^{\beta}))\notag\\
=&\int_{\Omega_{\varepsilon^{\bar{\gamma}}}}(\mathbb{C}^{0}e(\psi_{\alpha}\bar{v}),e(\psi_{\beta}\bar{v}))\notag\\
&+O(1)
\begin{cases}
\varepsilon^{(d-1)\bar{\gamma}},&\alpha=1,2,...,d,\,\beta=1,2,...,\frac{d(d+1)}{2},\,\alpha<\beta,\\
\varepsilon^{d\bar{\gamma}},&\alpha,\beta=d+1,...,\frac{d(d+1)}{2},\,\alpha<\beta.
\end{cases}
\end{align}
Using \eqref{ZH0000}--\eqref{ZH006} with $\bar{v}^{\ast}$ replaced by $\bar{v}$, it follows from \eqref{KKTA} that
\begin{align*}
\mathrm{III}=&O(1)
\begin{cases}
|\ln\varepsilon|,&d=2,\,\alpha=1,\beta=2,\\
\varepsilon^{\frac{d-2}{12m}},&d\geq3,\,\alpha,\beta=1,2,...,d,\,\alpha<\beta,\\
\varepsilon^{\frac{d-1}{12m}},&d\geq2,\,\alpha=1,2,...,d,\,\beta=d+1,...,\frac{d(d+1)}{2},\\
\varepsilon^{\frac{d}{12m}},&d\geq3,\,\alpha,\beta=d+1,...,\frac{d(d+1)}{2},\,\alpha<\beta,
\end{cases}
\end{align*}
which, together with \eqref{KKAA1233333} and \eqref{FATL001}--\eqref{FATL002}, leads to that
\begin{align*}
a_{11}^{12}=O(1)|\ln\varepsilon|,\quad d=2,
\end{align*}
and
\begin{align*}
&a_{11}^{\alpha\beta}-a_{11}^{\ast\alpha\beta}\notag\\
=&O(1)
\begin{cases}
\varepsilon^{\min\{\frac{1}{6},\frac{d-2}{12m}\}},&d\geq3,\alpha,\beta=1,2,...,d,\,\alpha<\beta,\\
\varepsilon^{\min\{\frac{1}{6},\frac{d-2}{12m}\}},&d\geq2,\,\alpha=1,2,...,d,\,\beta=d+1,...,\frac{d(d+1)}{2},\\
\varepsilon^{\min\{\frac{1}{6},\frac{d}{12m}\}},&d\geq3,\,\alpha,\beta=d+1,...,\frac{d(d+1)}{2},\,\alpha<\beta.
\end{cases}
\end{align*}
This completes the proofs of \eqref{M001}--\eqref{M005}.

\noindent{\bf Step 3.} Proof of \eqref{AZQ001}. Note that for $\alpha=1,2,...,\frac{d(d+1)}{2}$, $v_{1}^{\alpha}+v_{2}^{\alpha}-v_{1}^{\ast\alpha}-v_{2}^{\ast\alpha}$ satisfies
\begin{align*}
\begin{cases}
\mathcal{L}_{\lambda,\mu}(v_{1}^{\alpha}+v_{2}^{\alpha}-v_{1}^{\ast\alpha}-v_{2}^{\ast\alpha})=0,&\mathrm{in}\;\,D\setminus(\overline{D_{1}\cup D_{1}^{\ast}\cup D_{2}}),\\
v_{1}^{\alpha}+v_{2}^{\alpha}-v_{1}^{\ast\alpha}-v_{2}^{\ast\alpha}=\psi_{\alpha}-v_{1}^{\ast\alpha}-v_{2}^{\ast\alpha},&\mathrm{on}\;\,\partial D_{1}\setminus D_{1}^{\ast},\\
v_{1}^{\alpha}+v_{2}^{\alpha}-v_{1}^{\ast\alpha}-v_{2}^{\ast\alpha}=v_{1}^{\alpha}+v_{2}^{\alpha}-\psi_{\alpha},&\mathrm{on}\;\,\partial D_{1}^{\ast}\setminus(D_{1}\cup\{0\}),\\
v_{1}^{\alpha}+v_{2}^{\alpha}-v_{1}^{\ast\alpha}-v_{2}^{\ast\alpha}=0,&\mathrm{on}\;\,\partial D_{2}\cup\partial D.
\end{cases}
\end{align*}
Analogously as above, using the standard boundary and interior estimates for elliptic systems, we deduce that for $x\in\partial D_{1}\setminus D_{1}^{\ast}$,
\begin{align}\label{LAQ007}
&|(v_{1}^{\alpha}+v_{2}^{\alpha}-v_{1}^{\ast\alpha}-v_{2}^{\ast\alpha})(x',x_{d})|\notag\\
=&|(v_{1}^{\ast\alpha}+v_{2}^{\ast\alpha})(x',x_{d}-\varepsilon)-(v_{1}^{\ast\alpha}+v_{2}^{\ast\alpha})(x',x_{d})|\leq C\varepsilon.
\end{align}
A consequence of Corollary \ref{coro00z} gives that for $x\in\partial D_{1}^{\ast}\setminus(D_{1}\cup\mathcal{C}_{\varepsilon^{\frac{1}{m}}})$,
\begin{align}\label{LAQ008}
&|(v_{1}^{\alpha}+v_{2}^{\alpha}-v_{1}^{\ast\alpha}-v_{2}^{\ast\alpha})(x',x_{d})|\notag\\
=&|(v_{1}^{\alpha}+v_{2}^{\alpha})(x',x_{d})-(v_{1}^{\alpha}+v_{2}^{\alpha})(x',x_{d}+\varepsilon)|\leq C\varepsilon.
\end{align}
Since $v_{1}^{\alpha}+v_{2}^{\alpha}-v_{1}^{\ast\alpha}-v_{2}^{\ast\alpha}=0$ on $\partial D_{2}$, we infer from Corollary \ref{coro00z} again that for $x\in\Omega_{R}^{\ast}\cap\{|x'|=\varepsilon^{\frac{1}{m}}\}$,
\begin{align*}
&|(v_{1}^{\alpha}+v_{2}^{\alpha}-v_{1}^{\ast\alpha}-v_{2}^{\ast\alpha})(x',x_{d})|\notag\\
=&|(v_{1}^{\alpha}+v_{2}^{\alpha}-v_{1}^{\ast\alpha}-v_{2}^{\ast\alpha})(x',x_{d})-(v_{1}^{\alpha}+v_{2}^{\alpha}-v_{1}^{\ast\alpha}-v_{2}^{\ast\alpha})(x',h_{2}(x'))|\notag\\ \leq& C\delta^{1-1/m}\varepsilon\leq C\varepsilon^{2-1/m},
\end{align*}
where in the last line the fact that the exponential function decays faster than the power function was utilized. This, in combination with \eqref{LAQ007}--\eqref{LAQ008}, reads that
\begin{align}\label{MIH01}
|v_{1}^{\alpha}+v_{2}^{\alpha}-v_{1}^{\ast\alpha}-v_{2}^{\ast\alpha}|\leq C\varepsilon,\quad\;\,\mathrm{on}\;\,\partial \big(D\setminus\big(\overline{D_{1}\cup D_{1}^{\ast}\cup D_{2}\cup\mathcal{C}_{\varepsilon^{\frac{1}{m}}}}\big)\big).
\end{align}
Analogous to \eqref{con035}, it follows from \eqref{MIH01}, the maximum principle, the rescale argument, the interpolation inequality and the standard elliptic estimates that
\begin{align}\label{ZQWZW001}
|\nabla(v_{1}^{\alpha}+v_{2}^{\alpha}-v_{1}^{\ast\alpha}-v_{2}^{\ast\alpha})|\leq C\varepsilon^{\frac{1}{4}},\quad\mathrm{in}\;\,D\setminus\big(\overline{D_{1}\cup D_{1}^{\ast}\cup D_{2}\cup\mathcal{C}_{\varepsilon^{\frac{1}{4m}}}}\big).
\end{align}

Let $\tilde{\gamma}=\frac{1}{4m}$. Proceeding as above, we decompose $\sum\limits^{2}_{i=1}a_{i1}^{\alpha\beta}$ into the following three parts:
\begin{align*}
\mathrm{I}=&\int_{\Omega\setminus\Omega_{R}}(\mathbb{C}^{0}e(v_{1}^{\alpha}+v_{2}^{\alpha}),e(v_{1}^{\beta})),\notag\\
\mathrm{II}=&\int_{\Omega_{\varepsilon^{\tilde{\gamma}}}}(\mathbb{C}^{0}e(v_{1}^{\alpha}+v_{2}^{\alpha}),e(v_{1}^{\beta})),\notag\\
\mathrm{III}=&\int_{\Omega_{R}\setminus\Omega_{\varepsilon^{\tilde{\gamma}}}}(\mathbb{C}^{0}e(v_{1}^{\alpha}+v_{2}^{\alpha}),e(v_{1}^{\beta})).
\end{align*}

To begin with, using the same arguments as in \eqref{KKAA123}, we deduce from \eqref{ZQWZW001} that
\begin{align}\label{GAZ0011}
\mathrm{I}=\int_{\Omega^{\ast}\setminus\Omega^{\ast}_{R}}(\mathbb{C}^{0}e(v_{1}^{\ast\alpha}+v_{2}^{\ast\alpha}),e(v_{1}^{\ast\beta}))+O(1)\varepsilon^{\frac{1}{4}}.
\end{align}
Second, an immediate consequence of Proposition \ref{thm86} and Corollary \ref{coro00z} gives that
\begin{align}\label{GAZ001}
|\mathrm{II}|\leq \int_{|x'|\leq \varepsilon^{\tilde{\gamma}}}C(\varepsilon+|x'|^{m})^{1-1/m}\leq C\varepsilon^{\frac{m+d-2}{4m}}.
\end{align}
With regard to the third term $\mathrm{III}$, it can be further split as follows:
\begin{align*}
\mathrm{III}_{1}=&\int_{(\Omega_{R}\setminus\Omega_{\varepsilon^{\tilde{\gamma}}})\setminus(\Omega^{\ast}_{R}\setminus\Omega^{\ast}_{\varepsilon^{\tilde{\gamma}}})}(\mathbb{C}^{0}e(v_{1}^{\alpha}+v_{2}^{\alpha}),e(v_{1}^{\beta})),\\
\mathrm{III}_{2}=&\int_{\Omega^{\ast}_{R}\setminus\Omega^{\ast}_{\varepsilon^{\tilde{\gamma}}}}\sum^{2}_{i=1}\Big[(\mathbb{C}^{0}e(v_{i}^{\alpha}-v_{i}^{\ast\alpha}),e(v_{1}^{\ast\beta}))+(\mathbb{C}^{0}e(v_{i}^{\ast\alpha}),e(v_{1}^{\beta}-v_{1}^{\ast\beta}))\notag\\
&\quad\quad\quad\quad\quad\quad+(\mathbb{C}^{0}e(v_{i}^{\alpha}-v_{i}^{\ast\alpha}),e(v_{1}^{\beta}-v_{1}^{\ast\beta}))\Big],\\
\mathrm{III}_{3}=&\int_{\Omega^{\ast}_{R}\setminus\Omega^{\ast}_{\varepsilon^{\tilde{\gamma}}}}(\mathbb{C}^{0}e(v_{1}^{\ast\alpha}+v_{2}^{\ast\alpha}),e(v_{1}^{\ast\beta})).
\end{align*}
Making use of Proposition \ref{thm86} and Corollary \ref{coro00z} again, we arrive at
\begin{align}\label{GAZ00195}
|\mathrm{III}_{1}|\leq \int_{\varepsilon^{\tilde{\gamma}}\leq |x'|\leq R}\frac{C\varepsilon(\varepsilon+|x'|^{m})^{1-1/m}}{|x'|^{m}}\leq C
\begin{cases}
\varepsilon|\ln\varepsilon|,&d=2,\\
\varepsilon,&d\geq3.
\end{cases}
\end{align}
By \eqref{ZQWZW001}, we get
\begin{align}\label{GAP001}
|\mathrm{III}_{2}|\leq C\varepsilon^{\frac{1}{4}}.
\end{align}
For the last term $\mathrm{III}_{3}$, it follows from Corollary \ref{coro00z} and \eqref{GZIJ} that
\begin{align}\label{HMGD001}
\mathrm{III}_{3}=&\int_{\Omega^{\ast}_{R}}(\mathbb{C}^{0}e(v_{1}^{\ast\alpha}+v_{2}^{\ast\alpha}),e(v_{1}^{\ast\beta}))-\int_{\Omega^{\ast}_{\varepsilon^{\tilde{\gamma}}}}(\mathbb{C}^{0}e(v_{1}^{\ast\alpha}+v_{2}^{\ast\alpha}),e(v_{1}^{\ast\beta}))\notag\\
=&\int_{\Omega^{\ast}_{R}}(\mathbb{C}^{0}e(v_{1}^{\ast\alpha}+v_{2}^{\ast\alpha}),e(v_{1}^{\ast\beta}))+O(1)\varepsilon^{\frac{m+d-2}{2m}}.
\end{align}
From \eqref{GAZ00195}--\eqref{HMGD001}, we obtain
\begin{align*}
\mathrm{III}=&\int_{\Omega^{\ast}_{R}}(\mathbb{C}^{0}e(v_{1}^{\ast\alpha}+v_{2}^{\ast\alpha}),e(v_{1}^{\ast\beta}))+O(1)\varepsilon^{\frac{1}{4}},
\end{align*}
which, together with \eqref{GAZ0011}--\eqref{GAZ001}, reads that
\begin{align*}
\sum\limits^{2}_{i=1}a_{i1}^{\alpha\beta}=\sum\limits^{2}_{i=1}a_{i1}^{\ast\alpha\beta}+O(1)\varepsilon^{\frac{1}{4}}.
\end{align*}
In exactly the same way, we also have
\begin{align*}
\sum\limits^{2}_{j=1}a_{1j}^{\alpha\beta}=\sum\limits^{2}_{j=1}a_{1j}^{*\alpha\beta}+O(\varepsilon^{\frac{1}{4}}).
\end{align*}
Consequently, we complete the proof of \eqref{AZQ001}.

\end{proof}

From \eqref{AST123}, we see that for $\beta=1,2,...,\frac{d(d+1)}{2}$,
\begin{align}\label{GCT}
|v_{1}^{\beta}-v_{1}^{\ast\beta}|\leq C\varepsilon^{\frac{1}{2}},\quad\mathrm{in}\;D\setminus\big(\overline{D_{1}\cup D_{1}^{\ast}\cup D_{2}\cup\mathcal{C}_{\varepsilon^{\frac{1}{2m}}}}\big).
\end{align}
Then applying the standard boundary estimates with \eqref{GCT}, we get
\begin{align*}
|\nabla(v_{1}^{\beta}-v_{1}^{\ast\beta})|\leq C\varepsilon^{\frac{1}{2}},\quad\mathrm{on}\;\partial D.
\end{align*}
Consequently,
\begin{align*}
|b_1^{\beta}-b_{1}^{\ast\beta}|\leq\left|\int_{\partial D}\frac{\partial(v_{1}^{\beta}-v_{1}^{\ast\beta})}{\partial \nu_0}\large\Big|_{+}\cdot\varphi\right|\leq C\|\varphi\|_{C^{0}(\partial D)}\varepsilon^{\frac{1}{2}}.
\end{align*}
Analogously,
\begin{align*}
|b_{2}^{\beta}-b_{2}^{\ast\beta}|\leq C\|\varphi\|_{C^{0}(\partial D)}\varepsilon^{\frac{1}{2}}.
\end{align*}
That is, we have
\begin{lemma}\label{KM323}
Assume as above. Then for a sufficiently small $\varepsilon>0$,
\begin{align*}
b_{i}^{\beta}=b_{i}^{\ast\beta}+O(\varepsilon^{\frac{1}{2}}),\quad i=1,2,\;\beta=1,2,...,\frac{d(d+1)}{2},
\end{align*}
which yields that
\begin{align*}
\sum\limits^{2}_{i=1}b_{i}^{\beta}=\sum\limits^{2}_{i=1}b_{i}^{\ast\beta}+O(\varepsilon^{\frac{1}{2}}),
\end{align*}
where $b_{i}^{\ast\beta}$, $b_{i}^{\beta}$, $i=1,2,\beta=1,2,...,\frac{d(d+1)}{2}$ are defined in \eqref{LMZR} and \eqref{KAT001}, respectively.

\end{lemma}

Combining Lemmas \ref{lemmabc} and \ref{KM323}, we obtain the following convergence results.
\begin{lemma}\label{COOO}
Assume as above. Let $C_{2}^{\alpha}$, $\alpha=1,2,...,\frac{d(d+1)}{2}$ be defined in \eqref{Decom}. Then for a sufficiently small $\varepsilon>0$,
\begin{align}\label{LGR001}
C_{2}^{\alpha}=&C_{\ast}^{\alpha}+O(r_{\varepsilon}),\quad\alpha=1,2,...,\frac{d(d+1)}{2},
\end{align}
where $C_{\ast}^{\alpha}$ is defined by \eqref{ZZWWWW}, $r_{\varepsilon}$ is defined in \eqref{JTD}.

\end{lemma}
\begin{remark}
The asymptotic results in \eqref{LGR001} actually provide the explicit upper and lower bounds on the free constants $C_{2}^{\alpha}$, $\alpha=1,2,...,\frac{d(d+1)}{2}$, which improves the results of Lemma 4.1 in \cite{BLL2015} and Proposition 4.1 in \cite{BLL2017} in terms of the boundness of these free constants.
\end{remark}

Before proving Lemma \ref{COOO}, we first recall the following lemma on the linear space of rigid displacement $\Psi$, which is Lemma 6.1 of \cite{BLL2017}.
\begin{lemma}\label{GLW}
Let $\xi$ be an element of $\Psi$, defined by \eqref{LAK01} with $d\geq2$. If $\xi$ vanishes at $d$ distinct points $\bar{x}_{1}$, $i=1,2,...,d$, which do not lie on a $(d-1)$-dimensional plane, then $\xi=0$.
\end{lemma}

\begin{proof}[The proof of Lemma \ref{COOO}]
We divide into three cases to complete the proof of Lemma \ref{COOO}.

$(\mathrm{i})$ If $m\geq d+1$, for $\alpha=1,2,...,\frac{d(d+1)}{2}$, we replace the elements of $\alpha$-th column in the matrix $\mathbb{D}$ defined in \eqref{GGDA} by column vector $\Big(\sum\limits_{i=1}^{2}b_{i}^{1},...,\sum\limits_{i=1}^{2}b_{i}^{\frac{d(d+1)}{2}}\Big)^{T}$ and then obtain new matrix $\mathbb{D}^{\ast\alpha}$ as follows:
\begin{gather*}
\mathbb{D}^{\alpha}=
\begin{pmatrix}
\sum\limits^{2}_{i,j=1}a_{ij}^{11}&\cdots&\sum\limits_{i=1}^{2}b_{i}^{1}&\cdots&\sum\limits^{2}_{i,j=1}a_{ij}^{1\,\frac{d(d+1)}{2}} \\\\ \vdots&\ddots&\vdots&\ddots&\vdots\\\\ \sum\limits^{2}_{i,j=1}a_{ij}^{\frac{d(d+1)}{2}\,1}&\cdots&\sum\limits_{i=1}^{2}b_{i}^{\frac{d(d+1)}{2}}&\cdots&\sum\limits^{2}_{i,j=1}a_{ij}^{\frac{d(d+1)}{2}\,\frac{d(d+1)}{2}}
\end{pmatrix}.
\end{gather*}
Then we deduce from Lemmas \ref{lemmabc} and \ref{KM323} that for $\alpha=1,2,...,\frac{d(d+1)}{2}$,
\begin{align}\label{KTAZ001}
\det\mathbb{D}=\det\mathbb{D}^{\ast}+O(\varepsilon^{\frac{1}{4}}),\quad\det\mathbb{D}^{\alpha}=\det\mathbb{D}^{\ast\alpha}+O(\varepsilon^{\frac{1}{4}}).
\end{align}
Claim that $\mathbb{D}^{\ast}$ is a positive definite matrix and then $\det\mathbb{D}^{\ast}>0$. In fact, it follows from ellipticity condition \eqref{ellip} that for any $\xi=(\xi_{1},\xi_{2},...,\xi_{\frac{d(d+1)}{2}})^{T}\neq0$,
\begin{align}\label{AJZ001}
\xi^{T}\mathbb{D}^{\ast}\xi=&\int_{\Omega^{\ast}}\Bigg(\mathbb{C}^{0}e\bigg(\sum^{\frac{d(d+1)}{2}}_{\alpha=1}\xi_{\alpha}(v_{1}^{\ast\alpha}+v_{2}^{\ast\alpha})\bigg),e\bigg(\sum^{\frac{d(d+1)}{2}}_{\beta=1}\xi_{\beta}(v_{1}^{\ast\beta}+v_{2}^{\ast\beta})\bigg)\Bigg)\notag\\
\geq&\frac{1}{C}\int_{\Omega^{\ast}}\bigg|e\bigg(\sum^{\frac{d(d+1)}{2}}_{\alpha=1}\xi_{\alpha}(v_{1}^{\ast\alpha}+v_{2}^{\ast\alpha})\bigg)\bigg|^{2}>0,
\end{align}
where in the last line we utilized the fact that $e\big(\sum^{\frac{d(d+1)}{2}}_{\alpha=1}\xi_{\alpha}(v_{1}^{\ast\alpha}+v_{2}^{\ast\alpha})\big)\not\equiv0$ in $\Omega^{\ast}$. Otherwise, if $e\big(\sum^{\frac{d(d+1)}{2}}_{\alpha=1}\xi_{\alpha}(v_{1}^{\ast\alpha}+v_{2}^{\ast\alpha})\big)\equiv0$ in $\Omega^{\ast}$, then there exist some constants $a_{i}$, $i=1,2,...,\frac{d(d+1)}{2}$ such that $\sum_{\alpha=1}^{\frac{d(d+1)}{2}}\xi_{\alpha}(v_{1}^{\ast\alpha}+v_{2}^{\ast\alpha})=\sum^{\frac{d(d+1)}{2}}_{i=1}a_{i}\psi_{i}$, where $\psi_{i},\,i=1,2,...,\frac{d(d+1)}{2}$ are defined in \eqref{OPP}. In light of the fact that $\psi_{i},\,i=1,2,...,\frac{d(d+1)}{2}$ are linear independent and $\sum_{\alpha=1}^{\frac{d(d+1)}{2}}\xi_{\alpha}(v_{1}^{\ast\alpha}+v_{2}^{\ast\alpha})=0$ on $\partial D$, we deduce from Lemma \ref{GLW} that $a_{i}=0$, $i=1,2,...,\frac{d(d+1)}{2}$. Since $\sum_{\alpha=1}^{\frac{d(d+1)}{2}}\xi_{\alpha}(v_{1}^{\ast\alpha}+v_{2}^{\ast\alpha})=\sum_{\alpha=1}^{\frac{d(d+1)}{2}}\xi_{\alpha}\psi_{\alpha}=0$ on $\partial D_{1}^{\ast}$, then it follows from the linear independence of $\big\{\psi_{i}\big|\,i=1,2,...,\frac{d(d+1)}{2}\big\}$ again that $\xi=(\xi_{1},\xi_{2},...,\xi_{\frac{d(d+1)}{2}})^{T}=0$. This is a contradiction.

From \eqref{KTAZ001}, we obtain
\begin{align}\label{AGU01}
\frac{\det\mathbb{D}^{\alpha}}{\det\mathbb{D}}=&\frac{\det\mathbb{D}^{\ast\alpha}}{\det\mathbb{D}^{\ast}}\frac{1}{1-{\frac{\det\mathbb{D}^{\ast}-\det\mathbb{D}}{\det\mathbb{D}^{\ast}}}}+\frac{\det\mathbb{D}^{\alpha}-\det\mathbb{D}^{\ast\alpha}}{\det\mathbb{D}}\notag\\
=&\frac{\det\mathbb{D}^{\ast\alpha}}{\det\mathbb{D}^{\ast}}(1+O(\varepsilon^{\frac{1}{4}})).
\end{align}
In light of \eqref{PLA001}, it follows from Cramer's rule and \eqref{AGU01} that for $\alpha=1,2,...,\frac{d(d+1)}{2}$,
\begin{align}\label{MKZ001}
C_{2}^{\alpha}=&\frac{\det\mathbb{D}^{\alpha}}{\det\mathbb{D}}(1+O(\rho_{2}^{-1}(d,m;\varepsilon)))=\frac{\det\mathbb{D}^{\ast\alpha}}{\det\mathbb{D}^{\ast}}(1+O(\varepsilon^{\min\{\frac{1}{4},\frac{m-d-1}{m}\}})).
\end{align}

$(\mathrm{ii})$ If $d-1\leq m<d+1$, define
\begin{gather}\mathbb{A}_{0}=\begin{pmatrix} a_{11}^{d+1\,d+1}&\cdots&a_{11}^{d+1\frac{d(d+1)}{2}} \\\\ \vdots&\ddots&\vdots\\\\a_{11}^{\frac{d(d+1)}{2}d+1}&\cdots&a_{11}^{\frac{d(d+1)}{2}\frac{d(d+1)}{2}}\end{pmatrix},\label{HARBN001}\\
%\end{gather}
%\begin{gather}
\mathbb{B}_{0}=\begin{pmatrix} \sum\limits^{2}_{i=1}a_{i1}^{d+1\,1}&\cdots&\sum\limits^{2}_{i=1}a_{i1}^{d+1\,\frac{d(d+1)}{2}} \\\\ \vdots&\ddots&\vdots\\\\ \sum\limits^{2}_{i=1}a_{i1}^{\frac{d(d+1)}{2}1}&\cdots&\sum\limits^{2}_{i=1}a_{i1}^{\frac{d(d+1)}{2}\frac{d(d+1)}{2}}\end{pmatrix},\notag\\
\mathbb{C}_{0}=\begin{pmatrix} \sum\limits^{2}_{j=1}a_{1j}^{1\,d+1}&\cdots&\sum\limits^{2}_{j=1}a_{1j}^{1\frac{d(d+1)}{2}} \\\\ \vdots&\ddots&\vdots\\\\ \sum\limits^{2}_{j=1}a_{1j}^{\frac{d(d+1)}{2}\,d+1}&\cdots&\sum\limits^{2}_{j=1}a_{1j}^{\frac{d(d+1)}{2}\frac{d(d+1)}{2}}\end{pmatrix}.\notag
\end{gather}
For $\alpha=1,2,...,\frac{d(d+1)}{2}$, by substituting column vector $\Big(b_{1}^{d+1},...,b_{1}^{\frac{d(d+1)}{2}}\Big)^{T}$ for the elements of $\alpha$-th column in the matrix $\mathbb{B}_{0}$, we generate new matrix $\mathbb{B}_{0}^{\alpha}$ as follows:
\begin{gather*}
\mathbb{B}_{0}^{\alpha}=
\begin{pmatrix}
\sum\limits^{2}_{i=1}a_{i1}^{d+1\,1}&\cdots&b_{1}^{d+1}&\cdots&\sum\limits^{2}_{i=1}a_{i1}^{d+1\,\frac{d(d+1)}{2}} \\\\ \vdots&\ddots&\vdots&\ddots&\vdots\\\\ \sum\limits^{2}_{i=1}a_{i1}^{\frac{d(d+1)}{2}\,1}&\cdots&b_{1}^{\frac{d(d+1)}{2}}&\cdots&\sum\limits^{2}_{i=1}a_{i1}^{\frac{d(d+1)}{2}\frac{d(d+1)}{2}}
\end{pmatrix}.
\end{gather*}
Let
\begin{align*}
\mathbb{F}_{0}=\begin{pmatrix} \mathbb{A}_{0}&\mathbb{B}_{0} \\  \mathbb{C}_{0}&\mathbb{D}
\end{pmatrix},\quad \mathbb{F}_{0}^{\alpha}=\begin{pmatrix} \mathbb{A}_{0}&\mathbb{B}_{0}^{\alpha} \\  \mathbb{C}_{0}&\mathbb{D}^{\alpha}
\end{pmatrix},\quad\alpha=1,2,...,\frac{d(d+1)}{2}.
\end{align*}

Therefore, it follows from Lemmas \ref{lemmabc} and \ref{KM323} again that
\begin{align*}
\det\mathbb{F}_{0}=\det\mathbb{F}_{0}^{\ast}+O(\varepsilon^{\frac{d+1-m}{12m}}),\quad\det\mathbb{F}_{0}^{\alpha}=\det\mathbb{F}^{\ast\alpha}_{0}+O(\varepsilon^{\frac{d+1-m}{12m}}),
\end{align*}
which reads that
\begin{align}\label{QGH01}
\frac{\det\mathbb{F}_{0}^{\alpha}
}{\det\mathbb{F}_{0}}=&\frac{\det\mathbb{F}_{0}^{\ast\alpha}}{\det\mathbb{F}_{0}^{\ast}}\frac{1}{1-{\frac{\det\mathbb{F}_{0}^{\ast}-\det\mathbb{F}_{0}}{\det\mathbb{F}_{0}^{\ast}}}}+\frac{\det\mathbb{F}_{0}^{\alpha}-\det\mathbb{F}_{0}^{\ast\alpha}}{\det\mathbb{F}_{0}}\notag\\
=&\frac{\det\mathbb{F}_{0}^{\ast\alpha}}{\det\mathbb{F}_{0}^{\ast}}(1+O(\varepsilon^{\frac{d+1-m}{12m}})).
\end{align}
Following the same argument as in \eqref{AJZ001}, we deduce that $\det\mathbb{F}_{0}^{\ast}\neq0$. Denote
\begin{align}\label{LHAZ001}
\rho_{d}(\varepsilon)=&
\begin{cases}
|\ln\varepsilon|,&d=2,\\
1,&d\geq3.
\end{cases}
\end{align}
Then applying Cramer's rule to \eqref{PLA001}, we deduce from \eqref{QGH01} that for $\alpha=1,2,...,\frac{d(d+1)}{2}$,
\begin{align}\label{MKZ002}
C_{2}^{\alpha}=&\frac{\det\mathbb{F}_{0}^{\alpha}}{\det \mathbb{F}_{0}}(1+O(\rho_{0}^{-1}(d,m;\varepsilon)\rho_{d}(\varepsilon)))\notag\\
=&\frac{\det\mathbb{F}_{0}^{\ast\alpha}}{\det \mathbb{F}_{0}^{\ast}}
\begin{cases}
1+O(\varepsilon^{\min\{\frac{d+1-m}{12m},\frac{m-d+1}{m}\}}),&d-1<m<d+1,\\
1+O(|\ln\varepsilon|^{-1}),&m=d-1.
\end{cases}
\end{align}

$(\mathrm{iii})$ If $m<d-1$, by replacing the elements of $\alpha$-th column in the matrix $\mathbb{B}$, by column vector $\big(b_{1}^{1},...,b_{1}^{\frac{d(d+1)}{2}}\big)^{T}$, we obtain new matrices $\mathbb{B}_{1}^{\alpha}$ as follows:
\begin{gather*}
\mathbb{B}_{1}^{\alpha}=
\begin{pmatrix}
\sum\limits^{2}_{i=1}a_{i1}^{11}&\cdots&b_{1}^{1}&\cdots&\sum\limits^{2}_{i=1}a_{i1}^{1\,\frac{d(d+1)}{2}} \\\\ \vdots&\ddots&\vdots&\ddots&\vdots\\\\ \sum\limits^{2}_{i=1}a_{i1}^{\frac{d(d+1)}{2}\,1}&\cdots&b_{1}^{\frac{d(d+1)}{2}}&\cdots&\sum\limits^{2}_{i=1}a_{i1}^{\frac{d(d+1)}{2}\frac{d(d+1)}{2}}
\end{pmatrix}.
\end{gather*}
Write
\begin{align*}
\mathbb{F}_{1}=\begin{pmatrix} \mathbb{A}&\mathbb{B} \\  \mathbb{C}&\mathbb{D}
\end{pmatrix},\quad\mathbb{F}^{\alpha}_{1}=\begin{pmatrix} \mathbb{A}&\mathbb{B}_{1}^{\alpha} \\  \mathbb{C}&\mathbb{D}^{\alpha}
\end{pmatrix},\;\,\alpha=1,2,...,\frac{d(d+1)}{2}.
\end{align*}
Then from Lemmas \ref{lemmabc} and \ref{KM323}, we derive that for $\alpha=1,2,...,\frac{d(d+1)}{2}$,
\begin{align*}
\det\mathbb{F}_{1}=\det\mathbb{F}_{1}^{\ast}+O(\varepsilon^{\min\{\frac{1}{6},\frac{d-1-m}{12m}\}}),\quad\det\mathbb{F}_{1}^{\alpha}=\det\mathbb{F}^{\ast\alpha}_{1}+O(\varepsilon^{\min\{\frac{1}{6},\frac{d-1-m}{12m}\}}).
\end{align*}
By the same argument as in \eqref{AJZ001}, we obtain that $\det\mathbb{F}^{\ast}_{1}\neq0$. Then we have
\begin{align*}
\frac{\det\mathbb{F}_{1}^{\alpha}
}{\det\mathbb{F}_{1}}=&\frac{\det\mathbb{F}_{1}^{\ast\alpha}}{\det\mathbb{F}_{1}^{\ast}}\frac{1}{1-{\frac{\det\mathbb{F}_{1}^{\ast}-\det\mathbb{F}_{1}}{\det\mathbb{F}_{1}^{\ast}}}}+\frac{\det\mathbb{F}_{1}^{\alpha}-\det\mathbb{F}_{1}^{\ast\alpha}}{\det\mathbb{F}_{1}}\notag\\
=&\frac{\det\mathbb{F}^{\ast\alpha}_{1}}{\det\mathbb{F}_{1}^{\ast}}(1+O(\varepsilon^{\min\{\frac{1}{6},\frac{d-1-m}{12m}\}})),\quad\alpha=1,2,...,\frac{d(d+1)}{2}.
\end{align*}
This, in combination with \eqref{PLA001} and Cramer's rule, reads that for $\alpha=1,2,...,\frac{d(d+1)}{2}$,
\begin{align}\label{MKZ003}
C_{2}^{\alpha}=&\frac{\det\mathbb{F}_{1}^{\alpha}}{\det \mathbb{F}_{1}}=\frac{\det\mathbb{F}_{1}^{\ast\alpha}}{\det \mathbb{F}^{\ast}_{1}}(1+O(\varepsilon^{\min\{\frac{1}{6},\frac{d-1-m}{12m}\}})).
\end{align}

Then Lemma \ref{COOO} follows from \eqref{MKZ001} and \eqref{MKZ002}--\eqref{MKZ003}.

\end{proof}

We are now ready to give the proof of Theorem \ref{JGR}.
\begin{proof}[The proof of Theorem \ref{JGR}.]
By linearity, we obtain that $u_{b}^{\ast}$ can be decomposed as follows:
\begin{align}\label{KTG001}
u_{b}^{\ast}=\sum^{\frac{d(d+1)}{2}}_{\alpha=1}C_{\ast}^{\alpha}(v_{1}^{\ast\alpha}+v_{2}^{\ast\alpha})+v_{0}^{\ast},
\end{align}
where $v_{0}^{\ast}$ and $v_{i}^{\ast\alpha}$, $i=1,2$ are defined by \eqref{ZG001} and \eqref{qaz001111}, respectively. Then in view of \eqref{KTG001} and using integration by parts, it follows from \eqref{AZQ001}, Lemmas \ref{KM323} and \ref{COOO} that
\begin{align*}
\mathcal{B}_{\beta}[\varphi]-\mathcal{B}_{\beta}^{\ast}[\varphi]=&\int_{\partial D_{1}}\frac{\partial u_{b}}{\partial\nu_{0}}\Big|_{+}\cdot\psi_{\beta}-\int_{\partial D_{1}^{\ast}}\frac{\partial u_{b}^{\ast}}{\partial\nu_{0}}\Big|_{+}\cdot\psi_{\beta}\notag\\
=&\sum^{\frac{d(d+1)}{2}}_{\alpha=1}(C_{\ast}^{\alpha}-C_{2}^{\alpha})\sum^{2}_{i=1}a_{i1}^{\alpha\beta}+\sum^{\frac{d(d+1)}{2}}_{\alpha=1}C_{\ast}^{\alpha}\left(\sum^{2}_{i=1}a_{i1}^{\ast\alpha\beta}-\sum^{2}_{i=1}a_{i1}^{\alpha\beta}\right)\notag\\
&+b_{1}^{\beta}-b_{1}^{\ast\beta}=O(r_{\varepsilon}),
\end{align*}
where $r_{\varepsilon}$ is defined by \eqref{JTD}. This completes the proof.

\end{proof}

\section{Application in the gradient estimates and asymptotics}\label{KGRA90}

As an immediate consequence of Theorem \ref{JGR}, we obtain the optimal upper and lower bounds on the gradient in all dimensions and the asymptotic expansions of the gradient for two adjacent inclusions with different principal curvatures in dimensions three, which will be presented in the following two theorems.

\subsection{Gradient estimates}
To begin with, for $\alpha=1,2,...,d$, define the blow-up factor matrices as follows:
\begin{gather}\label{PTCA001}
\mathbb{A}^{\ast\alpha}_{1}=\begin{pmatrix}\mathcal{B}_{\alpha}^{\ast}[\varphi]&a_{11}^{\ast\alpha\,d+1}&\cdots&a_{11}^{\ast\alpha\frac{d(d+1)}{2}} \\ \mathcal{B}_{d+1}^{\ast}[\varphi]&a_{11}^{\ast d+1\,d+1}&\cdots&a_{11}^{\ast d+1\frac{d(d+1)}{2}}\\ \vdots&\vdots&\ddots&\vdots\\ \mathcal{B}_{\frac{d(d+1)}{2}}^{\ast}[\varphi]&a_{11}^{\ast\frac{d(d+1)}{2}d+1}&\cdots&a_{11}^{\ast\frac{d(d+1)}{2}\frac{d(d+1)}{2}}
\end{pmatrix}.
\end{gather}
Second, for $\alpha=d+1,...,\frac{d(d+1)}{2}$, after replacing the elements of $\alpha$-th column in the matrix $\mathbb{A}_{0}^{\ast}$ defined in \eqref{LAGT001} by column vector $\big(\mathcal{B}_{d+1}^{\ast}[\varphi],...,\mathcal{B}_{\frac{d(d+1)}{2}}^{\ast}[\varphi]\big)^{T}$, we generate the new matrix $\mathbb{A}_{2}^{\ast\alpha}$ as follows:
\begin{gather}\label{DKY001}
\mathbb{A}_{2}^{\ast\alpha}=
\begin{pmatrix}
a_{11}^{\ast d+1\,d+1}&\cdots&\mathcal{B}_{d+1}^{\ast}[\varphi]&\cdots&a_{11}^{\ast d+1\,\frac{d(d+1)}{2}} \\\\ \vdots&\ddots&\vdots&\ddots&\vdots\\\\a_{11}^{\ast\frac{d(d+1)}{2}\,d+1}&\cdots&\mathcal{B}_{\frac{d(d+1)}{2}}^{\ast}[\varphi]&\cdots&a_{11}^{\ast\frac{d(d+1)}{2}\,\frac{d(d+1)}{2}}
\end{pmatrix}.
\end{gather}

Third, for $\alpha=1,2,...,\frac{d(d+1)}{2}$, we substitute column vector $\big(\mathcal{B}_{1}^{\ast}[\varphi],...,\mathcal{B}^{\ast}_{\frac{d(d+1)}{2}}[\varphi]\big)^{T}$ for the elements of $\alpha$-th column in the matrix $\mathbb{A}^{\ast}$ defined in \eqref{WZW} and then derive new matrices $\mathbb{A}_{3}^{\ast\alpha}$ as follows:
\begin{gather}\label{JGAT001}
\mathbb{A}_{3}^{\ast\alpha}=
\begin{pmatrix}
a_{11}^{\ast11}&\cdots&\mathcal{B}_{1}^{\ast}[\varphi]&\cdots&a_{11}^{\ast1\,\frac{d(d+1)}{2}} \\\\ \vdots&\ddots&\vdots&\ddots&\vdots\\\\a_{11}^{\ast\frac{d(d+1)}{2}\,1}&\cdots&\mathcal{B}_{\frac{d(d+1)}{2}}^{\ast}[\varphi]&\cdots&a_{11}^{\ast\frac{d(d+1)}{2}\,\frac{d(d+1)}{2}}
\end{pmatrix}.
\end{gather}
Here and below, $a\lesssim b$ (or $a\gtrsim b$) represents $a\leq Cb$ (or $a\geq \frac{1}{C}b$) for some positive constant $C$, depending only on $d,R,\tau_{0},\tau_{3},\tau_{4},$ and the $C^{2,\alpha}$ norms of $\partial D_{1}$ and $\partial D_{2}$, but not on the distance parameter $\varepsilon$, the curvature parameters $\tau_{1}$ and $\tau_{2}$, and the Lam\'{e} constants $\mathcal{L}_{d}^{\alpha}$, $\alpha=1,2,...,\frac{d(d+1)}{2}$. Let $a\simeq b$ denote $\frac{a}{b}=1+o(1)$, where $b\neq0$ and $\lim\limits_{\varepsilon\rightarrow0}o(1)=0$. This means that $b$ is approximately equal to $a$.

\begin{theorem}\label{MGA001}
Let $D_{1},D_{2}\subset D\subset\mathbb{R}^{d}\,(d\geq2)$ be defined as above, conditions $\mathrm{(}${\bf{H1}}$\mathrm{)}$--$\mathrm{(}${\bf{H3}}$\mathrm{)}$ hold. Let $u\in H^{1}(D;\mathbb{R}^{d})\cap C^{1}(\overline{\Omega};\mathbb{R}^{d})$ be the solution of (\ref{La.002}). Then for a sufficiently small $\varepsilon>0$,

$(\rm{i})$ if $m>d$, there exists some integer $d+1\leq \alpha_{0}\leq\frac{d(d+1)}{2}$ such that $|\mathcal{B}^{\ast}_{\alpha_{0}}[\varphi]|\neq0$, then for some $1\leq i_{\alpha_{0}}\leq d-1$ and $x\in\{x'|\,x_{i_{\alpha_{0}}}=\sqrt[m]{\varepsilon},\,x_{i}=0,i\neq i_{\alpha_{0}},1\leq i\leq d-1\}\cap\Omega$,
\begin{align*}
|\nabla u|\lesssim
\begin{cases}
\frac{\max\limits_{d+1\leq\alpha\leq\frac{d(d+1)}{2}}\tau_{2}^{\frac{d+1}{m}}|\mathcal{L}_{d}^{\alpha}|^{-1}|\mathcal{B}^{\ast}_{\alpha}[\varphi]|}{1+\tau_{1}}\frac{\varepsilon^{\frac{1-m}{m}}}{\rho_{2}(d,m;\varepsilon)},&m\geq d+1,\\
\frac{\max\limits_{d+1\leq\alpha\leq\frac{d(d+1)}{2}}|\det\mathbb{A}_{2}^{\ast\alpha}|}{(1+\tau_{1})|\det\mathbb{A}_{0}^{\ast}|}\frac{1}{\varepsilon^{\frac{m-1}{m}}},&d<m<d+1,
\end{cases}
\end{align*}
and
\begin{align*}
|\nabla u|\gtrsim
\begin{cases}
\frac{\tau_{1}^{\frac{d+1}{m}}||\mathcal{B}^{\ast}_{\alpha_{0}}[\varphi]||}{(1+\tau_{2})|\mathcal{L}_{d}^{\alpha_{0}}|}\frac{\varepsilon^{\frac{1-m}{m}}}{\rho_{2}(d,m;\varepsilon)},&m\geq d+1,\\
\frac{|\det\mathbb{A}^{\ast\alpha_{0}}_{2}|}{(1+\tau_{2})|\det\mathbb{A}^{\ast}_{0}|}\frac{1}{\varepsilon^{\frac{m-1}{m}}},&d<m<d+1;
\end{cases}
\end{align*}

$(\rm{ii})$ if $m\leq d$, there exists some integer $1\leq \alpha_{0}\leq d$ such that $\det\mathbb{A}_{1}^{\ast\alpha_{0}}\neq0$, then for $x\in\{|x'|=0\}\cap\Omega$,
\begin{align*}
|\nabla u|\lesssim
\begin{cases}
\frac{\max\limits_{1\lesssim\alpha\lesssim d}\tau_{2}^{\frac{d-1}{m}}|\mathcal{L}_{d}^{\alpha}|^{-1}|\det\mathbb{A}_{1}^{\ast\alpha}|}{|\det\mathbb{A}_{0}^{\ast}|}\frac{1}{\varepsilon\rho_{0}(d,m;\varepsilon)},&d-1\leq m\leq d,\\
\frac{\max\limits_{1\leq\alpha\leq d}|\det\mathbb{A}_{3}^{\ast\alpha}|}{|\det \mathbb{A}^{\ast}|}\frac{1}{\varepsilon},&m<d-1,
\end{cases}
\end{align*}
and
\begin{align*}
|\nabla u|\gtrsim
\begin{cases}
\frac{\tau_{1}^{\frac{d-1}{m}}|\det\mathbb{A}_{1}^{\ast\alpha_{0}}|}{|\mathcal{L}_{d}^{\alpha_{0}}||\det\mathbb{A}_{0}^{\ast}|}\frac{1}{\varepsilon\rho_{0}(d,m;\varepsilon)},&d-1\leq m\leq d,\\
\frac{|\det\mathbb{A}_{3}^{\ast\alpha_{0}}|}{|\det \mathbb{A}^{\ast}|}\frac{1}{\varepsilon},&m<d-1,
\end{cases}
\end{align*}
where $\mathbb{A}^{\ast}$ and $\mathbb{A}_{0}^{\ast}$ are, respectively, defined in \eqref{WZW} and \eqref{LAGT001}, the blow-up factors $\mathcal{B}_{\alpha}^{\ast}[\varphi]$, $\alpha=d+1,...,\frac{d(d+1)}{2}$ are defined by \eqref{KGF001}, $\mathbb{A}_{1}^{\ast\alpha}$, $\alpha=1,2,...,d$, $\mathbb{A}_{2}^{\ast\alpha}$, $\alpha=d+1,...,\frac{d(d+1)}{2}$, $\mathbb{A}_{3}^{\ast\alpha}$, $\alpha=1,2,...,d$ are defined by \eqref{PTCA001}--\eqref{JGAT001}.
\end{theorem}
\begin{remark}
By using the stress concentration factors $\mathcal{B}_{\beta}^{\ast}[\varphi]$, $\beta=1,2,...,\frac{d(d+1)}{2}$, we give the optimal upper and lower bounds on the blow-up rate of the stress in Theorem \ref{MGA001} and thus prove its optimality. Moreover, the results in Theorem \ref{MGA001} improve and make complete the gradient estimates in the previous work \cite{BLL2015,BLL2017,L2018,HJL2018,LX2020} by capturing the blow-up factor matrices in all dimensions.

\end{remark}

\begin{proof}[The proof of Theorem \ref{MGA001}]

{\bf Step 1.} We divide into three cases to calculate the difference $C_{1}^{\alpha}-C_{2}^{\alpha}$, $\alpha=1,2,...,\frac{d(d+1)}{2}$.

{\bf Case 1.} Consider $m\geq d+1$.
Then applying Cramer's rule to \eqref{JGRO001}, it follows from Theorem \ref{JGR} and \eqref{LMC}--\eqref{LMC1} that for $\alpha=1,2,...,\frac{d(d+1)}{2}$,
\begin{align}\label{MNT001}
C_{1}^{\alpha}-C_{2}^{\alpha}=&\frac{\mathcal{B}_{\alpha}[\varphi]}{a_{11}^{\alpha\alpha}}(1+O(\rho_{2}^{-1}(d,m;\varepsilon))),
\end{align}
which reads that
\begin{align}\label{HND001}
|C_{1}^{\alpha}-C_{2}^{\alpha}|\lesssim&
\begin{cases}
\frac{\tau_{2}^{\frac{d-1}{m}}|\mathcal{B}_{\alpha}^{\ast}[\varphi]|}{|\mathcal{L}_{d}^{\alpha}|\rho_{0}(d,m;\varepsilon)},&\alpha=1,2,...,d,\\
\frac{\tau_{2}^{\frac{d+1}{m}}|\mathcal{B}_{\alpha}^{\ast}[\varphi]|}{|\mathcal{L}_{d}^{\alpha}|\rho_{2}(d,m;\varepsilon)},&\alpha=d+1,...,\frac{d(d+1)}{2},
\end{cases}
\end{align}
and
\begin{align}\label{HND002}
|C_{1}^{\alpha}-C_{2}^{\alpha}|\gtrsim&
\begin{cases}
\frac{\tau_{1}^{\frac{d-1}{m}}|\mathcal{B}_{\alpha}^{\ast}[\varphi]|}{|\mathcal{L}_{d}^{\alpha}|\rho_{0}(d,m;\varepsilon)},&\alpha=1,2,...,d,\\
\frac{\tau_{1}^{\frac{d+1}{m}}|\mathcal{B}_{\alpha}^{\ast}[\varphi]|}{|\mathcal{L}_{d}^{\alpha}|\rho_{2}(d,m;\varepsilon)},&\alpha=d+1,...,\frac{d(d+1)}{2}.
\end{cases}
\end{align}

{\bf Case 2.} Consider $d-1\leq m<d+1$. For $\alpha=1,2,...,d$, denote
\begin{gather*}
\mathbb{A}^{\alpha}_{1}=\begin{pmatrix}\mathcal{B}_{\alpha}[\varphi]&a_{11}^{\alpha\,d+1}&\cdots&a_{11}^{\alpha\frac{d(d+1)}{2}} \\ \mathcal{B}_{d+1}[\varphi]&a_{11}^{d+1\,d+1}&\cdots&a_{11}^{d+1\frac{d(d+1)}{2}}\\ \vdots&\vdots&\ddots&\vdots\\ \mathcal{B}_{\frac{d(d+1)}{2}}[\varphi]&a_{11}^{\frac{d(d+1)}{2}d+1}&\cdots&a_{11}^{\frac{d(d+1)}{2}\frac{d(d+1)}{2}}
\end{pmatrix},
\end{gather*}
and, for $\alpha=d+1,...,\frac{d(d+1)}{2}$, after replacing the elements of $\alpha$-th column in the matrix $\mathbb{A}_{0}$ defined in \eqref{HARBN001} by column vector $\big(\mathcal{B}_{d+1}[\varphi],...,\mathcal{B}_{\frac{d(d+1)}{2}}[\varphi]\big)^{T}$, we obtain a new matrix denoted by $\mathbb{A}_{2}^{\alpha}$ as follows:
\begin{gather*}
\mathbb{A}_{2}^{\alpha}=
\begin{pmatrix}
a_{11}^{d+1\,d+1}&\cdots&\mathcal{B}_{d+1}[\varphi]&\cdots&a_{11}^{d+1\,\frac{d(d+1)}{2}} \\\\ \vdots&\ddots&\vdots&\ddots&\vdots\\\\a_{11}^{\frac{d(d+1)}{2}\,d+1}&\cdots&\mathcal{B}_{\frac{d(d+1)}{2}}[\varphi]&\cdots&a_{11}^{\frac{d(d+1)}{2}\,\frac{d(d+1)}{2}}
\end{pmatrix}.
\end{gather*}
Then utilizing Cramer's rule for \eqref{JGRO001}, we deduce from Theorem \ref{JGR} and Lemma \ref{lemmabc} that
\begin{align}\label{LAMNZ001}
C_{1}^{\alpha}-C_{2}^{\alpha}=&
\begin{cases}
\frac{\prod\limits_{i\neq\alpha}^{d}a_{11}^{ii}\det\mathbb{A}_{1}^{\alpha}}{\prod\limits_{i=1}^{d}a_{11}^{ii}\det \mathbb{A}_{0}}\big(1+O(\rho_{0}^{-1}(d,m;\varepsilon)\rho_{d}(\varepsilon))\big),\quad\alpha=1,2,...,d,\\
\frac{\det\mathbb{A}_{2}^{\alpha}}{\det \mathbb{A}_{0}}\big(1+O(\rho_{0}^{-1}(d,m;\varepsilon))\big),\quad\quad\quad\alpha=d+1,...,\frac{d(d+1)}{2},
\end{cases}
%=&\frac{\det\mathbb{A}_{1}^{\ast\alpha}}{\det \mathbb{A}_{0}^{\ast}}\frac{1+O(\bar{\varepsilon}_{0}(d,m;\sigma))}{\mathcal{L}_{d}^{\alpha}\mathcal{M}_{0}\rho_{0}(d,m;\varepsilon)}.
\end{align}
where $\rho_{d}(\varepsilon)$ is defined by \eqref{LHAZ001}. Note that by using Theorem \ref{JGR} and Lemma \ref{lemmabc} again, we have
\begin{align*}
\det\mathbb{A}_{1}^{\alpha}=&\det\mathbb{A}^{\ast\alpha}_{1}+O(\varepsilon^{\frac{d+1-m}{12m}}),\quad\alpha=1,2,...,d,\\
\det\mathbb{A}_{2}^{\alpha}=&\det\mathbb{A}^{\ast\alpha}_{2}+O(\varepsilon^{\frac{d+1-m}{12m}}),\quad\alpha=d+1,...,\frac{d(d+1)}{2},\\
\det\mathbb{A}_{0}=&\det\mathbb{A}_{0}^{\ast}+O(\varepsilon^{\frac{d+1-m}{12m}}),
\end{align*}
which, in combination with \eqref{LAMNZ001}, leads to that for $\alpha=1,2,...,d,$
\begin{align}\label{ZKA001}
\frac{\tau_{1}^{\frac{d-1}{m}}|\det\mathbb{A}_{1}^{\ast\alpha}|}{|\mathcal{L}_{d}^{\alpha}|\det\mathbb{A}_{0}^{\ast}}\lesssim|C_{1}^{\alpha}-C_{2}^{\alpha}|\lesssim&\frac{\tau_{2}^{\frac{d-1}{m}}|\det\mathbb{A}_{1}^{\ast\alpha}|}{|\mathcal{L}_{d}^{\alpha}|\det\mathbb{A}_{0}^{\ast}},
\end{align}
and, for $\alpha=d+1,...,\frac{d(d+1)}{2}$,
\begin{align}\label{ZKA002}
C_{1}^{\alpha}-C_{2}^{\alpha}\simeq&\frac{\det\mathbb{A}_{2}^{\ast\alpha}}{\det \mathbb{A}_{0}^{\ast}}.
\end{align}

{\bf Case 3.} Consider $m<d-1$. For $\alpha=1,2,...,\frac{d(d+1)}{2}$, we substitute column vector $\big(\mathcal{B}_{1}[\varphi],...,\mathcal{B}_{\frac{d(d+1)}{2}}[\varphi]\big)^{T}$ for the elements of $\alpha$-th column in the matrix $\mathbb{A}$ defined in \eqref{GGDA01}, and then denote this new matrix by $\mathbb{A}_{3}^{\alpha}$ as follows:
\begin{gather*}
\mathbb{A}_{3}^{\alpha}=
\begin{pmatrix}
a_{11}^{11}&\cdots&\mathcal{B}_{1}[\varphi]&\cdots&a_{11}^{1\,\frac{d(d+1)}{2}} \\\\ \vdots&\ddots&\vdots&\ddots&\vdots\\\\a_{11}^{\frac{d(d+1)}{2}\,1}&\cdots&\mathcal{B}_{\frac{d(d+1)}{2}}[\varphi]&\cdots&a_{11}^{\frac{d(d+1)}{2}\,\frac{d(d+1)}{2}}
\end{pmatrix}.
\end{gather*}
Therefore, from Theorem \ref{JGR} and Lemma \ref{lemmabc}, we obtain that for $\alpha=1,2,...,\frac{d(d+1)}{2}$,
\begin{align*}
\det\mathbb{A}=\det\mathbb{A}^{\ast}+O(\varepsilon^{\min\{\frac{1}{6},\frac{d-1-m}{12m}\}}),\quad \det\mathbb{A}_{3}^{\alpha}=\det\mathbb{A}_{3}^{\ast\alpha}+O(\varepsilon^{\min\{\frac{1}{6},\frac{d-1-m}{12m}\}}),
\end{align*}
which, together with \eqref{JGRO001} and Cramer's rule, gives that
\begin{align}\label{GMARZT001}
C_{1}^{\alpha}-C_{2}^{\alpha}=&\frac{\det\mathbb{A}_{3}^{\alpha}}{\det \mathbb{A}}\simeq\frac{\det\mathbb{A}_{3}^{\ast\alpha}}{\det \mathbb{A}^{\ast}}.
\end{align}

\noindent{\bf Step 2.}
To begin with, it follows from Corollary \ref{coro00z} and \eqref{LGR001} that
\begin{align}\label{KGTA}
|\nabla u_{b}|=\left|\sum_{\alpha=1}^{\frac{d(d+1)}{2}}C_{2}^\alpha\nabla({v}_{1}^\alpha+{v}_{2}^\alpha)+\nabla{v}_{0}\right|\leq C\delta^{-\frac{d}{2}}e^{-\frac{1}{2C\delta^{1-1/m}}},\;\;\mathrm{in}\;\Omega_{R}.
\end{align}
In light of decomposition \eqref{Decom002}, \eqref{Le2.025} and \eqref{KGTA}, we divide into two cases to prove Theorem \ref{MGA001} as follows:

$(\mathrm{i})$ for $m>d$, recalling the assumed condition that there exists some integer $d+1\leq \alpha_{0}\leq\frac{d(d+1)}{2}$ such that $\det\mathbb{A}_{2}^{\ast\alpha_{0}}\neq0$, without loss of generality, we let $\alpha_{0}=d+1$. On one hand, from \eqref{HND001} and \eqref{ZKA001}--\eqref{ZKA002}, we obtain that for $x\in\{x'=(\sqrt[m]{\varepsilon},0,...,0)'\}\cap\Omega$,
\begin{align*}
|\nabla u|\lesssim&\sum_{\alpha=d+1}^{\frac{d(d+1)}{2}}|C_{1}^\alpha-C_{2}^\alpha||\nabla{v}_{1}^\alpha|+\left(\sum_{\alpha=1}^{d}|C_{1}^\alpha-C_{2}^\alpha||\nabla{v}_{1}^\alpha|+|\nabla u_{b}|\right)\notag\\
\lesssim&
\begin{cases} \frac{\max\limits_{d+1\leq\alpha\leq\frac{d(d+1)}{2}}\tau_{2}^{\frac{d+1}{m}}|\mathcal{L}_{d}^{\alpha}|^{-1}|\mathcal{B}^{\ast}_{\alpha}[\varphi]|}{1+\tau_{1}}\frac{\varepsilon^{\frac{1-m}{m}}}{\rho_{2}(d,m;\varepsilon)},&m\geq d+1,\\
\frac{\max\limits_{d+1\leq\alpha\leq\frac{d(d+1)}{2}}|\det\mathbb{A}_{2}^{\ast\alpha}|}{(1+\tau_{1})|\det\mathbb{A}_{0}^{\ast}|}\frac{1}{\varepsilon^{\frac{m-1}{m}}},&d<m<d+1.
\end{cases}
\end{align*}

On the other hand, by making use of \eqref{Le2.025}, we obtain for $x\in\{x'=(\sqrt[m]{\varepsilon},0,...,0)'\}\cap\Omega$,
\begin{align*}
|\partial_{x_{d}}(v_{1}^{d+1})_{d}|\geq|\partial_{x_{d}}(\bar{u}_{1}^{d+1})_{d}|-|\partial_{x_{d}}(v_{1}^{d+1}-\bar{u}_{1}^{d+1})_{d}|\geq\frac{1}{C\varepsilon^{1-1/m}},
\end{align*}
and, for $\alpha=d+2,...,\frac{d(d+1)}{2},\,d\geq3,$
\begin{align*}
|\partial_{x_{d}}(v_{1}^{\alpha})_{d}|\leq&|\partial_{x_{d}}(\bar{u}_{1}^{\alpha})_{d}|+|\partial_{x_{d}}(v_{1}^{\alpha}-\bar{u}_{1}^{\alpha})_{d}|=|\partial_{x_{d}}\mathcal{F}_{\alpha}^{d}|+|\partial_{x_{d}}(v_{1}^{\alpha}-\bar{u}_{1}^{\alpha})_{d}|\leq C,
\end{align*}
where $\mathcal{F}_{\alpha}^{d}$ is the $d$-th element of correction term $\mathcal{F}_{\alpha}$ defined by \eqref{QLA001}. This, together with \eqref{HND002} and \eqref{ZKA002}, leads to that
\begin{align}\label{KAZL002}
&\left|\sum_{\alpha=d+1}^{\frac{d(d+1)}{2}}(C_{1}^\alpha-C_{2}^\alpha)\nabla{v}_{1}^\alpha\right|\notag\\
&\gtrsim\left|(C_{1}^{d+1}-C_{2}^{d+1})\partial_{x_{d}}(v_{1}^{d+1})_{d}\right|-
\begin{cases}
0,&d=2,\\
\Big|\sum\limits^{\frac{d(d+1)}{2}}_{\alpha=d+2}(C^{\alpha}_{1}-C_{2}^{\alpha})\partial_{x_{d}}(v_{1}^{\alpha})_{d}\Big|,&d\geq3
\end{cases}\notag\\
&\gtrsim
\begin{cases}
\frac{\tau_{1}^{\frac{d+1}{m}}|\mathcal{B}^{\ast}_{d+1}[\varphi]|}{(1+\tau_{2})|\mathcal{L}_{d}^{d+1}|}\frac{\varepsilon^{\frac{1-m}{m}}}{\rho_{2}(d,m;\varepsilon)},&m\geq d+1,\\
\frac{|\det\mathbb{A}^{\ast\alpha_{0}}_{2}|}{(1+\tau_{2})|\det\mathbb{A}^{\ast}_{0}|}\frac{1}{\varepsilon^{\frac{m-1}{m}}},&d<m<d+1.
\end{cases}
\end{align}
Observe that by using \eqref{HND001} and \eqref{ZKA001}, we deduce that for $x\in\{x'=(\sqrt[m]{\varepsilon},0,...,0)'\}\cap\Omega$,
\begin{align}\label{KAZL001}
\sum_{\alpha=1}^{d}|C_{1}^\alpha-C_{2}^\alpha||\nabla{v}_{1}^\alpha|+|\nabla u_{b}|\lesssim\frac{1}{\varepsilon\rho_{0}(d,m;\varepsilon)},\quad m\geq d-1.
\end{align}
Then combining \eqref{KAZL002} and \eqref{KAZL001}, we get
\begin{align*}
|\nabla u|\gtrsim&\left|\sum_{\alpha=d+1}^{\frac{d(d+1)}{2}}(C_{1}^\alpha-C_{2}^\alpha)\nabla{v}_{1}^\alpha\right|-\left(\sum_{\alpha=1}^{d}|C_{1}^\alpha-C_{2}^\alpha||\nabla{v}_{1}^\alpha|+|\nabla u_{b}|\right)\notag\\
\gtrsim&
\begin{cases}
\frac{\tau_{1}^{\frac{d+1}{m}}|\mathcal{B}^{\ast}_{d+1}[\varphi]|}{(1+\tau_{2})|\mathcal{L}_{d}^{d+1}|}\frac{\varepsilon^{\frac{1-m}{m}}}{\rho_{2}(d,m;\varepsilon)},&m\geq d+1,\\
\frac{|\det\mathbb{A}^{\ast d+1}_{2}|}{(1+\tau_{2})|\det\mathbb{A}^{\ast}_{0}|}\frac{1}{\varepsilon^{\frac{m-1}{m}}},&d<m<d+1;
\end{cases}
\end{align*}

$(\mathrm{ii})$ for $m\leq d$, then we deduce from \eqref{ZKA001}--\eqref{GMARZT001} that for $x\in\{x'=0'\}\cap\Omega$,
\begin{align*}
|\nabla u|\lesssim&\sum_{\alpha=1}^{d}|C_{1}^\alpha-C_{2}^\alpha||\nabla{v}_{1}^\alpha|+\left(\sum_{\alpha=d+1}^{\frac{d(d+1)}{2}}|C_{1}^\alpha-C_{2}^\alpha||\nabla{v}_{1}^\alpha|+|\nabla u_{b}|\right)\notag\\
\lesssim&
\begin{cases}
\frac{\max\limits_{1\lesssim\alpha\lesssim d}\tau_{2}^{\frac{d-1}{m}}|\mathcal{L}_{d}^{\alpha}|^{-1}|\det\mathbb{A}_{1}^{\ast\alpha}|}{|\det\mathbb{A}_{0}^{\ast}|}\frac{1}{\varepsilon\rho_{0}(d,m;\varepsilon)},&d-1\leq m\leq d,\\
\frac{\max\limits_{1\leq\alpha\leq d}|\det\mathbb{A}_{3}^{\ast\alpha}|}{|\det \mathbb{A}^{\ast}|}\frac{1}{\varepsilon},&m<d-1,
\end{cases}
\end{align*}
and
\begin{align*}
|\nabla u|\geq&\left|\sum^{d}_{\alpha=1}(C^{\alpha}_{1}-C_{2}^{\alpha})\nabla v_{1}^{\alpha}\right|-\left(\sum^{\frac{d(d+1)}{2}}_{\alpha=d+1}|C^{\alpha}_{1}-C_{2}^{\alpha}||\nabla v_{1}^{\alpha}|+|\nabla u_{b}|\right)\notag\\
\geq&\left|\sum^{d}_{\alpha=1}(C^{\alpha}_{1}-C_{2}^{\alpha})\partial_{x_{d}} (v_{1}^{\alpha})_{\alpha_{0}}\right|-C\notag\\
\gtrsim&
\begin{cases}
\frac{\tau_{1}^{\frac{d-1}{m}}|\det\mathbb{A}_{1}^{\ast\alpha_{0}}|}{|\mathcal{L}_{d}^{\alpha_{0}}||\det\mathbb{A}_{0}^{\ast}|}\frac{1}{\varepsilon\rho_{0}(d,m;\varepsilon)},&d-1\leq m\leq d,\\
\frac{|\det\mathbb{A}_{3}^{\ast\alpha_{0}}|}{|\det \mathbb{F}^{\ast}|}\frac{1}{\varepsilon},&m<d-1,
\end{cases}
\end{align*}
where we used the fact that for $x\in\{x'=0'\}\cap\Omega$,
\begin{align*}
\sum^{\frac{d(d+1)}{2}}_{\alpha=d+1}|C^{\alpha}_{1}-C_{2}^{\alpha}||\nabla v_{1}^{\alpha}|+|\nabla u_{b}|\leq C,
\end{align*}
and
$$|\partial_{x_{d}}(v_{1}^{\alpha_{0}})_{\alpha_{0}}|\geq|\partial_{x_{d}}(\bar{u}_{1}^{\alpha_{0}})_{\alpha_{0}}|-|\partial_{x_{d}}(v_{1}^{\alpha_{0}}-\bar{u}_{1}^{\alpha_{0}})_{\alpha_{0}}|\geq\frac{1}{C\varepsilon},$$
and, if $\alpha=1,2,...,d,\,\alpha\neq\alpha_{0}$,
$$|\partial_{x_{d}}(v_{1}^{\alpha})_{\alpha_{0}}|\leq|\partial_{x_{d}}\mathcal{F}_{\alpha}^{\alpha_{0}}|+|\partial_{x_{d}}(v_{1}^{\alpha}-\bar{u}_{1}^{\alpha})_{\alpha_{0}}|\leq C.$$

Then combining the results above, we complete the proof of Theorem \ref{MGA001}.

\end{proof}

\subsection{Gradient asymptotics}
In this section, we aim to establish the asymptotic expansions of the stress concentration in the presence of two adjacent $m$-convex inclusions with different principal curvatures as follows:
\begin{align}\label{ZCZ009}
h_{1}(x')-h_{2}(x')=\sum^{2}_{i=1}\tau_{i}|x_{i}|^{m},\quad\mathrm{in}\;\Omega_{R},
\end{align}
where $\tau_{i}$, $i=1,2,$ are two positive constants independent of $\varepsilon$. For $i=0,2,$ denote
\begin{align*}
\Gamma\Big[\frac{i+2}{m}\Big]:=&
\begin{cases}
\Gamma(1-\frac{i+2}{m})\Gamma(\frac{i+2}{m}),&m>i+2,\\
1,&m=i+2,
\end{cases}
\end{align*}
where $\Gamma(s)=\int^{+\infty}_{0}t^{s-1}e^{-t}dt$, $s>0$ is the Gamma function. Introduce the definite constants as follows:
\begin{align}\label{WEN}
\mathcal{M}_{1}=\frac{2\pi\Gamma[\frac{2}{m}]}{m\tau},\quad\mathcal{M}_{2}=\frac{\pi\Gamma[\frac{4}{m}]}{m\tau^{2}},
\end{align}
where $\tau=\sqrt[m]{\tau_{1}\tau_{2}}$. Then we have
\begin{theorem}\label{coro00389}
Let $D_{1},D_{2}\subset D\subset\mathbb{R}^{3}$ be defined as above, conditions \eqref{ZCZ009} and $\mathrm{(}${\bf{H2}}$\mathrm{)}$--$\mathrm{(}${\bf{H3}}$\mathrm{)}$ hold. Let $u\in H^{1}(D;\mathbb{R}^{3})\cap C^{1}(\overline{\Omega};\mathbb{R}^{3})$ be the solution of \eqref{La.002}. Then for a sufficiently small $\varepsilon>0$ and $x\in\Omega_{R}$,

$(\rm{i})$ for $m>4$, if $\det\mathcal{B}_{\alpha}^{\ast}[\varphi]\neq0$, $\alpha=1,2,...6$,
\begin{align*}
\nabla u=&\sum^{3}_{\alpha=1}\frac{m\pi\mathcal{B}_{\alpha}^{\ast}[\varphi]\varepsilon^{\frac{m-2}{m}}}{4\mathcal{L}_{3}^{\alpha}\mathcal{M}_{1}\int^{\frac{\pi}{2}}_{0}E(\theta)d\theta}\frac{1+O(\varepsilon^{\min\{\frac{1}{4},\frac{m-4}{m}\}})}{1+\mathcal{G}^{\ast\alpha}_{m}\varepsilon^{\frac{m-2}{m}}}\nabla\bar{u}_{1}^{\alpha}\notag\\
&+\sum^{6}_{\alpha=4}\frac{m\pi\mathcal{B}_{\alpha}^{\ast}[\varphi]\varepsilon^{\frac{m-4}{m}}}{4\mathcal{L}_{3}^{\alpha}\mathcal{M}_{2}\int^{\frac{\pi}{2}}_{0}E(\theta)F(\theta)d\theta}\frac{1+O(\varepsilon^{\min\{\frac{1}{4},\frac{m-4}{m}\}})}{1+\mathcal{G}^{\ast\alpha}_{m}\varepsilon^{\frac{m-4}{m}}}\nabla\bar{u}_{1}^{\alpha}+O(1)\|\varphi\|_{C^{0}(\partial D)};
\end{align*}

$(\rm{ii})$ for $m=4$, if $\det\mathcal{B}_{\alpha}^{\ast}[\varphi]\neq0$, $\alpha=1,2,...6$,
\begin{align*}
\nabla u=&\sum^{3}_{\alpha=1}\frac{\pi\mathcal{B}_{\alpha}^{\ast}[\varphi]\sqrt{\varepsilon}}{\mathcal{L}_{3}^{\alpha}\mathcal{M}_{1}\int^{\frac{\pi}{2}}_{0}E(\theta)d\theta}\frac{1+O(|\ln\varepsilon|^{-1})}{1+\mathcal{G}^{\ast\alpha}_{4}\sqrt{\varepsilon}}\nabla\bar{u}_{1}^{\alpha}\notag\\
&+\sum^{6}_{\alpha=4}\frac{\pi\mathcal{B}_{\alpha}^{\ast}[\varphi]}{\mathcal{L}_{3}^{\alpha}\mathcal{M}_{2}\int^{\frac{\pi}{2}}_{0}E(\theta)F(\theta)d\theta}\frac{1+O(\varepsilon^{\frac{1}{6m}}|\ln\varepsilon|)}{|\ln\varepsilon|+\mathcal{G}^{\ast\alpha}_{4}}\nabla\bar{u}_{1}^{\alpha}+O(1)\|\varphi\|_{C^{0}(\partial D)};
\end{align*}

$(\rm{iii})$ for $2<m<4$, if $\det\mathbb{A}_{0}^{\ast}\neq0$, $\det\mathbb{A}_{1}^{\ast\alpha}\neq0$, $\alpha=1,2,3$, and $\det\mathbb{A}_{2}^{\ast\alpha}\neq0$, $\alpha=4,5,6$,
\begin{align*}
\nabla u=&\sum^{3}_{\alpha=1}\frac{\det\mathbb{A}_{1}^{\ast\alpha}}{\det\mathbb{A}_{0}^{\ast}}\frac{m\pi\varepsilon^{\frac{m-2}{m}}}{4\mathcal{L}_{3}^{\alpha}\mathcal{M}_{1}\int^{\frac{\pi}{2}}_{0}E(\theta)d\theta}\frac{1+O(\varepsilon^{\min\{\frac{m-2}{m},\frac{4-m}{12m}\}})}{1+\mathcal{G}^{\ast\alpha}_{m}\varepsilon^{\frac{m-2}{m}}}\nabla\bar{u}_{1}^{\alpha}\notag\\
&+\sum^{6}_{\alpha=4}\frac{\det\mathbb{A}_{2}^{\ast\alpha}}{\det \mathbb{A}_{0}^{\ast}}(1+O(\varepsilon^{\min\{\frac{m-2}{m},\frac{4-m}{12m}\}}))\nabla\bar{u}_{1}^{\alpha}+O(1)\|\varphi\|_{C^{0}(\partial D)};
\end{align*}

$(\rm{iv})$ for $m=2$, if $\det\mathbb{A}_{0}^{\ast}\neq0$, $\det\mathbb{A}_{1}^{\ast\alpha}\neq0$, $\alpha=1,2,3$, and $\det\mathbb{A}_{2}^{\ast\alpha}\neq0$, $\alpha=4,5,6$,
\begin{align*}
\nabla u=&\sum^{3}_{\alpha=1}\frac{\det\mathbb{A}_{1}^{\ast\alpha}}{\det\mathbb{A}_{0}^{\ast}}\frac{\pi}{2\mathcal{L}_{3}^{\alpha}\mathcal{M}_{1}\int^{\frac{\pi}{2}}_{0}E(\theta)d\theta}\frac{1+O(|\ln\varepsilon|^{-1})}{|\ln\varepsilon|+\mathcal{G}^{\ast\alpha}_{2}}\nabla\bar{u}_{1}^{\alpha}\notag\\
&+\sum^{6}_{\alpha=4}\frac{\det\mathbb{A}_{2}^{\ast\alpha}}{\det \mathbb{A}_{0}^{\ast}}(1+O(|\ln\varepsilon|^{-1}))\nabla\bar{u}_{1}^{\alpha}+O(1)\|\varphi\|_{C^{0}(\partial D)},
\end{align*}
where the explicit auxiliary functions $\bar{u}_{1}^{\alpha}$, $\alpha=1,2,...,6$ are defined in \eqref{zzwz002} with $d=3$, the constants $\mathcal{M}_{i}$, $i=1,2$, are defined by \eqref{WEN}, the Lam\'{e} constants $\mathcal{L}_{3}^{\alpha}$, $\alpha=1,2,3$ are defined by \eqref{AZ110} with $d=3$, the blow-up factors $\mathcal{B}_{\alpha}^{\ast}[\varphi]$, $\alpha=1,2,...,6$, $\mathbb{A}_{0}^{\ast}$, $\mathbb{A}_{1}^{\ast\alpha}$, $\alpha=1,2,3$, and $\mathbb{A}_{2}^{\ast\alpha}$, $\alpha=4,5,6$ are, respectively, defined by \eqref{KGF001}, \eqref{LAGT001} and \eqref{PTCA001}--\eqref{DKY001} with $d=3$, $E(\theta)$ and $F(\theta)$ are, respectively, defined by \eqref{LRANM001} and \eqref{LRANM002}, the geometry constants $\mathcal{G}_{m}^{\ast\alpha}$, $\alpha=1,2,...,6$ are defined by \eqref{ZWZWZW0019}--\eqref{ZWZLH}.
\end{theorem}
\begin{remark}
The case of $m=2$ in condition \eqref{ZCZ009} corresponds to the strictly convex inclusions, which were studied in the previous work \cite{LLY2019} for the perfect conductivity equation.
%Since the major objective of this paper is to capture a family of unified stress concentration factors in Theorem \ref{JGR}, then we restrict ourselves to the setup above for the sake of presentation and readability.
\end{remark}

\begin{proof}
\noindent{\bf Step 1.}
Since the computational results for the off-diagonal elements of coefficient matrix $\mathbb{A}$ defined by \eqref{GGDA01} under condition \eqref{ZCZ009} are the same to that of Lemma \ref{lemmabc} with $d=3$, then it suffices to give the calculations for the principal diagonal elements. Following the same argument as in \eqref{KATZ001} and \eqref{QZH001}, we obtain that for $\alpha=1,2,3$,
\begin{align}\label{KATZ00100}
a_{11}^{\alpha\alpha}=&\mathcal{L}_{3}^{\alpha}\left(\int_{\varepsilon^{\frac{1}{12m}}<|x'|<R}\frac{dx'}{h_{1}(x')-h_{2}(x')}+\int_{|x'|<\varepsilon^{\frac{1}{12m}}}\frac{dx'}{\varepsilon+h_{1}(x')-h_{2}(x')}\right)\notag\\
&+M_{3}^{\ast\alpha}+O(1)\varepsilon^{\frac{1}{6m}},
\end{align}
and, for $\alpha=4,5,6$,
\begin{align}\label{GRC001}
a_{11}^{\alpha\alpha}=&\frac{\mathcal{L}_{3}^{\alpha}}{2}\left(\int_{\varepsilon^{\frac{1}{12m}}<|x'|<R}\frac{|x'|^{2}}{h_{1}(x')-h_{2}(x')}+\int_{|x'|<\varepsilon^{\frac{1}{12m}}}\frac{|x'|^{2}}{\varepsilon+h_{1}(x')-h_{2}(x')}\right)\nonumber\\
&+M_{3}^{\ast\alpha}+O(1)\varepsilon^{\frac{1}{6m}},
\end{align}
where $M_{3}^{\ast\alpha}$, $\alpha=1,2,...,6$ are defined by \eqref{FT001} with $d=3$.

{\bf Case 1.} In the case of $\alpha=1,2,3$, using \eqref{ZCZ009}, it follows from a direct computation that
\begin{align}\label{HTZ}
&\int_{\varepsilon^{\frac{1}{12m}}<|x'|<R}\frac{dx'}{\sum^{2}_{i=1}\tau_{i}|x_{i}|^{m}}+\int_{|x'|<\varepsilon^{\frac{1}{12m}}}\frac{dx'}{\varepsilon+\sum^{2}_{i=1}\tau_{i}|x_{i}|^{m}}\notag\\
&=\int_{|x'|<R}\frac{1}{\varepsilon+\sum^{2}_{i=1}\tau_{i}|x_{i}|^{m}}+\int_{\varepsilon^{\frac{1}{12m}}<|x'|<R}\frac{\varepsilon}{\sum^{2}_{i=1}\tau_{i}|x_{i}|^{m}(\varepsilon+\sum^{2}_{i=1}\tau_{i}|x_{i}|^{m})}\notag\\
&=\int_{|x'|<R}\frac{1}{\varepsilon+\sum^{2}_{i=1}\tau_{i}|x_{i}|^{m}}+O(1)\varepsilon^{\frac{5m+1}{6m}}\notag\\
&=\frac{8}{m\sqrt[m]{\tau_{1}\tau_{2}}}\int_{0}^{\frac{\pi}{2}}E(\theta)\int_{0}^{R(\theta)}\frac{t}{\varepsilon+t^{m}}\,dtd\theta+O(1)\varepsilon^{\frac{5m+1}{6m}}\notag\\
&=
\begin{cases}
\frac{\pi}{\sqrt{\tau_{1}\tau_{2}}}|\ln\varepsilon|+\frac{4\int^{\frac{\pi}{2}}_{0}\ln R(\theta)d\theta}{\sqrt{\tau_{1}\tau_{2}}}+O(1)\varepsilon^{\frac{5m+1}{6m}},&m=2,\\
\frac{8\Gamma[\frac{2}{m}]\int^{\frac{\pi}{2}}_{0}E(\theta)d\theta}{m^{2}\sqrt[m]{\tau_{1}\tau_{2}}}\frac{1}{\varepsilon^{1-2/m}}-\frac{8\int^{\frac{\pi}{2}}_{0}E(\theta)(R(\theta))^{2-m}d\theta}{m(m-2)\sqrt[m]{\tau_{1}\tau_{2}}}+O(1)\varepsilon^{\frac{5m+1}{6m}},&m>2,
\end{cases}
\end{align}
where
\begin{align}
E(\theta)=&(\sin\theta)^{2/m-1}(\cos\theta)^{2/m+1}+(\sin\theta)^{2/m+1}(\cos\theta)^{2/m-1},\label{LRANM001}\\
R(\theta)=&R\big(\tau_{1}^{-2/m}(\cos\theta)^{4/m}+\tau_{2}^{-2/m}(\sin\theta)^{4/m}\big)^{-1/2}.\label{KN}
\end{align}
Then substituting \eqref{HTZ} into \eqref{KATZ00100}, we obtain that for $\alpha=1,2,3$,
\begin{align}\label{LMT01}
a_{11}^{\alpha\alpha}=&\frac{4\mathcal{L}_{3}^{\alpha}\mathcal{M}_{1}\int^{\frac{\pi}{2}}_{0}E(\theta)d\theta}{m\pi}\rho_{0}(3,m;\varepsilon)+\mathcal{K}_{m}^{\ast\alpha}+O(1)\varepsilon^{\frac{1}{6m}},
\end{align}
where $\mathcal{M}_{1}$ is defined in \eqref{WEN}, and for $\alpha=1,2,3,$
\begin{align}\label{LKM}
\mathcal{K}_{m}^{\ast\alpha}=&
\begin{cases}
M_{3}^{\ast\alpha}+\frac{4\mathcal{L}_{3}^{\alpha}\int^{\frac{\pi}{2}}_{0}\ln R(\theta)d\theta}{\sqrt{\tau_{1}\tau_{2}}},&m=2,\\
M_{3}^{\ast\alpha}-\frac{8\mathcal{L}_{3}^{\alpha}}{m(m-2)\sqrt[m]{\tau_{1}\tau_{2}}}\int^{\frac{\pi}{2}}_{0}E(\theta)(R(\theta))^{2-m}d\theta,&m>2.
\end{cases}
\end{align}

{\bf Case 2.} We now calculate $a_{11}^{\alpha\alpha}$ for $\alpha=4,5,6$. To begin with, if $2\leq m<4$, a direct application of \eqref{LMC1} yields that
\begin{align*}
a_{11}^{\alpha\alpha}=a_{11}^{\ast\alpha\alpha}+O(\varepsilon^{\frac{4-m}{12m}}).
\end{align*}

On the other hand, if $m\geq4$, similar to \eqref{HTZ}, it follows from a direct calculation that
\begin{align*}
&\int_{\varepsilon^{\frac{1}{12m}}<|x'|<R}\frac{|x'|^{2}}{\sum^{2}_{i=1}\tau_{i}|x_{i}|^{m}}+\int_{|x'|<\varepsilon^{\frac{1}{12m}}}\frac{|x'|^{2}}{\varepsilon+\sum^{2}_{i=1}\tau_{i}|x_{i}|^{m}}\notag\\
&=\int_{|x'|<R}\frac{|x'|^{2}}{\varepsilon+\sum^{2}_{i=1}\tau_{i}|x_{i}|^{m}}+\int_{\varepsilon^{\frac{1}{12m}}<|x'|<R}\frac{\varepsilon|x'|^{2}}{\sum^{2}_{i=1}\tau_{i}|x_{i}|^{m}(\varepsilon+\sum^{2}_{i=1}\tau_{i}|x_{i}|^{m})}\notag\\
&=\int_{|x'|<R}\frac{|x'|^{2}}{\varepsilon+\sum^{2}_{i=1}\tau_{i}|x_{i}|^{m}}+O(1)\varepsilon^{\frac{5m+2}{6m}}\notag\\
&=\frac{8}{m\sqrt[2m]{\tau_{1}\tau_{2}}}\int_{0}^{\frac{\pi}{2}}E(\theta)F(\theta)\int_{0}^{R(\theta)}\frac{t^{3}}{\varepsilon+t^{m}}\,dtd\theta+O(1)\varepsilon^{\frac{5m+2}{6m}}\notag\\
&=
\begin{cases}
\frac{\int^{\frac{\pi}{2}}_{0}E(\theta)F(\theta)d\theta}{2\sqrt[8]{\tau_{1}\tau_{2}}}|\ln\varepsilon|+\frac{2}{\sqrt[8]{\tau_{1}\tau_{2}}}\int^{\frac{\pi}{2}}_{0}E(\theta)F(\theta)\ln R(\theta)d\theta+O(1)\varepsilon^{\frac{11}{12}},&m=4,\\
\frac{8\Gamma[\frac{4}{m}]\int^{\frac{\pi}{2}}_{0}E(\theta)F(\theta)d\theta}{m^{2}\sqrt[2m]{\tau_{1}\tau_{2}}}\frac{1}{\varepsilon^{\frac{m-4}{m}}}-\frac{8\int^{\frac{\pi}{2}}_{0}E(\theta)F(\theta)(R(\theta))^{4-m}d\theta}{m(m-4)\sqrt[2m]{\tau_{1}\tau_{2}}}+O(1)\varepsilon^{\frac{5m+2}{6m}},&m>4,
\end{cases}
\end{align*}
where $E(\theta)$ and $R(\theta)$ are, respectively, defined by \eqref{LRANM001}--\eqref{KN},
\begin{align}\label{LRANM002}
F(\theta)=\sqrt[m]{\frac{\tau_{2}}{\tau_{1}}}(\cos\theta)^{4/m}+\sqrt[m]{\frac{\tau_{1}}{\tau_{2}}}(\sin\theta)^{4/m}.
\end{align}
This, together with \eqref{GRC001}, reads that
\begin{align}\label{LMT0111}
a_{11}^{\alpha\alpha}=&\frac{4\mathcal{L}_{3}^{\alpha}\mathcal{M}_{2}\int^{\frac{\pi}{2}}_{0}E(\theta)F(\theta)d\theta}{m\pi}\rho_{2}(3,m;\varepsilon)+\mathcal{K}_{m}^{\ast\alpha}+O(1)\varepsilon^{\frac{1}{6m}},
\end{align}
where $\mathcal{M}_{2}$ is defined in \eqref{WEN}, and for $\alpha=4,5,6,$
\begin{align}\label{LKM001}
\mathcal{K}_{m}^{\ast\alpha}=&
\begin{cases}
M^{\ast\alpha}_{3}+\frac{\mathcal{L}_{3}^{\alpha}}{\sqrt[8]{\tau_{1}\tau_{2}}}\int^{\frac{\pi}{2}}_{0}E(\theta)F(\theta)\ln R(\theta)d\theta,&m=4,\\
M^{\ast\alpha}_{3}-\frac{4\mathcal{L}_{3}^{\alpha}\int^{\frac{\pi}{2}}_{0}E(\theta)F(\theta)(R(\theta))^{4-m}d\theta}{m(m-4)\sqrt[2m]{\tau_{1}\tau_{2}}},&m>4.
\end{cases}
\end{align}

We now demonstrate that the constants $\mathcal{K}_{m}^{\ast\alpha}$, $\alpha=1,2,...,6$ captured in \eqref{LKM} and \eqref{LKM001} are actually independent of the length parameter $R$ of the narrow region. If not, suppose that there exist $\mathcal{K}_{m}^{\ast\alpha}(r_{1}^{\ast})$ and $\mathcal{K}_{m}^{\ast\alpha}(r_{2}^{\ast})$, $r_{i}^{\ast}>0,\,i=1,2,\,r_{1}^{\ast}\neq r_{2}^{\ast}$, both independent of $\varepsilon$, such that \eqref{LMT01} and \eqref{LMT0111} hold. Then, we have
\begin{align*}
\mathcal{K}_{m}^{\ast\alpha}(r_{1}^{\ast})-\mathcal{K}_{m}^{\ast\alpha}(r_{2}^{\ast})=O(1)\varepsilon^{\frac{1}{12m}},
\end{align*}
which yields that $\mathcal{K}_{m}^{\ast\alpha}(r_{1}^{\ast})=\mathcal{K}_{m}^{\ast\alpha}(r_{2}^{\ast})$.

\noindent{\bf Step 2.} Denote
\begin{align}\label{ZWZWZW0019}
\mathcal{G}^{\ast\alpha}_{m}=\frac{m\pi\mathcal{K}^{\ast\alpha}_{m}}{4\mathcal{L}_{3}^{\alpha}\mathcal{M}_{1}\int^{\frac{\pi}{2}}_{0}E(\theta)d\theta},\quad\mathrm{for}\;\alpha=1,2,3,\;m\geq2,
\end{align}
and
\begin{align}\label{ZWZLH}
\mathcal{G}^{\ast\alpha}_{m}=\frac{m\pi\mathcal{K}^{\ast\alpha}_{m}}{4\mathcal{L}_{3}^{\alpha}\mathcal{M}_{2}\int^{\frac{\pi}{2}}_{0}E(\theta)F(\theta)d\theta},\quad\mathrm{for}\;\alpha=4,5,6,\;m\geq4.
\end{align}
Then it follows from \eqref{LMT01} and \eqref{LMT0111} that for $\alpha=1,2,3$, if $m\geq2$,
\begin{align}\label{AZW01}
\frac{1}{a_{11}^{\alpha\alpha}}=&\frac{\rho^{-1}_{0}(3,m;\varepsilon)}{\frac{\mathcal{K}^{\ast\alpha}_{m}}{\mathcal{G}^{\ast\alpha}_{m}}}\frac{1}{1-\frac{\frac{\mathcal{K}^{\ast\alpha}_{m}}{\mathcal{G}^{\ast\alpha}_{m}}-\rho^{-1}_{0}(3,m;\varepsilon)a_{11}^{\alpha\alpha}}{\frac{\mathcal{K}^{\ast\alpha}_{m}}{\mathcal{G}^{\ast\alpha}_{m}}}}\notag\\
=&\frac{\rho^{-1}_{0}(3,m;\varepsilon)}{\frac{\mathcal{K}^{\ast\alpha}_{m}}{\mathcal{G}^{\ast\alpha}_{m}}}\frac{1}{1+\mathcal{G}^{\ast\alpha}_{m}\rho^{-1}_{0}(3,m;\varepsilon)+O(\varepsilon^{\frac{1}{6m}}\rho_{0}^{-1}(3,m;\varepsilon))}\notag\\
=&
\begin{cases}
\frac{\pi}{2\mathcal{L}_{3}^{\alpha}\mathcal{M}_{1}\int^{\frac{\pi}{2}}_{0}E(\theta)d\theta}\frac{1+O(\varepsilon^{\frac{1}{12}}|\ln\varepsilon|)}{|\ln\varepsilon|+\mathcal{G}^{\ast\alpha}_{2}},&m=2,\\
\frac{m\pi\varepsilon^{\frac{m-2}{m}}}{4\mathcal{L}_{3}^{\alpha}\mathcal{M}_{1}\int^{\frac{\pi}{2}}_{0}E(\theta)d\theta}\frac{1+O(\varepsilon^{\frac{6m-11}{6m}})}{1+\mathcal{G}^{\ast\alpha}_{m}\varepsilon^{\frac{m-2}{m}}},&m>2,
\end{cases}
\end{align}
and for $\alpha=4,5,6$, if $m\geq 4$,
\begin{align}\label{AZW02}
\frac{1}{a_{11}^{\alpha\alpha}}=&\frac{\rho^{-1}_{2}(3,m;\varepsilon)}{\frac{\mathcal{K}^{\ast\alpha}_{m}}{\mathcal{G}^{\ast\alpha}_{m}}}\frac{1}{1-\frac{\frac{\mathcal{K}^{\ast\alpha}_{m}}{\mathcal{G}^{\ast\alpha}_{m}}-\rho^{-1}_{2}(3,m;\varepsilon)a_{11}^{\alpha\alpha}}{\frac{\mathcal{K}^{\ast\alpha}_{m}}{\mathcal{G}^{\ast\alpha}_{m}}}}\notag\\
=&\frac{\rho^{-1}_{2}(3,m;\varepsilon)}{\frac{\mathcal{K}^{\ast\alpha}_{m}}{\mathcal{G}^{\ast\alpha}_{m}}}\frac{1}{1+\mathcal{G}^{\ast\alpha}_{m}\rho^{-1}_{2}(3,m;\varepsilon)+O(\varepsilon^{\frac{1}{12m}}\rho_{2}^{-1}(3,m;\varepsilon))}\notag\\
=&
\begin{cases}
\frac{\pi}{\mathcal{L}_{3}^{\alpha}\mathcal{M}_{2}\int^{\frac{\pi}{2}}_{0}E(\theta)F(\theta)d\theta}\frac{1+O(\varepsilon^{\frac{1}{6m}}|\ln\varepsilon|)}{|\ln\varepsilon|+\mathcal{G}^{\ast\alpha}_{4}},&m=4,\\
\frac{m\pi\varepsilon^{\frac{m-4}{m}}}{4\mathcal{L}_{3}^{\alpha}\mathcal{M}_{2}\int^{\frac{\pi}{2}}_{0}E(\theta)F(\theta)d\theta}\frac{1+O(\varepsilon^{\frac{6m-23}{6m}})}{1+\mathcal{G}^{\ast\alpha}_{m}\varepsilon^{\frac{m-4}{m}}},&m>4.
\end{cases}
\end{align}
Then substituting \eqref{AHNZ001} and \eqref{AZW01}--\eqref{AZW02} into \eqref{MNT001} and \eqref{LAMNZ001}, we obtain that

$(\rm{i})$ if $m\geq4$, then
\begin{align}\label{ZWWWZ001}
C_{1}^{\alpha}-C_{2}^{\alpha}=&\frac{\mathcal{B}_{\alpha}[\varphi]}{a_{11}^{\alpha\alpha}}(1+O(\rho_{2}^{-1}(3,m;\varepsilon)))\notag\\
=&
\begin{cases}
\frac{m\pi\mathcal{B}_{\alpha}^{\ast}[\varphi]\varepsilon^{\frac{m-2}{m}}}{4\mathcal{L}_{3}^{\alpha}\mathcal{M}_{1}\int^{\frac{\pi}{2}}_{0}E(\theta)d\theta}\frac{1+O(\varepsilon^{\min\{\frac{1}{4},\frac{m-4}{m}\}})}{1+\mathcal{G}^{\ast\alpha}_{m}\varepsilon^{\frac{m-2}{m}}},&\alpha=1,2,3,\,m>4,\\
\frac{\pi\mathcal{B}_{\alpha}^{\ast}[\varphi]\sqrt{\varepsilon}}{\mathcal{L}_{3}^{\alpha}\mathcal{M}_{1}\int^{\frac{\pi}{2}}_{0}E(\theta)d\theta}\frac{1+O(|\ln\varepsilon|^{-1})}{1+\mathcal{G}^{\ast\alpha}_{4}\sqrt{\varepsilon}},&\alpha=1,2,3,\,m=4,\\
\frac{m\pi\mathcal{B}_{\alpha}^{\ast}[\varphi]\varepsilon^{\frac{m-4}{m}}}{4\mathcal{L}_{3}^{\alpha}\mathcal{M}_{2}\int^{\frac{\pi}{2}}_{0}E(\theta)F(\theta)d\theta}\frac{1+O(\varepsilon^{\min\{\frac{1}{4},\frac{m-4}{m}\}})}{1+\mathcal{G}^{\ast\alpha}_{m}\varepsilon^{\frac{m-4}{m}}},&\alpha=4,5,6,\,m>4,\\
\frac{\pi\mathcal{B}_{\alpha}^{\ast}[\varphi]}{\mathcal{L}_{3}^{\alpha}\mathcal{M}_{2}\int^{\frac{\pi}{2}}_{0}E(\theta)F(\theta)d\theta}\frac{1+O(\varepsilon^{\frac{1}{6m}}|\ln\varepsilon|)}{|\ln\varepsilon|+\mathcal{G}^{\ast\alpha}_{4}},&\alpha=4,5,6,\,m=4;
\end{cases}
\end{align}

$(\rm{ii})$ if $2\leq m<4$, then for $\alpha=1,2,3$,
\begin{align}\label{ZWWWZ002}
C_{1}^{\alpha}-C_{2}^{\alpha}=&\frac{\prod\limits_{i\neq\alpha}^{3}a_{11}^{ii}\det\mathbb{A}_{1}^{\alpha}}{\prod\limits_{i=1}^{3}a_{11}^{ii}\det \mathbb{A}_{0}}(1+O(\rho_{0}^{-1}(3,m;\varepsilon)))\notag\\
=&
\begin{cases}
\frac{\det\mathbb{A}_{1}^{\ast\alpha}}{\det\mathbb{A}_{0}^{\ast}}\frac{m\pi\varepsilon^{\frac{m-2}{m}}}{4\mathcal{L}_{3}^{\alpha}\mathcal{M}_{1}\int^{\frac{\pi}{2}}_{0}E(\theta)d\theta}\frac{1+O(\varepsilon^{\min\{\frac{m-2}{m},\frac{4-m}{12m}\}})}{1+\mathcal{G}^{\ast\alpha}_{m}\varepsilon^{\frac{m-2}{m}}},&2<m<4,\\
\frac{\det\mathbb{A}_{1}^{\ast\alpha}}{\det\mathbb{A}_{0}^{\ast}}\frac{\pi}{2\mathcal{L}_{3}^{\alpha}\mathcal{M}_{1}\int^{\frac{\pi}{2}}_{0}E(\theta)d\theta}\frac{1+O(|\ln\varepsilon|^{-1})}{|\ln\varepsilon|+\mathcal{G}^{\ast\alpha}_{2}},&m=2,
\end{cases}
\end{align}
and, for $\alpha=4,5,6$, a consequence of \eqref{LAMNZ001} yields that
\begin{align}\label{ZWWWZ003}
C_{1}^{\alpha}-C_{2}^{\alpha}=&\frac{\det\mathbb{A}_{2}^{\alpha}}{\det \mathbb{A}_{0}}(1+O(\rho_{0}^{-1}(3,m;\varepsilon)))\notag\\
=&
\begin{cases}
\frac{\det\mathbb{A}_{2}^{\ast\alpha}}{\det \mathbb{A}_{0}^{\ast}}\big(1+O(\varepsilon^{\min\{\frac{m-2}{m},\frac{4-m}{12m}\}})\big),&2<m<4,\\
\frac{\det\mathbb{A}_{2}^{\ast\alpha}}{\det \mathbb{A}_{0}^{\ast}}(1+O(|\ln\varepsilon|^{-1})),&m=2.
\end{cases}
\end{align}

Consequently, in light of decomposition \eqref{Decom002}, it follows from Corollary \ref{thm86}, \eqref{KGTA} and \eqref{ZWWWZ001}--\eqref{ZWWWZ003} that Theorem \ref{coro00389} holds.

\end{proof}

\section{The perfect conductivity problem}\label{SECCC000}
In this section, we consider the corresponding scalar case, that is, the perfect conductivity equation as follows:
\begin{align}\label{con002}
\begin{cases}
\Delta u=0,&\hbox{in}\;\Omega,\\
u=C_{i}, &\hbox{on}\;\partial D_{i},\;i=1,2,\\
\int_{\partial D_{i}}\frac{\partial u}{\partial\nu}\big|_{+}=0,\;&i=1,2,\\
u=\varphi, &\mathrm{on}\;\partial D,
\end{cases}
\end{align}
where $\varphi\in C^{2}(\partial D)$ and the free constants $C_{1}$ and $C_{2}$ are determined by the third line of (\ref{con002}) and
$$\frac{\partial u}{\partial\nu}\Big|_{+}:=\lim_{\tau\rightarrow0}\frac{u(x+\nu\tau)-u(x)}{\tau}.$$
%Here and below $\nu$ denotes the unit outer normal to the domains and the subscript $\pm$ represents the limit from outside and inside the domain, respectively.

Similarly as before, define the blow-up factors $\mathcal{Q}[\varphi]$ and $\mathcal{Q}^{\ast}[\varphi]$ as follows:
\begin{align}\label{LBD001}
\mathcal{Q}[\varphi]=\int_{\partial D_{1}}\frac{\partial u_{b}}{\partial\nu}\Big|_{+},\quad \mathcal{Q}^{\ast}[\varphi]=\int_{\partial D_{1}^{\ast}}\frac{\partial u_{b}^{\ast}}{\partial\nu}\Big|_{+},
\end{align}
where $u_{b}$ and $u_{b}^{\ast}$ satisfy
\begin{align*}
\begin{cases}
\Delta u_{b}=0,&\mathrm{in}\;\Omega,\\
u_{b}=C_{2},&\mathrm{on}\;\partial D_{1}\cup\partial D_{2},\\
%\int_{(\partial D_{1}^{\ast}\cup\partial D_{2})\setminus\Sigma'}\frac{\partial v_{e}^{\ast}}{\partial\nu}\big|_{+}=0,\\
u_{b}=\varphi,&\mathrm{on}\;\partial D,
\end{cases}\quad
\begin{cases}
\Delta u_{b}^{\ast}=0,&\mathrm{in}\;\Omega^{\ast},\\
u_{b}^{\ast}=C^{\ast},&\mathrm{on}\;(\partial D_{1}^{\ast}\setminus\{0\})\cup\partial D_{2},\\
%\int_{(\partial D_{1}^{\ast}\cup\partial D_{2})\setminus\Sigma'}\frac{\partial v_{e}^{\ast}}{\partial\nu}\big|_{+}=0,\\
u_{b}^{\ast}=\varphi,&\mathrm{on}\;\partial D,
\end{cases}
\end{align*}
respectively. Here the explicit value of $C^{\ast}$ is given by
\begin{align}\label{KTN001}
C^{\ast}=&
\begin{cases}
\frac{\sum\limits^{2}_{i=1}b_{i}^{\ast}}{\sum\limits^{2}_{i,j=1}a_{ij}^{\ast}},&m\geq d-1,\\
\frac{a_{11}^{\ast}b_{2}^{\ast}-b_{1}^{\ast}a_{21}^{\ast}}{a_{11}^{\ast}a_{22}^{\ast}-a_{12}^{\ast}a_{21}^{\ast}},&m<d-1,
\end{cases}
\end{align}
where, for $i,j=1,2,$
\begin{align*}
a_{ij}^{\ast}:=\int_{\Omega^{\ast}}\nabla v_{i}^{\ast}\nabla v_j^{\ast},\quad b_i^{\ast}:=-\int_{\partial D}\frac{\partial v_{i}^{\ast}}{\partial \nu}\large\Big|_{+}\cdot v_{0}^{\ast},
\end{align*}
with $v_{i}^{\ast}$, $i=0,1,2$ verifying
\begin{equation*}
\begin{cases}
\Delta v_{0}^{\ast}=0,&\mathrm{in}~\Omega^{\ast},\\
v_{0}^{\ast}=0,&\mathrm{on}~\partial D_{1}^{\ast}\cup \partial D_{2},\\
v_{0}^{\ast}=\varphi,&\mathrm{on}~\partial{D},
\end{cases}
\end{equation*}
and
\begin{align*}
\begin{cases}
\Delta v_{1}^{\ast}=0,&\mathrm{in}~\Omega^{\ast},\\
v_{1}^{\ast}=1,&\mathrm{on}~\partial{D}_{1}^{\ast}\setminus\{0\},\\
v_{1}^{\ast}=0,&\mathrm{on}~\partial D_{2}\cup\partial{D},
\end{cases}\quad
\begin{cases}
\Delta v_{2}^{\ast}=0,&\mathrm{in}~\Omega^{\ast},\\
v_{2}^{\ast}=1,&\mathrm{on}~\partial D_{2},\\
v_{2}^{\ast}=0,&\mathrm{on}~(\partial D_{1}^{\ast}\setminus\{0\})\cup\partial{D},
\end{cases}
\end{align*}
respectively.

Denote
\begin{align}\label{LATZ}
\bar{r}_{\varepsilon}=&
\begin{cases}
\varepsilon^{\min\{\frac{1}{4},\frac{m-d+1}{m}\}},&m>d-1,\\
|\ln\varepsilon|^{-1},&m=d-1,\\
\varepsilon^{\min\{\frac{1}{6},\frac{d-1-m}{12m}\}},&m<d-1.
\end{cases}
\end{align}

The principal result in this section is stated as follows.
\begin{theorem}\label{thmmm}
Let $D_{1},D_{2}\subset D\subseteq\mathbb{R}^{d}\,(d\geq2)$ are defined as above, conditions $\mathrm{(}${\bf{H1}}$\mathrm{)}$--$\mathrm{(}${\bf{H3}}$\mathrm{)}$ hold, and $\varphi\in C^{2}(\partial D)$. Then for a sufficiently small $\varepsilon>0$,
\begin{align}\label{KGC}
\mathcal{Q}[\varphi]=\mathcal{Q}^{\ast}[\varphi]+(\bar{r}_{\varepsilon}),
\end{align}
where the blow-up factors $\mathcal{Q}[\varphi]$ and $\mathcal{Q}^{\ast}[\varphi]$ are defined in \eqref{LBD001}, $\bar{r}_{\varepsilon}$ is defined by \eqref{LATZ}.
\end{theorem}
\begin{remark}
We here would like to point out that Gorb and Novikov \cite{GN2012} were the first to capture a unified blow-up factor similar to $\mathcal{Q}[\varphi]$. Li \cite{Li2020} established the convergence between the blow-up factors $\mathcal{Q}[\varphi]$ and $\mathcal{Q}^{\ast}[\varphi]$ only in the case of $m\geq d-1$. By contrast, we completely solve the convergence in \eqref{KGC} for the generalized $m$-convex inclusions in all dimensions. The result in Theorem \ref{thmmm} is obtained by constructing the explicit value of $C^{\ast}$ in \eqref{KTN001}, which is different from that in \cite{GN2012,Li2020}.
\end{remark}

As seen in \cite{BLY2009,LLY2019,Li2020}, we decompose the solution $u$ of \eqref{con002} as follows:
\begin{align}\label{con00KL}
u=(C_{1}-C_{2})v_{1}+C_{2}(v_{1}+v_{2})+v_{0},\quad\mathrm{in}\;\Omega,
\end{align}
where $v_{i}$, $i=0,1,2$, are the solutions of
\begin{align*}
\begin{cases}
\Delta v_{0}=0,&\mathrm{in}\;\Omega,\\
v_{0}=0,&\mathrm{on}\;\partial D_{1}\cup\partial D_{2},\\
v_{0}=\varphi,&\mathrm{on}\;\partial D,
\end{cases}
\quad
\begin{cases}
\Delta v_{i}=0,&\mathrm{in}\;\Omega,\\
v_{i}=\delta_{ij},&\mathrm{on}\;\partial D_{j},\;i,j=1,2,\\
v_{i}=0,&\mathrm{on}\;\partial D,
\end{cases}
\end{align*}
respectively. For $i,j=1,2,$ denote
\begin{align*}
a_{ij}:=-\int_{\partial D_{j}}\frac{\partial v_{i}}{\partial\nu}\Big|_{+},\quad b_{j}:=\int_{\partial D_{j}}\frac{\partial v_{0}}{\partial\nu}\Big|_{+}.
\end{align*}

Observe that by applying the proofs of Theorems 2.1 and 2.2 in \cite{ZH202101} with a slight modification, we obtain that for $i=1,2,$
\begin{align}
\nabla v_{i}=&(-1)^{i-1}\nabla\bar{v}+O(\delta^{\frac{m-2}{m}}),\;\,\mathrm{in}\;\Omega_{R},\quad\|\nabla v_{i}\|_{L^{\infty}(\Omega\setminus\Omega_{R})}\leq C,\label{OKA005}\\
\nabla v_{i}^{\ast}=&(-1)^{i-1}\nabla\bar{v}^{\ast}+O(|x'|^{m-2}),\;\,\mathrm{in}\;\Omega_{R}^{\ast},\quad\|\nabla v_{i}^{\ast}\|_{L^{\infty}(\Omega^{\ast}\setminus\Omega_{R}^{\ast})}\leq C,\label{OKA005000}
\end{align}
and
\begin{align}
|\nabla v_{0}|+|\nabla(v_{1}+v_{2})|\leq& C\delta^{-\frac{d}{2}}e^{-\frac{1}{2C\delta^{1-1/m}}},\quad\mathrm{in}\;\Omega_{R},\label{LT}\\
|\nabla v_{0}^{\ast}|+|\nabla(v_{1}^{\ast}+v_{2}^{\ast})|\leq& C|x'|^{-\frac{md}{2}}e^{-\frac{1}{2C|x'|^{m-1}}},\quad\mathrm{in}\;\Omega_{R}^{\ast},\label{LTFA001}
\end{align}
where $\bar{v}$ is defined by \eqref{zh001} and $\delta$ is defined in \eqref{deta}.

Consequently, in light of \eqref{OKA005}--\eqref{LTFA001}, an immediate consequence of Lemmas \ref{lemmabc} and \ref{KM323} yields the following results.
\begin{lemma}\label{LGVC}
Assume as above. Then for a sufficiently small $\varepsilon>0$,
\begin{align*}
b_{i}=b_{i}^{\ast}+O(\varepsilon^{\frac{1}{2}}),\quad i=1,2,
\end{align*}
and
\begin{align*}
\begin{cases}
\tau_{2}^{-\frac{d-1}{m}}\lesssim\frac{a_{11}}{\rho_{0}(d,m;\varepsilon)}\lesssim\tau_{1}^{-\frac{d-1}{m}},&m\geq d-1,\\
a_{11}=a_{11}^{\ast}+O(\varepsilon^{\min\{\frac{1}{6},\frac{d-1-m}{12m}\}}),&m<d-1,
\end{cases}
\end{align*}
and
\begin{align*}
\sum^{2}_{i=1}a_{i1}=\sum^{2}_{i=1}a_{i1}^{\ast}+O(\varepsilon^{\frac{1}{4}}),\quad\sum^{2}_{i,j=1}a_{ij}=\sum^{2}_{i,j=1}a_{ij}^{\ast}+O(\varepsilon^{\frac{1}{4}}).
\end{align*}

\end{lemma}

\begin{proof}[The proof of Theorem \ref{thmmm}]
Applying the third line of (\ref{con002}) to decomposition \eqref{con00KL}, we have
\begin{align}\label{LHAT001}
\begin{cases}
a_{11}(C_{1}-C_{2})+\sum\limits^{2}_{i=1}a_{i1}C_{2}=b_{1},\\
a_{12}(C_{1}-C_{2})+\sum\limits^{2}_{i=1}a_{i2}C_{2}=b_{2}.
\end{cases}
\end{align}
Then adding the first line of \eqref{LHAT001} to the second line, we get
\begin{align*}
\begin{cases}
a_{11}(C_{1}-C_{2})+\sum\limits^{2}_{i=1}a_{i1}C_{2}=b_{1},\\
\sum\limits^{2}_{j=1}a_{1j}(C_{1}-C_{2})+\sum\limits^{2}_{i,j=1}a_{ij}C_{2}=\sum\limits^{2}_{i=1}b_{i},
\end{cases}
\end{align*}
which, together with Lemma \ref{LGVC} and Cramer's rule, yields that

$(\rm{i})$ if $m\geq d-1$, then
\begin{align*}
C_{2}=&\frac{a_{11}\sum\limits^{2}_{i=1}b_{i}-b_{1}\sum\limits^{2}_{j=1}a_{1j}}{a_{11}\sum\limits^{2}_{i,j=1}a_{ij}-\sum\limits^{2}_{i,j=1}a_{i1}a_{1j}}=\frac{\sum\limits^{2}_{i=1}b_{i}}{\sum\limits^{2}_{i,j=1}a_{ij}}(1+O(\rho^{-1}_{0}(d,m;\varepsilon)))\notag\\
=&\frac{\sum\limits^{2}_{i=1}b_{i}^{\ast}}{\sum\limits^{2}_{i,j=1}a_{ij}^{\ast}}
\begin{cases}
1+O(\varepsilon^{\min\{\frac{1}{4},\frac{m-d+1}{m}\}}),&m>d-1,\\
1+O(|\ln\varepsilon|^{-1}),&m=d-1;
\end{cases}
\end{align*}

$(\rm{ii})$ if $m<d-1$, then
\begin{align*}
C_{2}=&\frac{a_{11}\sum\limits^{2}_{i=1}b_{i}-b_{1}\sum\limits^{2}_{j=1}a_{1j}}{a_{11}\sum\limits^{2}_{i,j=1}a_{ij}-\sum\limits^{2}_{i,j=1}a_{i1}a_{1j}}=\frac{a_{11}^{\ast}\sum\limits^{2}_{i=1}b_{i}^{\ast}-b_{1}^{\ast}\sum\limits^{2}_{j=1}a_{1j}^{\ast}}{a_{11}^{\ast}\sum\limits^{2}_{i,j=1}a_{ij}^{\ast}-\sum\limits^{2}_{i,j=1}a_{i1}^{\ast}a_{1j}^{\ast}}(1+O(\varepsilon^{\min\{\frac{1}{6},\frac{d-1-m}{12m}\}}))\notag\\
=&\frac{a_{11}^{\ast}b_{2}^{\ast}-b_{1}^{\ast}a_{12}^{\ast}}{a_{11}^{\ast}a_{22}^{\ast}-a_{12}^{\ast}a_{21}^{\ast}}(1+O(\varepsilon^{\min\{\frac{1}{6},\frac{d-1-m}{12m}\}})).
\end{align*}
We now demonstrate that $\sum\limits^{2}_{i,j=1}a_{ij}^{\ast}\neq0$ and $a_{11}^{\ast}a_{22}^{\ast}-a_{12}^{\ast}a_{21}^{\ast}\neq0$. To begin with, it follows from the Hopf Lemma that
\begin{align*}
\frac{\partial v_{1}^{\ast}}{\partial\nu}\Big|_{\partial D_{1}^{\ast}\setminus\{0\}}<0,\quad\frac{\partial v_{1}^{\ast}}{\partial\nu}\Big|_{\partial D}<0,\quad\frac{\partial v_{2}^{\ast}}{\partial\nu}\Big|_{\partial D_{1}^{\ast}\setminus\{0\}}>0,\quad\frac{\partial v_{2}^{\ast}}{\partial\nu}\Big|_{\partial D}<0.
\end{align*}
This implies that
\begin{align*}
\sum\limits^{2}_{i,j=1}a_{ij}^{\ast}=&-\left(\int_{\partial D_{1}^{\ast}\cup\partial D_{2}}\frac{\partial v_{1}^{\ast}}{\partial\nu}\Big|_{+}+\int_{\partial D_{1}^{\ast}\cup\partial D_{2}}\frac{\partial v_{2}^{\ast}}{\partial\nu}\Big|_{+}\right)\notag\\
=&-\left(\int_{\partial D}\frac{\partial v_{1}^{\ast}}{\partial\nu}\Big|_{+}+\int_{\partial D}\frac{\partial v_{2}^{\ast}}{\partial\nu}\Big|_{+}\right)>0,
\end{align*}
and
\begin{align*}
a_{11}^{\ast}a_{22}^{\ast}-a_{12}^{\ast}a_{21}^{\ast}=&\int_{\partial D_{1}^{\ast}}\frac{\partial v_{1}^{\ast}}{\partial\nu}\Big|_{+}\int_{\partial D_{2}}\frac{\partial v_{2}^{\ast}}{\partial\nu}\Big|_{+}-\int_{\partial D_{1}^{\ast}}\frac{\partial v_{2}^{\ast}}{\partial\nu}\Big|_{+}\int_{\partial D_{2}}\frac{\partial v_{1}^{\ast}}{\partial\nu}\Big|_{+}\notag\\
=&\int_{\partial D_{1}^{\ast}}\frac{\partial v_{1}^{\ast}}{\partial\nu}\Big|_{+}\int_{\partial D}\frac{\partial v_{2}^{\ast}}{\partial\nu}\Big|_{+}-\int_{\partial D_{1}^{\ast}}\frac{\partial v_{2}^{\ast}}{\partial\nu}\Big|_{+}\int_{\partial D}\frac{\partial v_{1}^{\ast}}{\partial\nu}\Big|_{+}>0,
\end{align*}
where we utilized the fact that
\begin{align*}
\int_{\partial D_{1}^{\ast}\cup\partial D_{2}}\frac{\partial v_{i}^{\ast}}{\partial\nu}\Big|_{+}=\int_{\partial D}\frac{\partial v_{i}^{\ast}}{\partial\nu}\Big|_{+},\quad i=1,2.
\end{align*}

Therefore, we obtain
\begin{align}\label{OFAT}
C_{2}=C^{\ast}+O(\bar{r}_{\varepsilon}),
\end{align}
where $C^{\ast}$ is defined in \eqref{KTN001}, $\bar{r}_{\varepsilon}$ is defined by \eqref{LATZ}. By linearity, we have
\begin{align}\label{KTF}
u_{b}=C_{2}(v_{1}+v_{2})+v_{0},\quad u_{b}^{\ast}=C^{\ast}(v_{1}^{\ast}+v_{2}^{\ast})+v_{0}^{\ast}.
\end{align}
Then in view of decomposition \eqref{KTF} and using integration by parts, we rewrite $\mathcal{Q}[\varphi]$ and $\mathcal{Q}^{\ast}[\varphi]$ as follows:
\begin{align*}
\mathcal{Q}[\varphi]=-C_{2}\sum^{2}_{i=1}a_{i1}+b_{1},\quad\mathcal{Q}^{\ast}[\varphi]=-C^{\ast}\sum^{2}_{i=1}a_{i1}^{\ast}+b_{1}^{\ast}.
\end{align*}
This, in combination with \eqref{OFAT} and Lemma \ref{LGVC}, leads to that
\begin{align*}
\mathcal{Q}[\varphi]-\mathcal{Q}^{\ast}[\varphi]=&(C^{\ast}-C_{2})\sum^{2}_{i=1}a_{i1}+C^{\ast}\left(\sum^{2}_{i=1}a_{i1}^{\ast}-\sum^{2}_{i=1}a_{i1}\right)+b_{1}-b_{1}^{\ast}=O(\bar{r}_{\varepsilon}).
\end{align*}
The proof is complete.
\end{proof}
From \eqref{con00KL} and \eqref{KTF}, we get
\begin{align}\label{MDA}
\nabla u=(C_{1}-C_{2})\nabla v_{1}+\nabla u_{b},\quad\mathrm{in}\;\Omega.
\end{align}
Then substituting \eqref{MDA} into the third line of \eqref{con002}, we derive
\begin{align*}
a_{11}(C_{1}-C_{2})=\mathcal{Q}[\varphi],
\end{align*}
which yields that
\begin{align}\label{OLGN001}
\nabla u=\frac{\mathcal{Q}[\varphi]}{a_{11}}\nabla v_{1}+\nabla u_{b},\quad\mathrm{in}\;\Omega.
\end{align}

Therefore, in view of decomposition \eqref{OLGN001} and combining Theorem \ref{thmmm}, Lemma \ref{LGVC}, \eqref{OKA005}--\eqref{LTFA001}, we immediately obtain the following two corollaries. The first corollary is listed as follows.
\begin{corollary}\label{CCCOOO0}
Assume that $D_{1},D_{2}\subset D\subseteq\mathbb{R}^{d}\,(d\geq2)$ are defined as above, conditions $\mathrm{(}${\bf{H1}}$\mathrm{)}$--$\mathrm{(}${\bf{H3}}$\mathrm{)}$ hold, and $\varphi\in C^{2}(\partial D)$. Let $u\in H^{1}(D)\cap C^{1}(\overline{\Omega})$ be the solution of \eqref{con002}. If $\mathcal{Q}^{\ast}[\varphi]\neq0$, then for a sufficiently small $\varepsilon>0$ and $x\in\{|x'|=0\}\cap\Omega$,

$(\rm{i})$ for $m\geq d-1$,
\begin{align*}
\frac{\tau_{1}^{\frac{d-1}{m}}\mathcal{Q}^{\ast}[\varphi]}{\varepsilon\rho_{0}(d,m;\varepsilon)}\lesssim|\nabla u|\lesssim&
\frac{\tau_{2}^{\frac{d-1}{m}}\mathcal{Q}^{\ast}[\varphi]}{\varepsilon\rho_{0}(d,m;\varepsilon)};
\end{align*}

$(\rm{ii})$ for $m<d-1$,
\begin{align*}
|\nabla u|\simeq\frac{\mathcal{Q}^{\ast}[\varphi]}{a_{11}^{\ast}}\frac{1}{\varepsilon},
\end{align*}
where $\rho_{0}(d,m;\varepsilon)$ is defined by \eqref{rate00}.
\end{corollary}
\begin{remark}
Although the optimality of the blow-up rate for any $m,d\geq2$ has been solved in the previous work \cite{ZH202101}, we here use a unified blow-up factor to show the optimality of the blow-up rate again.
\end{remark}

%Similarly as in \eqref{KATZ001}, we have
%\begin{align}\label{zadz0109865}
%a_{11}=&\int_{\varepsilon^{\frac{1}{12m}}<|x'|<R}\frac{dx'}{h_{1}(x')-h_{2}(x')}+\int_{|x'|<\varepsilon^{\gamma}}\frac{dx'}{\varepsilon+h_{1}(x')-h_{2}(x')}\notag\\
%&+A_{R}^{\ast}+O(1)\varepsilon^{\frac{1}{6m}},
%\end{align}
%where
Denote
\begin{align}\label{LDABN}
\mathcal{G}^{\ast}_{m}:=&
\begin{cases}
\frac{1}{\mathcal{M}_{1}}\left(M_{R}^{\ast}+\frac{4\int^{\frac{\pi}{2}}_{0}\ln R(\theta)d\theta}{\sqrt{\tau_{1}\tau_{2}}}\right),&m=2,\\
\frac{m\pi}{4\mathcal{M}_{1}\int^{\frac{\pi}{2}}_{0}E(\theta)d\theta}\left(M_{R}^{\ast}-\frac{8\int^{\frac{\pi}{2}}_{0}E(\theta)(R(\theta))^{2-m}d\theta}{m(m-2)\sqrt[m]{\tau_{1}\tau_{2}}}\right),&m>2,
\end{cases}
\end{align}
where $E(\theta)$ and $R(\theta)$ are defined by \eqref{LRANM001}--\eqref{KN},
\begin{align*}
M_{R}^{\ast}=&\int_{\Omega^{\ast}\setminus\Omega_{R}^{\ast}}|\nabla v_{1}^{\ast}|^{2}+2\int_{\Omega_{R}^{\ast}}\nabla\bar{v}^{\ast}\cdot\nabla(v_{1}^{\ast}-\bar{v}^{\ast})+\int_{\Omega^{\ast}_{R}}\big(|\nabla(v_{1}^{\ast}-\bar{v}^{\ast})|^{2}+|\nabla_{x'}\bar{v}^{\ast}|^{2}\big).
\end{align*}
From \eqref{OKA005000}, we know that $M_{R}^{\ast}$ is a bounded constant. Then we state the second corollary as follows.

\begin{corollary}\label{CCCOOO6}
Assume that $D_{1},D_{2}\subset D\subseteq\mathbb{R}^{3}$ are defined as above, conditions \eqref{ZCZ009} and $\mathrm{(}${\bf{H2}}$\mathrm{)}$--$\mathrm{(}${\bf{H3}}$\mathrm{)}$ hold, and $\varphi\in C^{2}(\partial D)$. Let $u\in H^{1}(D)\cap C^{1}(\overline{\Omega})$ be the solution of \eqref{con002}. If $\mathcal{Q}^{\ast}[\varphi]\neq0$, then for a sufficiently small $\varepsilon>0$ and $x\in\Omega_{R}$,

$(\rm{i})$ if $m=2$, then
\begin{align*}
\nabla u=\frac{\mathcal{Q}^{\ast}[\varphi]}{\mathcal{M}_{1}}\frac{1+O(|\ln\varepsilon|^{-1})}{|\ln\varepsilon|+\mathcal{G}^{\ast}_{2}}\nabla\bar{v}+O(1);
\end{align*}

$(\rm{ii})$ if $m>2$, then
\begin{align*}
\nabla
u=\frac{m\pi\varepsilon^{\frac{m-2}{m}}\mathcal{Q}^{\ast}[\varphi]}{4\mathcal{M}_{1}\int^{\frac{\pi}{2}}_{0}E(\theta)d\theta}\frac{1+O(\varepsilon^{\min\{\frac{1}{4},\frac{m-2}{m}\}})}{1+\mathcal{G}^{\ast}_{m}\varepsilon^{\frac{m-2}{m}}}\nabla\bar{v}+O(1)\delta^{\frac{m-2}{m}},
\end{align*}
where $\delta$ is defined in \eqref{deta}, $\mathcal{M}_{1}$ is defined by \eqref{WEN}, $E(\theta)$ is defined by \eqref{LRANM001}, $\mathcal{G}_{m}^{\ast}$ is define in \eqref{LDABN}.

\end{corollary}

%\noindent{\bf{\large Statements.}} The authors state that there is no conflict of interest. The manuscript has no associated data.

\noindent{\bf{\large Acknowledgements.}} C. Miao was supported by the National Key Research and Development Program of China (No. 2020YFA0712900) and NSFC Grant 11831004. Z. Zhao was partially supported by CPSF (2021M700358).

\end{document}